\documentclass[11pt]{article}
\usepackage[utf8]{inputenc}
\usepackage{times}
\usepackage{bbm}
\usepackage{amsmath}
\usepackage{comment}
\usepackage{amsthm}
\usepackage{amsfonts}
\usepackage{amssymb}
\usepackage{bm}
\usepackage{bbm}
\usepackage{graphicx}
\usepackage[colorinlistoftodos]{todonotes}
\usepackage[english]{babel}
\usepackage{siunitx}
\usepackage{textcomp}
\usepackage{listings}
\usepackage[margin=1in]{geometry}
\usepackage[ruled,vlined]{algorithm2e}
\usepackage{indentfirst}
\usepackage{setspace}
\usepackage{enumerate}
\usepackage{thm-restate}
\usepackage{graphicx}
\usepackage{subcaption}
\usepackage{tikz}
\usetikzlibrary{calc}
\usepackage[percent]{overpic}

\usepackage{amsthm}

\theoremstyle{plain}
\newtheorem{theorem}{Theorem}[section]
\newtheorem{corollary}[theorem]{Corollary}
\newtheorem{lemma}[theorem]{Lemma}

\newtheorem*{theorem*}{Theorem}

\theoremstyle{definition}
\newtheorem{definition}[theorem]{Definition}

\newtheorem{fact}[theorem]{Fact}

\theoremstyle{remark}
\newtheorem{remark}[theorem]{Remark}
\newtheorem*{remark*}{Remark}

\usepackage{xcolor}

\definecolor{linkcol}{RGB}{0,70,160}
\definecolor{citecol}{RGB}{0,120,70}
\definecolor{urlcol}{RGB}{150,60,0}

\usepackage[
    colorlinks=true,
    linkcolor=linkcol,
    citecolor=citecol,
    urlcolor=urlcol,
    breaklinks=true,
    bookmarksopen=true,
    bookmarksnumbered=true
]{hyperref}

\newcommand{\R}{\mathbb{R}}

\newcommand{\E}{\mathbb{E}}
\newcommand{\Z}{\mathbb{Z}}
\newcommand{\Pp}{\mathbb{P}}

\newcommand{\pr}{\mathbb{P}}

\newcommand{\eps}{\varepsilon}

\newcommand{\mix}{\text{mix}}

\newcommand{\gap}{\mathrm{gap}}
\newcommand{\ber}{\mathrm{Ber}}

\newcommand{\Th}{\tau_\mix}
\newcommand{\Thh}{\tau_\mix(h)}

\newcommand{\dary}{\mathcal T_h}
\newcommand{\dreg}{\mathcal{\hat T}_h}
\newcommand{\daryv}{\mathcal T_h^v}
\newcommand{\dregv}{\mathcal{\hat T}_h^v}
\newcommand{\pu}{p_{\mathsf{u}}}
\newcommand{\pc}{p_{\mathsf{c}}}
\newcommand{\ps}{p_{\mathsf{s}}}
\newcommand{\dist}{\mathrm{dist}}

\title{Uniqueness and Mixing in the Low-Temperature Random-Cluster Model on Trees and Random Graphs}

 \author{Antonio Blanca \thanks{Department of CSE, Pennsylvania State University, ablanca@cse.psu.edu.  Research supported in part by NSF CAREER grant CCF-2143762.} \and Reza Gheissari \thanks{Department of Mathematics, Northwestern University, gheissari@northwestern.edu. Research supported in part by NSF CAREER grant DMS-2440509 and NSF DMS grant 2246780.} \and Heehyun Park \thanks{Department of CSE, Pennsylvania State University, hbp5148@psu.edu. Research supported in part by NSF CAREER grant CCF-2143762.} \and Xusheng Zhang \thanks{Department of Computer Science, Durham University, xusheng.zhang@durham.ac.uk. Research supported in part by EPSRC New Investigator Award UKRI155.}
 }

\date{}

\begin{document}

\maketitle

\thispagestyle{empty}

\begin{abstract}
    We study the random-cluster model on trees and treelike graphs at low temperatures. This is a model of dependent percolation parametrized by an edge probability $p\in (0,1)$ and a clustering weight $q\in [1,\infty)$, generalizing independent Bernoulli percolation ($q=1$) and closely related to the classical ferromagnetic Ising and Potts spin systems at integer $q$.
    For $q>2$, approximately sampling from this model on graphs of degree at most $\Delta$ is computationally hard. 
    At parameter $p$ below 
    a threshold $\pu(q,\Delta)$,
    that marks a sharp transition in typical cluster sizes and uniqueness of the distribution on $\Delta$-regular trees, 
    it is expected that sampling is easy and local Markov chains mix rapidly on all bounded degree graphs. 
    On typical graphs (e.g., random regular graphs), the same is predicted at $p > \ps(q,\Delta)$, where $\ps(q,\Delta)$ is a conjectured second transition point on the $\Delta$-regular tree with a wired boundary.

Our first result establishes this phase transition at $\ps(q,\Delta)$ for all $q$ on the infinite $\Delta$-regular wired tree, resolving a conjecture of H{\"a}ggstr{\"o}m (1996). In particular, we establish  weak spatial mixing at $p>\ps(q,\Delta)$ under sufficiently wired boundary conditions. 
    We use this understanding of correlation decay on the tree to obtain novel algorithmic implications. As our first application, we show that on the wired tree on $n$ vertices, whenever $q>1$ and $p>\ps(q,\Delta)$, the mixing time of random-cluster Glauber dynamics is a near-optimal $n^{1+o(1)}$. 
    We then extend these spatial and temporal mixing results from trees to treelike geometries with mostly wired boundaries, and use them to show that the random-cluster Glauber dynamics mixes rapidly on the random $\Delta$-regular graph for all $p>\ps(q,\Delta)$ as long as $q \ge C \log \Delta$, providing an efficient sampling algorithm for both the random-cluster and Potts models in this context.    
    The best prior sampling results for these models that held for all $p>\ps(q,\Delta)$ required $q\ge \Delta^{5}$ for integer $q$ and $q = \Delta^{\Omega(\Delta)}$ for non-integer~$q$.
\end{abstract}

\vfill

 \pagebreak

\setcounter{page}{1} 
\section{Introduction}
We study the random-cluster model and its local Glauber dynamics Markov chain on trees and treelike graphs in the low-temperature regime. The random-cluster model is an extensively studied model of random subgraphs both of interest in its own right and for its deep connections to the ferromagnetic Ising and Potts models; see the book~\cite{grimmett2006random} for more.
Formally, on a finite graph $G = (V,E)$, 
the random-cluster model is the probability distribution over $A \subseteq E$ given by 
\begin{align}\label{eq:random-cluster-measure}
    \mu_{G,p,q} (A) \propto p^{|A|} (1-p)^{|E\setminus A|} q^{c_G(A)}
\end{align}
where $c_G(A)$ counts the number of connected components of the subgraph $(V,A)$, and $p \in (0,1)$ and $q > 0$ are model parameters.

At $q=1$, the distribution~\eqref{eq:random-cluster-measure} is simply the independent edge percolation, while at integer $q\ge 2$ it admits a coupling to the  classical ferromagnetic Ising and Potts models.
This coupling 
allows samples generated for one model to be converted into samples from the other efficiently and is also at the heart of the famous Swendsen–Wang Markov chain, one of the most widely used algorithms for sampling from these distributions.  We study the random-cluster model when $q>1$ is any real number. The $q<1$ regime is also of interest, but qualitatively very different as the model does not exhibit positive association; see~\cite{log-concave-ii} for relevant work on sampling from the $q<1$ random-cluster model at all $p$ on all graphs.

The random-cluster model generically undergoes a phase transition as the parameter $p$ is varied. At small values of $p$ (high temperature), correlations between edges decay exponentially fast with their distance and no large connected components are present in a typical configuration. In contrast, at large values of $p$ (low temperature), a unique ``giant'' component of linear size emerges with high probability with all other component sizes having exponential tails on their diameters. 
For $q>2$, however, there is a \emph{coexistence} regime of these two different phases at intermediate values of $p$,
where the model assigns non-negligible probability both to disordered configurations with no large components and to ordered configurations with a unique giant component.

The specifics of this phase transition depend on the underlying graph.
For instance, the coexistence region is only at a single critical point $p = p_{\mathsf{c}}$ when the graph is amenable (e.g., the integer lattice graph $\Z^d$~\cite{duminil2019sharp}). 
On non-amenable graphs (e.g., infinite trees, random graphs and more generally graphs with exponential volume growth), the coexistence regime can take up an entire interval of parameter values, which we denote by $(\pu,\ps)$. This poses special algorithmic and mathematical challenges as 
the transition points $\pu$ and $\ps$ do not align with the point at which a typical sample from the distribution transitions from high vs.\ low temperature behavior. 
Notably, the coexistence regime $(\pu,\ps)$
is characterized by the slow convergence of natural Markov chains and underlies the computational hardness of approximate counting and sampling for the random-cluster and Potts models for $q>2$~\cite{goldberg2012approximating,GSVY}.

Focusing on Markov chain sampling, we recall the definition of the heat-bath Glauber dynamics for the random-cluster model, which we henceforth refer to as \emph{random-cluster dynamics}. In discrete-time, this is the Markov chain which from configuration $X_t\subseteq E$, generates $X_{t+1}$ by 
\begin{enumerate}\setlength{\itemsep}{2pt}
    \item Picking an edge $e \in E$ uniformly at random;
    \item Setting $X_{t+1}$ to be 
    \begin{align}\label{eq:Glauber-update-rule}
        X_{t+1}  = 	X_t \cup \{e\} \text{ with probability } \Big\{\begin{array}{ll}
	\hat p := \frac{p}{q(1-p)+p} & \mbox{if $e$ is a ``cut-edge'' in $(V,X_t)$;} \\
	p & \mbox{otherwise;}
    \end{array}
    \end{align}
    and $X_{t+1} = X_t \setminus \{e\}$ otherwise.
\end{enumerate}
We say the edge $e$ is a cut-edge in $(V,X_t)$ if changing the state of $e$ changes the number of connected components in $(V,X_t)$. 

The mixing time of a Markov chain is the number of steps it takes 
to be close in total-variation to the stationary distribution
$\mu_{G,p,q}$ from a worst-case initialization; namely, 
\begin{align}\label{eq:tmix}
    T_{\mix} = \min \Big\{t\ge 0: \max_{X_0} \|\mathbb P(X_t \in \cdot \mid X_0) - \mu_{G,p,q}\|_{\textsc{tv}} <1/4 \Big\}\,.
\end{align}
When $q\in (1,2]$, the mixing time of the random-cluster dynamics is expected to always be polynomial on any input graph for any $p \in (0,1)$. 
This was shown at $q=2$ in~\cite{GuoJerrum} and remains an important problem for $q \in (1,2)$. 

When $q>2$, the expectation is that the random-cluster dynamics mixes quickly at high temperatures $p<\pu$,
slows down exponentially in the intermediate coexistence regime $p\in (\pu,\ps)$, but interestingly becomes rapidly mixing again as soon as $p>\ps$. 
This trichotomy of behaviors has been fully characterized on the complete graph~\cite{BGJ,BS-MF,GSV,GLP}. 

Beyond the complete graph, an important research direction is to establish this trichotomy for other non-amenable graph families, especially trees and random graphs. These two settings will be the focus of our investigation, where
the random-cluster model remains not well-understood at low temperatures, neither in terms of its correlation decay properties nor its algorithmic tractability.
We study decay of correlations, specifically, a form of weak spatial mixing defined formally below, its implications for the uniqueness of the random-cluster distribution on the infinite regular tree, 
and its consequences for 
Markov chain mixing on trees and random $\Delta$-regular graphs.

\subsection{Our results}

Our first results are on the uniqueness and correlation decay properties on $\Delta$-regular trees of the random-cluster distribution when $p$ is bigger than the conjectured~$\ps$, two questions that are intimately related.
On a finite $\Delta$-regular tree, the random-cluster measure is simply a product distribution and therefore not especially interesting,
 so it is typically considered with
 \emph{boundary conditions} that force leaves to be counted as part of the same connected component, capturing the effect of external connections when the tree is part of a larger ambient graph.
 
H{\"a}ggstr{\"o}m~\cite{haggstrom1996random} considered
the extremal case where all leaves are wired together by the boundary condition and studied the problem of whether there 
is a unique random-cluster measure on the infinite $\Delta$-regular wired tree.
An \emph{infinite-volume Gibbs measure} on the $\Delta$-regular wired tree $\mathcal T_\infty = (V(\mathcal T_\infty),E(\mathcal T_\infty))$ is a probability distribution over subsets $A \subset E(\mathcal T_\infty)$ such that for any finite set $\Lambda \subset E(\mathcal T_\infty)$, resampling the configuration on $\Lambda$ given the exterior configuration $A(\Lambda^c)$ is consistent with the distribution~\eqref{eq:random-cluster-measure} on $\Lambda$ with  boundary wirings induced by $A(\Lambda^c)$, where all infinite components of $A(\Lambda^c)$ are considered wired together. We are more precise about this in Section~\ref{sec: uniqueness}. As a consequence of understanding the decay of correlations on finite wired trees,
\cite{haggstrom1996random} showed that there is a unique infinite-volume Gibbs measure when $p<\pu$ but non-uniqueness for $p\in (\pu,\ps)$ where 
\begin{align}\label{eq:p-Haggstrom}
    \ps = \ps(q,\Delta) := \frac{q}{\Delta + q - 2}\,.
\end{align} 
(The threshold $\pu$ has also been identified, but only implicitly as it does not admit a closed form.)

The long-standing Conjecture 1.9 of~\cite{haggstrom1996random} (see also Conjecture 10.97 in the book~\cite{grimmett2006random}) predicted that above $\ps$, the random-cluster model has a unique infinite-volume Gibbs measure on the wired tree. The works~\cite{Jonasson,GrimmettJanson} showed that this is the case for sufficiently large $p\gg \ps$, namely,  
$p \gtrsim \frac{\log \Delta}{\Delta}$.
Our first main theorem resolves this conjecture and completes the phase diagram on the wired tree for $q > 1$.

\begin{restatable}{theorem}{uniqueness}\label{thm:uniqueness}
    Fix $\Delta\ge 3$ and $q>1$ real. For all $p>\ps$, there is a unique random-cluster measure on the infinite  $\Delta$-regular wired tree. 
\end{restatable}

We state Theorem~\ref{thm:uniqueness} for the infinite $\Delta$-regular tree, which has the advantage of being a transitive graph, but our proof first establishes the corresponding result for the $(\Delta-1)$-ary tree.
We show that the measure on such trees has exponential decay of correlations and \emph{weak spatial mixing} for typical boundary conditions that are ``sufficiently wired", 
from which Theorem~\ref{thm:uniqueness} follows.
More generally, we prove that if $\xi$ is a random boundary condition on the complete $(\Delta-1)$-ary tree of height $h$, denoted $\dary$, that is \emph{$\theta$-wired}, meaning each leaf is part of a single wired boundary component with probability at least $\theta>0$ independently, then with high probability over $\xi$, for every fixed $r$ and all sufficiently large $h$,  
\begin{align}\label{eq:intro-wsm}
    \|\mu^\xi_{\dary} (A(\mathcal T_r)  \in \cdot ) - \mu_{\dary}^1(A(\mathcal T_r)\in \cdot) \|_{\textsc{tv}} \le e^{ - c_* h(1-o(1))}\,,
\end{align}
where $c_*$ is explicitly given in terms of the derivative of a tree recursion function at its fixed point. 
Here, $\mu_{\dary}^\xi$ is the random-cluster measure on $\dary$ with boundary condition $\xi$,  $\xi = 1$ is used for the all-wired one, and $A(\mathcal T_r)$ denotes the random-cluster configuration on the subtree $\mathcal T_r$.

This notion of weak spatial mixing on finite trees with sufficiently wired boundary conditions 
also has significant algorithmic implications in the form of mixing time guarantees for the random-cluster dynamics on treelike graphs in the entire low-temperature uniqueness region. 
Our first result in this direction gives a near-optimal mixing time for the random-cluster dynamics on trees with the all-wired boundary condition throughout the entire low-temperature uniqueness region. 

\begin{theorem}\label{thm:intro-wired-tree-mixing-time}
        Fix any $\Delta\ge 3$, $q>1$ real and $p>\ps$. The random-cluster dynamics on 
        the complete $(\Delta-1)$-ary tree on $n$ vertices
        with the all-wired boundary condition has mixing time $n^{1+o(1)}$.  
\end{theorem}
We in fact establish a stronger version of the above theorem for the more general class of $\theta$-wired boundary conditions as in~\eqref{eq:intro-wsm}, and the proof is robust to $O(1)$ modifications to the graph structure. In particular, it applies also to the complete $\Delta$-regular tree. 
Previously,~\cite{BCSV-Trees} established that the mixing time 
of the random-cluster dynamics is $O(n \log n)$ for $p \in (0,1)$ but only when $q$ is an integer, using the coupled Potts model to factorize entropy. We note that the study of Markov chains on trees in the presence of boundary conditions is surprisingly subtle, even for standard spin systems such as the Ising and Potts models; see, e.g.,~\cite{MSW,Kenyon,BCSV-Trees}. 
For the random-cluster model, it is even more delicate, since boundary conditions can effectively destroy the tree structure by wiring distant boundary vertices together.
Indeed, worst-case boundary conditions with multiple connected components on trees can induce exponentially slow mixing of the random-cluster dynamics~\cite{BCSV-Trees}.

Using our notion of weak spatial mixing for sufficiently wired trees 
when $p>p_{\mathsf{s}}$ together with the mixing time guarantees of Theorem~\ref{thm:intro-wired-tree-mixing-time}, we are able to obtain fast mixing 
for the random-cluster dynamics
in the low-temperature region $p>\ps$ for random $\Delta$-regular graphs. 
Low-temperature mixing and sampling for the random-cluster and Potts models on random graphs have seen much attention in recent years. In particular,~\cite{HeJePe23,Galanis_Goldberg_Smolarova_2025} gave algorithms based on cluster expansion machinery or random-cluster dynamics from good initializations to get fast sampling for the entire low-temperature regime under assumptions on $q$ being sufficiently large (asymptotically $\Delta^{\Omega(\Delta)}$ for $\Delta$ large). For integer $q$, this dependency was improved to $\Delta\ge \Delta_0$ and $q\ge \Delta^C$ in~\cite{carlson2022algorithms} and to $q\ge C\Delta^5$ in~\cite{galanis_et_al:LIPIcs.STACS.2026.39}\footnote{ 
The latter bound was established for the larger low-temperature regime $p > \pc$, where $\pc\in (\pu,\ps)$ is the order/disorder threshold; restricted to the $p>\ps$ setting, a modification of the analysis in~\cite{galanis_et_al:LIPIcs.STACS.2026.39} could yield an improved dependency $q=\Omega(\log \Delta)$ but still for integer $q$ only~\cite{Galanis}.}.

At the same time,~\cite{BG24-PTRF} showed that for all real $q>1$, for $p$ sufficiently large (though far from $\ps$), random-cluster dynamics are fast to mix from an arbitrary initialization. Our main result on random graphs is fast (worst-case initialization) mixing of random-cluster dynamics for all $p>\ps$ for $q  = \Omega(\log \Delta)$. 

\begin{theorem}\label{thm:intro-rrg-mixing}
    Fix $\Delta \ge 3$, and $q \ge C_0\log \Delta$ for  large $C_0 > 0$. For $p>\ps$, with probability $1-o(1)$, a random $\Delta$-regular $n$-vertex graph $G$ is such that random-cluster dynamics on $G$ has mixing time $n^{1+o(1)}$. 
\end{theorem}
The $\ps$ threshold is sharp as the mixing time of the random-cluster dynamics on random graphs becomes exponential in $n$ if $q>2$ in the interval $p \in (\pu,\ps)$~\cite{COGGRSV-Metastability-Potts-RRG,Bencs-et-al-RC-RRG}. We also note that in both Theorem~\ref{thm:intro-rrg-mixing} and Theorem~\ref{thm:intro-wired-tree-mixing-time}, the $n^{1+o(1)}$ is $n(\log n)^{O(\log \log n)}$. 

Given that the weak spatial mixing and mixing time inputs to Theorem~\ref{thm:intro-rrg-mixing} hold for all $p>\ps$ and $q>1$, it is natural to expect that Theorem~\ref{thm:intro-rrg-mixing} also holds for all $q>1$. The main obstacle to this in our proofs is that we use a localization scheme
that allows us to deduce the global mixing from the mixing time on balls of radius $\Theta(\log n)$ centered at each vertex.
This approach requires performing a union bound over all $n$ vertices, and therefore requires a large value of $c_*$ in~\eqref{eq:intro-wsm} to ensure a sufficiently strong decay, and this only holds for sufficiently large $q$; see also Remark~\ref{rmk:proof}.

Finally, we note that for integer $q$, Theorem~\ref{thm:intro-rrg-mixing} also implies an $n^{2+o(1)}$ mixing time bound for the Swendsen--Wang dynamics via the comparison estimates of~\cite{Ullrich-random-cluster}.

\subsection{Proof ideas}
We now give some more details on the proof ideas of the paper for establishing Theorems~\ref{thm:uniqueness},~\ref{thm:intro-wired-tree-mixing-time}, and \ref{thm:intro-rrg-mixing}. This subsection is organized according to those theorems, and the corresponding sections in the paper are Section~\ref{sec:trees-uniqueness-WSM},~\ref{sec: fast mixing} and~\ref{sec: rrg mixing} respectively. 

\medskip\noindent\textbf{Uniqueness and spatial mixing on trees.} \ 
We begin by sketching our approach to establishing uniqueness and decay of correlations on the wired $\Delta$-regular tree. At a high level, the reason $\ps$ from~\eqref{eq:p-Haggstrom} is predicted as the threshold for uniqueness, is that when $p>\ps$, one has that $\hat p$ defined in~\eqref{eq:Glauber-update-rule} is bigger than $1/(\Delta -1)$, and the random-cluster measure stochastically dominates a supercritical Bernoulli percolation on the tree. Thus, any random-cluster measure on the tree has a positive density of infinite components that are all connected together in the wired tree, and at $p>\ps$ any such positive density should be enough to make the model behave like it does when all of its leaves are wired together. 
 
Prior work on spatial mixing for random-cluster models on trees at low temperatures (e.g.,~\cite{BG24-PTRF,Galanis_Goldberg_Smolarova_2025}) justified this last step by looking for a cut of the tree separating the leaves from the root, such that \emph{every} vertex on that cut is connected to the wired boundary component. The complement of this event is a percolation event for vertices that are not connected to the wired boundary component. However, for general $q$, the transition point for this percolation event is strictly bigger than $\ps$; indeed this was the approach of~\cite{Jonasson} for proving uniqueness down to the higher threshold $p\gtrsim \frac{\log \Delta}{\Delta}$. 

To get down to the sharp threshold $p>\ps$, instead of using branching processes, we perform a recursion on suitable messages $f(v)$ associated to vertices of a tree, given by a transformation of the probability in its subtree of that vertex $v$ having a path down to the wired boundary component in its subtree. For the appropriate choice of monotone transformation, one then gets a recurrence of the form
\begin{align*}
    f(w) = \prod_{v \text{ child of } w} \Phi(f(v))
\end{align*}
for an explicit function $\Phi$. 
This choice of message function and recursion are inspired by high-temperature random-cluster recursions on trees from~\cite{blanca2020sampling} and~\cite{blanca2023sampling}. 

This tree recursion exactly captures the low-temperature behavior of the random-cluster model on trees, and when $p>\ps$ it exhibits two fixed points for the subtree connection probability: one at $0$, and another at $z^*\in (0,1)$. Notably, as soon as $p>\ps$, the fixed point corresponding to connection probability zero becomes unstable and thus as long as the connection probability is uniformly bounded away from zero, near the bottom of the tree, the recursion pushes it to $z^*$. The strict attractivity of the fixed point at $z^*$ implies this convergence is exponentially fast from any boundary condition that has positive probability of its leaves being connected to infinity, independently. Indeed, the exponential rate can be identified exactly with the derivative of the recursion function at $z^*$. In this manner, we obtain the claimed weak spatial mixing of~\eqref{eq:intro-wsm}, with its sharp exponential rate, among boundary conditions with a positive density of randomly wired leaves. 

To obtain uniqueness from weak spatial mixing, we show that  the boundary condition induced on $\partial \dary$ by any infinite-volume measure is sufficiently wired when $p > \ps$. Weak spatial mixing implies that the influence 
of the boundary condition is asymptotically the same as that of the wired boundary condition as $h \to \infty$. Therefore, every infinite-volume Gibbs measure has the same marginals on every finite set.

\medskip\noindent\textbf{Mixing time on wired trees.} \ 
Having established decay of correlations between the boundary and distant points (at least if the boundary is sufficiently wired), we now describe how that is used to bound the mixing time on trees with wired boundary conditions. A standard tool in the Markov chain mixing time literature for such implications is bounding the mixing time on  graphs by the mixing time on smaller self-similar graphs via \emph{block dynamics} (see e.g.,~\cite{Martinelli-notes}). 
We consider two blocks:
$B_0$ consisting of the edges in the top $(1-\delta) h$ levels of the tree, and $B_1$ containing those in the bottom $(1-\delta) h$ levels. We choose $\delta$ small, so that there is significant overlap between the blocks $B_0$ and $B_1$.
The block dynamics is the non-local Markov chain that in each step picks a block $B\in \{B_0,B_1\}$ and resamples its entire configuration conditionally on the configuration on~$E(\mathcal{T}_h)\setminus B$.

Since our spatial mixing estimates are only for sufficiently wired boundary conditions (and in fact are not true between free and wired boundary conditions), we also define a class of admissible boundary conditions on which we can recurse. 
This family of sufficiently wired boundary conditions on trees of depth $h$, denoted $\mathcal F(h)$, is defined as boundary conditions that are wired enough to satisfy~\eqref{eq:intro-wsm}. If we then let
\begin{align*}
    \Th (h) = \max\nolimits_{\xi \in \mathcal F(h)} T_{\mix}(\mathcal T_h^\xi)\,,
\end{align*}
with $\mathcal T_h^\xi$ denoting the $(\Delta-1)$-ary tree of height $h$ with boundary condition $\xi$, our main recursive estimate for the mixing time of the random-cluster dynamics on trees is the following: 
\begin{align}
\label{eq: tree recursion intro}
    \Th (h) \le h^{O(1)} \cdot \max \Big\{\Th((1-\delta) h)  \frac{|\mathcal T_h|}{|\mathcal T_{(1-\delta) h}|} , (\Delta -1)^{\delta h}  \Th( (1-\delta) h)\Big\}\,.
\end{align}
Iterating this bound down in scale leads to Theorem~\ref{thm:intro-wired-tree-mixing-time} since $h = O(\log n)$. The first term in the maximum is coming from the mixing time of $B_0$ and the second from $B_1$.

The key idea in this recursion is the following. Fix an arbitrary initial configuration $X_0$. If block $B_1$ updates first, then because the random-cluster distribution dominates a supercritical Bernoulli percolation on the tree when $p>\ps$, a positive density of the leaves of $B_0$ will be wired to the bottom of $B_1$ and therefore wired together through the boundary conditions. This then induces a $\theta$-wired boundary on $B_0$ to use the weak spatial mixing of~\eqref{eq:intro-wsm}, and couple the configuration on $B_0 \setminus B_1$ to a stationary sample. Finally, when $B_1$ is updated again, its induced boundary conditions by $B_0 \setminus B_1$ are equal to the stationary sample, and therefore the configuration on $B_1$ will with probability $1$ couple to a stationary sample. 

Some technical hurdles arise, however. The block dynamics method goes via a spectral decomposition and therefore necessitates consideration of worst-case induced boundary conditions, whereas our spatial mixing on trees is only for sufficiently wired boundary conditions. We therefore use the censoring method of~\cite{PW} to ``simulate" the block dynamics, and toss out the low-probability events of developing boundary conditions outside our class $\mathcal F$. Moreover, the block update on $B_1$ is not resampling on a tree, or independent trees, but a union of up to $(\Delta-1)^{\delta h}$ many trees whose roots are possibly wired together (destroying their product structure); this requires careful decoupling of the interactions between the trees through the configuration on $B_0 \setminus B_1$. 

\medskip\noindent\textbf{Mixing time on random regular graphs.} \
We finally describe how spatial and temporal mixing on trees are combined to give mixing time bounds on the random regular graph in its low-temperature regime. By standard localization methods for monotone Markov chains, it is sufficient to show that after time $n (\log n)^{O(\log \log n)}$, for any fixed edge $e\in E(G)$, the probability that $X_t^1$ and $X_t^0$ (the random-cluster dynamics initialized from $X_0^1 = E$ and $X_0^0 = \emptyset$ respectively) agree on $e$ is at least $1-o(1/n)$. 

The first stage of the analysis is to wait a \emph{burn-in period} of $T_{\textsc{Burn}}= O(n \log n)$ for all edges to have been updated at least once, and for the chains to induce sufficiently wired boundary conditions on the ball $B(e)$ centered at an endpoint of $e$ and of radius $r = \frac 15 \log_{\Delta -1} n$.  

After this burn-in period, the chains $X_{T_{\textsc{Burn}}}^0$ and $X_{T_{\textsc{Burn}}}^1$ are sandwiched between a Bernoulli $\hat p$ percolation for $\hat p$ from~\eqref{eq:Glauber-update-rule} and the all-wired configuration. By again using the censoring lemma of~\cite{PW}, it is sufficient to freeze the induced boundary conditions on $B(e)$ and only allow updates in $B(e)$ for $t \ge T_{\textsc{Burn}}$. If the boundary conditions induced by $X_{T_{\textsc{Burn}}}^0$ and $X_{T_{\textsc{Burn}}}^1$ are denoted $\eta^0, \eta^1$ respectively, we claim that with high probability $\eta^0,\eta^1$ are such that the following hold when $p>\ps$:  
\begin{enumerate}[(P1)]\setlength{\itemsep}{2pt}
    \item The mixing time on $B(e)$ with boundary conditions $\eta^0$ or $\eta^1$ is at most $|B(e)| (\log n)^{O(\log \log n)}$. 
    \item If $q>C_0 \log \Delta$, the total variation distance on $e$ between $\mu_{B(e)}^{\eta^0}$ and $\mu_{B(e)}^{\eta^1}$ is $o(1/n)$. 
\end{enumerate}
The formal statement can be found in Lemma~\ref{lem:P1-P2}. Theorem~\ref{thm:intro-rrg-mixing} follows from the combination of these, because after a time $S = n (\log n)^{O(\log \log n)}$ so that the number of updates on $B(e)$ exceeds its mixing time with the induced boundary conditions, the probability of $X_{T_{\textsc{Burn}} +S}^0(e) \ne X_{T_{\textsc{Burn}}+ S}^1(e)$ will be $o(1/n)$.

The properties (P1) and (P2) holding are essentially $O(1)$-sized modifications of~\eqref{eq:intro-wsm} and Theorem~\ref{thm:intro-wired-tree-mixing-time}. Indeed, while $B(e)$ is not necessarily a tree, it has at most one cycle. Also, the boundary conditions $\eta^0$ induced on $B(e)$, while not all-wired, stochastically dominate a distribution that except on $O(1)$ many leaves, independently wires other leaves with  probability at least $\theta>0$. The last fact uses properties of supercritical percolation on the random regular graph, as each leaf of $B(e)$ has a positive probability of being connected to the giant component of the Bernoulli($\hat p$) percolation on $G\setminus B(e)$: see Lemma~\ref{lem:rrg-percolation-properties}. 

The extensions of spatial and temporal mixing from trees with $\theta$-wired boundary conditions to allow for these $O(1)$-sized defects are carried out in Section~\ref{sec: treelike extensions}. 
Spatial and temporal mixing on trees with similar defects had to be handled in~\cite{BG21} in the high-temperature $p<\pu$ case, but the situation at $p>\ps$ is more delicate. Unlike $p<\pu$,  
at low temperatures, a defect (e.g., a cycle near the root edge $e$) macroscopically changes the probability of root to boundary connection. The key point is that it changes it ``in the same way" for the wired boundary as for the random $\theta$-wired boundary, and thus the spatial mixing should still hold. 

To make this more precise, to establish (P1) we carry out the mixing time recursion in~\eqref{eq: tree recursion intro} for a larger family of graphs which is technically more challenging.
For (P2), we use the explicit characterization of the weak spatial mixing rate $c_*$ of~\eqref{eq:intro-wsm} in terms of the derivative of the recursion function, to see that $q$ logarithmically large in $\Delta-1$ is sufficient for the TV distance with the choice of $r = \frac{1}{5}\log_{\Delta-1} n$ to be $o(1/n)$. 
In addition, the presence of a cycle introduces additional challenges as it breaks the tree recurrence. We are able to directly bound the message value at the ``topmost'' vertex of the cycle utilizing the partition function for the random-cluster on the cycle with arbitrary vertices wired together by the boundary condition. 

\begin{remark}
\label{rmk:proof}
For smaller values of $q$ than those covered by our argument, the same tree
recursion gives only a bound of order $n^{-\varepsilon}$ on the single-edge TV distance for some small
$\varepsilon>0$.
It is natural to ask whether this weaker estimate could be
bootstrapped across scales to a TV-distance that is $o(1/n)$, in the
spirit of~\cite{MS13} for the Ising Glauber dynamics. 
The non-locality of random-cluster dynamics however makes this particularly challenging; indeed, such a boosting argument has not even been successfully executed in the high temperature regime. 
\end{remark}

\section{Uniqueness and Weak Spatial Mixing on regular trees}\label{sec:trees-uniqueness-WSM}

In this section, we consider the Gibbs uniqueness of wired infinite-volume random-cluster distributions on the infinite $\Delta$-regular tree.
We establish the uniqueness of this measure
using a decay of correlation property known
as weak spatial mixing (WSM) that is of interest in its own right. 
We compile the required definitions and notations in Section~\ref{subsec:rc}, establish the WSM property in Section~\ref{sec: wsm}, and establish our uniqueness result (Theorem~\ref{thm:uniqueness} from the introduction) in Section~\ref{sec: uniqueness}.

\subsection{Random-cluster boundary conditions}
\label{subsec:rc}

The random-cluster distribution can be defined with \emph{boundary conditions} that capture the effect of external connections or wirings when, for instance, the graph is embedded in a larger ambient graph.
To define this notion, let $\Lambda\subseteq V(G)$ be an arbitrary subset of vertices.
A \emph{boundary condition} on $\Lambda$ is a partition
$
\xi = (P_1,\dots,P_k)
$
of $\Lambda$. Vertices belonging to the same part of the partition are said to be \emph{wired} together.

As such, the random-cluster distribution on $G$ with parameters $p \in (0,1)$ and $q > 0$ and boundary condition $\xi$ on $\Lambda \subseteq V(G)$ is defined as
\begin{align}\label{eq:random-cluster-with-bc}
\mu_{G,p,q}^{\xi}(A)
\;=\;
\frac{
p^{|A|}(1-p)^{|E \setminus A|}
\, q^{c_G(A,\xi)}
}{
Z_{G,p,q}^{\xi}
},
\end{align}
where 
$Z_{G,p,q}^{\xi}$ is the corresponding normalizing constant or partition function, and
$c_G(A,\xi)$ denotes the number of connected components in the subgraph $(V,A)$ 
when taking into account the wirings in $\xi$. (We will typically drop the parameter subscripts $p$ and $q$ to simplify the notation.)

Two important special cases of boundary conditions are the \emph{free} boundary condition, which corresponds to the partition of only singletons, and is equivalent to having no boundary conditions, 
and the \emph{wired} boundary condition, which corresponds to $\xi$ having only one part containing all the vertices from $\Lambda$. With slight abuse of notation we use $\xi=0$ and $\xi=1$ to denote the free and wired boundary condition, respectively.

By the FKG inequality for the random-cluster model, there is a stochastic relation between random-cluster models on $G$ with different boundary conditions when $q \geq 1$. Namely, if $\xi$ is a coarser partition than $\xi'$, then it has ``more wirings" and $\mu_{G}^{\xi} \succeq \mu_{G}^{\xi'}$ holds. Therefore, for a fixed $\Lambda$, the extremal boundary conditions are precisely the free and wired ones we just described. 

\subsection{Weak Spatial Mixing on finite trees}\label{sec: wsm}

We establish WSM for certain structured classes of
boundary conditions.
These classes 
contain at most one non-trivial wired boundary component, and are the relevant class needed for establishing the uniqueness of the wired infinite random-cluster distribution on regular trees.
We introduce some useful notation next.
Recall that we use
$\dreg = (V(\dreg),E(\dreg))$ to denote the complete $\Delta$-regular tree of height $h$. 
Similarly, we set $d = \Delta - 1$, and let $\dary$ denote the complete $d$-ary tree. 
The trees $\dreg$ and $\dary$ differ only on the number of children of the root vertex~$\rho$. 
Unless otherwise specified, boundary conditions on a finite tree $\mathcal T$ are partitions of $\Lambda = \partial \mathcal T$, the set of leaves of $\mathcal T$.

\begin{definition}\label{def:single-component-bc}
    A boundary condition $\xi$ for a tree $\mathcal T$ is \textit{single-component} if the boundary partition has at most one non-singleton element.
\end{definition}

We denote the specified wired component by $\mathcal C_1(\xi)$. The set $\mathcal C_1(\xi)$ may be empty or singleton; when it has at least two vertices, it is the unique non-singleton part of the partition. We use $v \sim \mathcal C_1(\xi)$ for the event that vertex $v$ is connected to $\mathcal C_1(\xi)$. If $\mathcal C_1(\xi)=\varnothing$, this event is understood as the empty event.

\begin{definition}\label{def:theta-wired-bc}
    A distribution over single-component boundary conditions on the boundary of a tree $\mathcal T$ is \emph{$\theta$-wired} if the distribution of the wired component stochastically dominates the distribution over random subsets $A \subseteq \partial \mathcal{T}$ where each vertex of $\partial\mathcal{T}$ is included in $A$ with probability $\theta$ independently.
\end{definition}

We will show that a form of WSM holds on $\Delta$-regular trees for 
\emph{$\theta$-wired} boundary conditions with high probability over the choice of the boundary condition. We first show that under typical $\theta$-wired boundary conditions the boundary is sufficiently connected so that vertices that are not too close to the boundary already have a uniform bias toward being connected to the wired boundary component. We then use a tree recursion to prove decay of influence of the boundary condition up the tree, yielding the WSM property. Throughout the remainder of the paper, we assume that $\theta>0$ whenever a $\theta$-wired boundary condition is considered.

To state the result formally, we introduce necessary notation first. We use $\dist(u, v)$ to denote the graph distance between vertices $u$ and $v$, and for any vertex $v$ of a tree $\mathcal{T}$ we let $\dist(v, \partial \mathcal{T}) = \min_{w \in \partial \mathcal{T}} \dist(v, w)$.

\begin{theorem}\label{thm wsm}
Let $\Delta \ge 3$, $q > 1$, $p > \ps$, and let $\mathsf M$ be a $\theta$-wired distribution on single-component boundary conditions of the complete $\Delta$-regular tree $\dreg$ of height $h$ and $\xi \sim \mathsf{M}$. Then there exist constants $c > 0$, $C>0$, $\beta\in(0,1)$, and $\gamma > 0$ depending only on $\Delta$, $p$, $q$ and $\theta$, such that for any vertex $v\in V(\dreg)$ with $\dist(v,\partial\dreg) \ge \gamma$, with probability at least $1-e^{ - c d^{ \sqrt{\dist(v,\partial \dreg)}}}$,
\begin{align*}
    |\mu_{\dreg}^1 (v\sim \partial\dreg) - \mu_{\dreg}^\xi(v\sim \mathcal C_1(\xi))|\le C \beta^{\dist(v,\partial\dreg)}\,.
\end{align*}
Moreover, the same holds for the complete $d$-ary tree $\dary$. 
\end{theorem}

We in fact prove Theorem~\ref{thm wsm} with the sharp $\beta$ for the exponential rate of decorrelation: see Remark~\ref{rem:choice-of-beta}. 
We rely on tree recursions 
to establish Theorem~\ref{thm wsm}.
To set them up,
for any $\mathcal T$,
let $\mathcal T^v$ denote the same tree rooted at vertex $v \in V(\mathcal T)$. For any  $u \in V(\mathcal{T}^v)$, let $\mathcal{T}^{v, u}$ denote the subtree at $u$ in $\mathcal T^v$. 

Fixing a single-component boundary condition $\xi$ on $\partial \mathcal T^v$, let
$Z^{\xi, 0}_{\mathcal{T}^v}$ denote the contribution to the partition function of the random-cluster configurations on $\mathcal{T}^v$ under the boundary condition $\xi$ in which $v$ is not connected to the wired component $\mathcal C_1(\xi)$;
analogously define $Z^{\xi, 1}_{\mathcal{T}^v}$
for the total weight of the configurations
that connect $v$ to $\mathcal C_1(\xi)$. (Note that $Z^{\xi}_{\mathcal{T}^v} = Z^{\xi, 0}_{\mathcal{T}^v}+Z^{\xi, 1}_{\mathcal{T}^v}$.)

For every $u\in V(\mathcal T^v)$, let $\xi_{(u)}$ be the restriction of the boundary partition $\xi$ to $\partial\mathcal T^{v,u}$, with specified wired component $\mathcal C_1(\xi_{(u)}) := \mathcal C_1(\xi)\cap\partial\mathcal T^{v,u}$. Thus $\xi_{(u)}$ is again a single-component boundary condition, and its wired component $C_1(\xi_{(u)})$ may be empty. If $\mathcal C_1(\xi_{(u)})=\varnothing$, we set $Z^{\xi_{(u)},1}_{\mathcal T^{v,u}} := 0$ and $Z^{\xi_{(u)},0}_{\mathcal T^{v,u}} := Z^{\xi_{(u)}}_{\mathcal T^{v,u}}$.

\begin{definition}\label{def:message recursion}
    For any finite rooted $\mathcal{T}^v$ and any single-component boundary condition $\xi$, define the \emph{message function} $f^\xi_{\mathcal{T}^v} : V(\mathcal{T}^v) \to [1,\infty) \cup \{\infty\}$  by 
    \begin{align}\label{message def}
        f^\xi_{\mathcal{T}^v}(u) = q\frac{Z^{\xi_{(u)}, 1}_{\mathcal{T}^{v, u}}}{Z^{\xi_{(u)}, 0}_{\mathcal{T}^{v, u}}} + 1.
    \end{align}
    For $u \in \partial \mathcal T^v$, we define $f^\xi_{\mathcal{T}^v}(u) = 1$ if $u \notin \mathcal C_1(\xi)$, $f^\xi_{\mathcal{T}^v}(u) = +\infty$ if $u\in \mathcal C_1(\xi)$. 
\end{definition}
The tree structure allows us to formulate a recursion on the message functions. 
Let $N(u)$ denote the set of children of a vertex $u \in V(\mathcal{T}^v)$.
The following lemma is a generalization of the high-temperature tree recursion argument in~\cite[Fact 3.2]{blanca2023sampling} and its proof is deferred to Appendix~\ref{appendix:tree-recursion}. 

\begin{lemma}\label{lem:recursion}
    For any finite rooted tree $\mathcal{T}^v$ and any single-component boundary condition $\xi$, the message function $f^\xi_{\mathcal{T}^v}$ satisfies the following for every internal vertex $u \in V(\mathcal{T}^v)$:
    \begin{align}\label{message recursion}
    f^\xi_{\mathcal{T}^v}(u) = \prod_{w \in N(u)}\Phi(f^\xi_{\mathcal{T}^v}(w))
    \end{align}
    where $\Phi(x) = \frac{x + (q-1)(1-p)}{(1-p)x + p + (q-1)(1-p)}$ for $x \in [1, \infty)$ and $\Phi(\infty) = (1-p)^{-1}$.
\end{lemma}

\subsubsection{Homogeneous tree recurrence and fixed point analysis}

We first analyze the message function recursion for the special case of the complete $d$-ary tree. In this setting, the recursion is one-dimensional, and we shall see that its fixed point structure 
identifies the uniqueness regime. These facts will be used in our proof of Theorem~\ref{thm wsm}.

Let us define a function $g : [1, \infty) \to \R$ as
\begin{equation}\label{eq: function g}
g(y) = \Phi(y)^d = \left(\frac{y + (q-1)(1-p)}{(1-p)y + p + (q-1)(1-p)} \right)^d.
\end{equation}
The fixed points of $g$ correspond to the possible stationary values of the message function for the $d$-ary tree case and are described in the following theorem. 
For $k \ge 1$, let $g^{(k)}(x) = g(g^{(k-1)}(x))$ and let $g^{(0)}(x) = x$. 
While we are only interested in the fixed points when $p > \ps$, we provide a full characterization of fixed points of $g$ which may be of future interest: see also Figure~\ref{fig:graph}.

\begin{theorem} \label{thm:fixed points:appendix}
    Let $\Delta \ge 3$, $q \geq 1$, and $p \in(0, 1)$.
    Then, the function $g$ has a fixed point at $y = 1$. Furthermore, the fixed points of $g$ in $[1, \infty)$ satisfy:
    \begin{enumerate}[(i)]
        \item If $p < \pu$, $g$ has the unique fixed point at $y=1$;
        \item If $q > 2$ and $p \in (\pu, \ps)$, $g$ has three fixed points;
        \item If $p > \ps$, $g$ has two fixed points;
        \item If $q > 2$ and $p = \pu$ or $\ps$, $g$ has two fixed points;
        \item If $q \leq 2$ and $p = \pu=\ps$, $g$ has the unique fixed point at $y = 1$.
    \end{enumerate}
\end{theorem}
\begin{figure}[t]
    \centering

    \begin{subfigure}[t]{0.325\linewidth}
        \centering
        \begin{overpic}[width=\linewidth,clip,trim={0cm 0cm 35.8cm 0cm}]{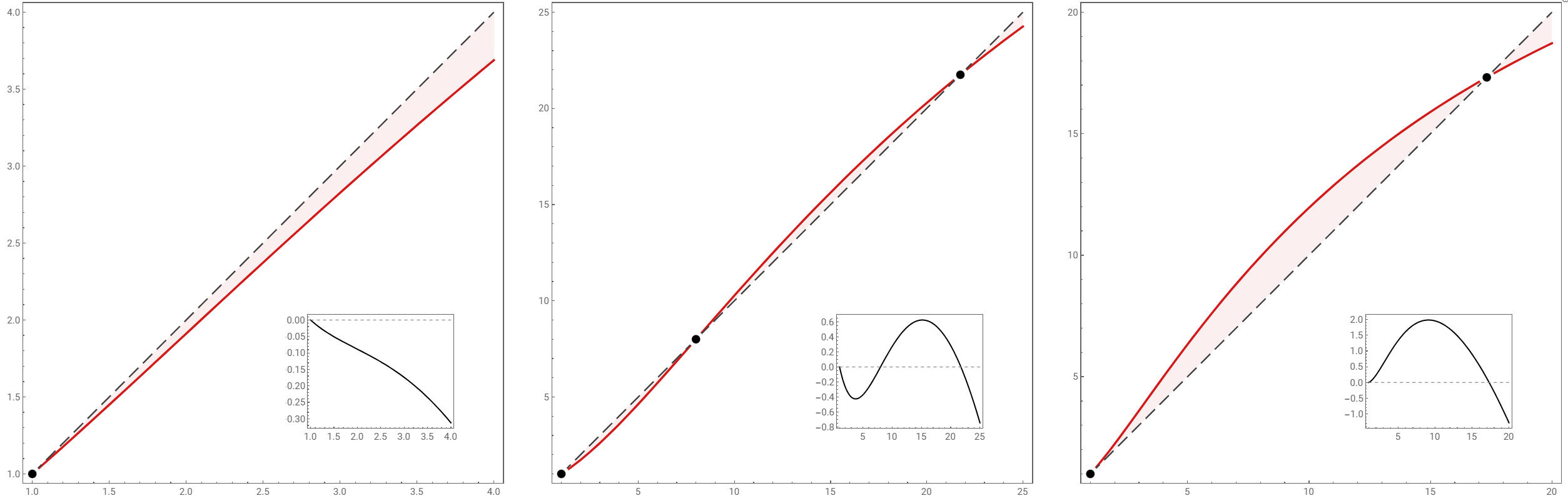}
            \put(78,66){\makebox(0,0)[l]{{\tiny{\color{red}$g(y)$}}}}
            \put(67,39){\makebox(0,0)[l]{{\tiny{$g(y)-y$}}}}
        \end{overpic}
        \caption{$p<\pu$}
    \end{subfigure}
    \hfill
    \begin{subfigure}[t]{0.325\linewidth}
        \centering
        \begin{overpic}[width=\linewidth,clip,trim={17.9cm 0cm 17.9cm 0cm}]{graph.png}
            \put(54,68){\makebox(0,0)[l]{{\tiny{\color{red}$g(y)$}}}}
            \put(67,39){\makebox(0,0)[l]{{\tiny{$g(y)-y$}}}}
        \end{overpic}
        \caption{$\pu<p<\ps$}
    \end{subfigure}
    \hfill
    \begin{subfigure}[t]{0.325\linewidth}
        \centering
        \begin{overpic}[width=\linewidth,clip,trim={35.8cm 0cm 0cm 0cm}]{graph.png}
            \put(40,60){\makebox(0,0)[l]{{\tiny{\color{red}$g(y)$}}}}
            \put(67,39){\makebox(0,0)[l]{{\tiny{$g(y)-y$}}}}
        \end{overpic}
        \caption{$p>\ps$}
    \end{subfigure}

    \caption{Plots of $g(y)$ in the various regimes of $p$ for $q>2$ and $\Delta \ge 3$. This depicts the transition in the number of fixed points as $p$ passes through $\pu$ and $\ps$ as shown in  Theorem~\ref{thm:fixed points:appendix}.}
    \label{fig:graph}
\end{figure}
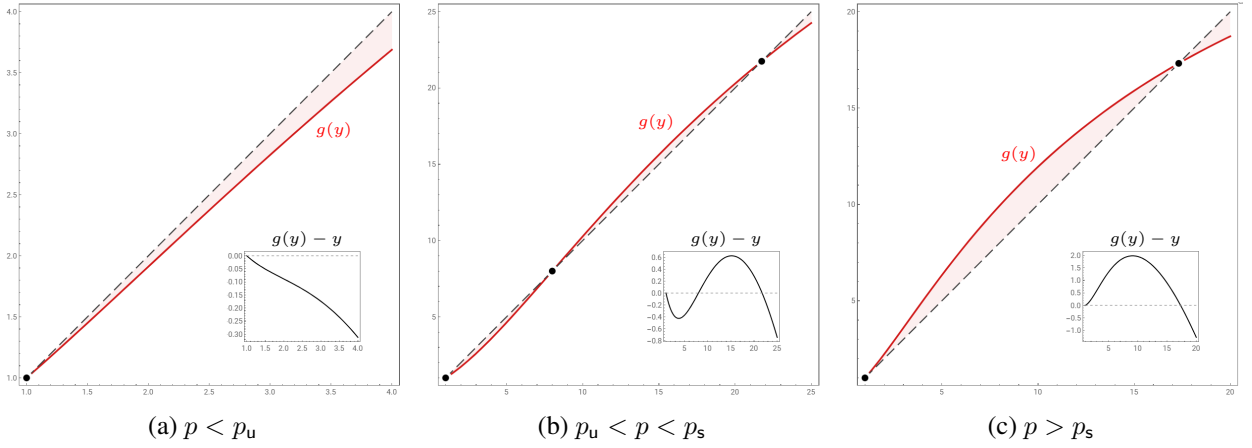

We will also use the following lemma about the function $g$ when $p > \ps$. Its proof as well as that of Theorem~\ref{thm:fixed points:appendix} are provided in Appendix \ref{appendix:fp}.

\begin{lemma} \label{lemma:fp:main}
    For $\Delta \geq 3, q > 1, p > \ps$, $g$ is continuously differentiable, strictly increasing and bounded above on $[1, \infty)$. The non-trivial fixed point $y^* > 1$ is the only attractive one; that is,
    for any initial point $y_0 > 1$, $\lim_{k \to \infty} g^{(k)}(y_0) = y^*$. Moreover, $g'(y^*) \in (0, 1)$.
\end{lemma}

\subsubsection{Weak Spatial Mixing: Proof of Theorem~\ref{thm wsm}}

The proof of Theorem~\ref{thm wsm} proceeds in two main steps. First, we show that under a typical $\theta$-wired boundary condition for any $\theta>0$, the induced messages are uniformly bounded away from $1$ at vertices shortly away from the boundary putting them in the basin of attraction of the $y^*$ fixed point. We then establish that in this regime the message recursion is contracting, so that discrepancies between different boundary conditions decay exponentially as one moves toward the root. Combining these two ingredients yields WSM as in Theorem~\ref{thm wsm}.

\begin{lemma}\label{lem:uniform-positive}
    Let $\Delta \ge 3$, $q > 1$, $p > \ps$, $\theta > 0$ and let $\mathsf M$ be a $\theta$-wired distribution on single-component boundary conditions of the complete $\Delta$-regular tree $\dregv$
    of height $h$ rooted at a vertex $v$.
    Then there exist constants $\varepsilon=\varepsilon(\Delta,p,q,\theta)>0$, and sufficiently large $M=M(\Delta,p,q) > 0$ (not depending on $\theta$), such that for all  $\ell \ge 1$, all vertices $x \in V(\dregv)$ with $\dist(v, x) = h - \ell$, 
    with probability at least $1-\exp\big( - \frac{\theta d^{\ell}}{8}+ h\log d \big)$, for $\xi\sim \mathsf{M}$, we have $1 + \eps \leq f^{\xi}_{\dregv}(x) \leq M$.
    The same holds for
    $\daryv$, the $d$-ary tree of height $h$ rooted at~$v$.
\end{lemma}

\begin{proof}
For ease of notation set $\mathcal T^v = \dregv$.
Let $k = h - \ell$ and
let $\mathcal{T}^v_{k} := \{u \in \mathcal{T}^v : \dist(v, u) \leq k\}$ be the truncated tree rooted at $v$.  
For $x \in \partial \mathcal{T}^v_{k}$,
let $A_x$ be the number of vertices wired by $\xi$ in $\partial\mathcal T^{v, x}$, and let $\mathcal{A}_x$ be the event $A_x \geq \frac{\theta |\partial\mathcal T^{v, x}|}{2}$.
Since $A_x$ stochastically dominates a $\mathrm{Binomial}(|\partial\mathcal T^{v, x}|, \theta)$, by Chernoff, 
\begin{equation}
    \label{eq:m:pr1}
\pr\left(\mathcal{A}_x\right) \geq 1 - e^{-\frac{\theta |\partial\mathcal T^{v, x}|}{8}}.
\end{equation}
This bound holds for every $x \in \partial \mathcal{T}^v_{k}$; hence, since $|\partial\mathcal T^{v, x}| \ge d^{\ell}$, we obtain by union bound that

\begin{equation}
\label{eq:prob:bound}
    \pr\Big(\bigcap\nolimits_{x \in \partial \mathcal{T}^v_{k}} \,\mathcal{A}_x\Big) \geq 1 - \Delta d^{k - 1}e^{-\frac{\theta d^{\ell}}{8}} = 1 - e^{-\frac{\theta d^{\ell}}{8} + (k-1) \log d + \log \Delta} \ge 1 - e^{-\frac{\theta d^{\ell}}{8} + h\log d}.
\end{equation}

We consider next the probability that $x \in \partial \mathcal{T}^v_k$ is connected to one of the vertices wired by $\xi$ in $\partial \mathcal T^{v,x}$.
It is a standard fact (see~Theorem 3.21 in~\cite{grimmett2006random}), that
the random–cluster distribution 
stochastically dominates the independent Bernoulli bond percolation with a parameter $\hat p = p/(p+(1-p)q)$.
Since $\mathcal T^v$ is the result of rooting at $v$ the complete $\Delta$-regular tree, 
$\mathcal T^{v,x}$ is a complete $d$-ary tree.
On the $d$–ary tree, $\hat p > 1/d$ when $p>\ps$, and so the percolation is in the supercritical regime.
The probability that $x$ is connected to $\partial \mathcal T^{v, x}$ is uniformly at least the probability of that supercritical branching process surviving forever, which we denote by $\varphi\in (0,1)$. 

Conditioned on $\mathcal A_x$, by symmetry, the probability that $x$ is connected to the wired component in $\partial \mathcal T^{v, x}$ is bounded below by $\frac{\varphi \theta}{2}$.
By the definition of the message $f^\xi_{\mathcal{T}^v}(x)$ at $x$, this probability can be rewritten as:
\begin{equation}
\label{eq:m:pr}
    \frac{Z^{\xi_{(x)},1}_{\mathcal{T}^{v,x}}(x)}{Z^{\xi_{(x)},0}_{\mathcal{T}^{v,x}}(x) + Z^{\xi_{(x)},1}_{\mathcal{T}^{v,x}}(x)} = \frac{f^\xi_{\mathcal{T}^v}(x) - 1}{f^\xi_{\mathcal{T}^v}(x) + q - 1} \geq \frac{\varphi \theta}{2}, 
\end{equation}
where recall that $\xi_{(x)}$ denotes the restriction of $\xi$ to $\partial \mathcal T^{v,x}$. 
This implies that $f^\xi_{\mathcal{T}^v}(x) \geq 1 + \frac{q\varphi \theta}{2-\varphi \theta}$.
Thus, conditioned on $\bigcap\nolimits_{x \in \partial \mathcal{T}^v_k} \,\mathcal{A}_x$, we have $f^\xi_{\mathcal{T}^v}(x) \geq 1 + \eps$ for all $x \in \partial \mathcal{T}^v_{k}$ with  $\eps = \frac{q\varphi \theta}{2-\varphi \theta}  > 0$, and the result follows from~\eqref{eq:prob:bound}.
For the upper bound on the message, the function $\Phi$ in \eqref{message recursion} is increasing on $[1, \infty)$ and moreover satisfies $\lim_{y \to \infty}\Phi(y) = \frac{1}{1-p}$.
Thus $\Phi \leq \frac{1}{1-p}$, and $f^{\xi}_{\mathcal T^v}(x) \leq (1-p)^{-\Delta}$ for all $x \notin \partial \mathcal{T}^v$.

The same argument goes through in exactly the same manner for the $d$-ary tree $\daryv$. The only difference is that one vertex in one of the subtrees 
$\mathcal T^{v,x}$ may have branching $d-1$ instead of $d$ but this only affects $\varphi$
and the probability of connecting to the wired boundary
by at most a constant factor.
\end{proof}

The above shows that for any $\theta$-wired boundary conditions, with high probability $\xi$ will be such that a little above the leaves of $\dregv$, all the messages lie in a compact interval uniformly bounded away from the fixed point at $1$.  
The following lemma establishes exponential contraction to the wired fixed point $y^*$ for complete \(d\)-ary trees starting from such incoming messages.

\begin{lemma}\label{lem: uniform contraction}
Let $\Delta \ge 3$, $q > 1$, and $p > \ps$.
Suppose we are given two assignments of leaf messages on $\partial \dary$, denoted by
$f_1,f_2 : \partial \dary \to [1,\infty)$, such that there exist $\eps \in (0,1)$ and $M>0$ for which for all $x\in \partial\dary$ we have
$
1+\eps \le f_1(x),f_2(x)\le M
$.
Let $\hat f_1,\hat f_2 : V(\dary)\to [1,\infty)$ be the corresponding message functions obtained by propagating these boundary values to the interior using the recursion in~\eqref{message recursion}. Then, there exist $C>0$ and $\beta\in(0,1)$ depending on $p$,
$q$, $d$, $\eps$, and $M$ such that for every vertex $u \in V(\dary):$
\[
|\hat f_1(u)-\hat f_2(u)|
\le C\,\beta^{\dist(u,\partial\dary)}.
\]
\end{lemma}

\begin{proof}
For ease of notation we set $\mathcal T = \dary$;
let $\rho$ denote the root of $\mathcal T$.
Fix a vertex $u \in V(\mathcal T)$ and let $s = \dist(u, \partial\mathcal T)$. 
Consider two sequences, $\{\tilde m_k\}^s_{k=0}$ and $\{\tilde M_k\}^s_{k=0}$ where 
$\tilde m_s = 1+\eps$, $\tilde M_s = M$, and $\tilde m_k = g(\tilde m_{k+1})$, $\tilde M_k = g(\tilde M_{k+1})$ for $k = 0, \dots, s-1$; $g$ is the function defined in \eqref{eq: function g}.
The role of these sequences is that 
since $g$ is increasing (see Lemma~\ref{lemma:fp:main}), we have that for any $k$ , any vertex $w \in V(\mathcal T^{\rho,u})$ at distance $k$ from $\partial \mathcal T$, and any $i \in \{1, 2\}$
\begin{equation}
\label{eq:message:sandwich}
    \tilde m_{s-k} \le \hat f_i(w) \le \tilde M_{s-k}.
\end{equation}
By Lemma~\ref{lemma:fp:main}, $g$ has a unique attractive fixed point at $y^*>1$ when 
$p > \ps$,
and $\tilde m_0 \to y^*$ and $\tilde M_0 \to y^*$ as $s \rightarrow \infty$. Moreover, since $g$ is differentiable on $[1, \infty)$, and $g'(y^*)<1$ by Lemma \ref{lemma:fp:main}, by continuity of $g'$ there exists $\eta > 0$ such that 
\[
    \beta = \sup_{z \in [y^* - \eta,y^* + \eta]} g'(z) \in (0, 1).
\]

For any $k > 0$, we have $|\tilde m_{k - 1} - \tilde m_k| = |g( \tilde m_k) - \tilde m_k| > 0$ unless $\tilde m_k = y^*$ and the same holds for $\tilde M_k$. 
Therefore, after a finite number $t_0(\eps,M)$ of applications of the function $g$ to 
$\tilde m_{s}$ and $\tilde M_{s}$, we have that $\tilde m_k,\tilde M_k \in [y^*-\eta,\,y^*+\eta]$ for all $k \leq s-t_0 = : K$. (Note that we may assume that $t_0(\eps,M) < s$, as otherwise the result is vacuously true for a large enough $C$.)
Then for $k\leq K$ by the mean value theorem:
\[
\tilde M_{k-1} - \tilde m_{k-1} = g(\tilde M_k) - g(\tilde m_k) = g'(\zeta_k)(\tilde M_k-\tilde m_k)
\]
for some $\zeta_k \in [\tilde m_k,\tilde M_k]$. Since $\zeta_k \in [y^*-\eta,y^*+\eta]$, we have $g'(\zeta_k)\le \beta$, hence
\[
\tilde M_{k-1}-\tilde m_{k-1} \leq \beta(\tilde M_k-\tilde m_k).
\]
Iterating this inequality up to the vertex $u$ yields 
\[
    \tilde M_0-\tilde m_0 \le (\tilde M_{K}-\tilde m_{K}) \beta^K = C\beta^s
\]
for some constant $C>0$ depending only on $p$, $q$, $d$, $\eps$ and $M$. Indeed, $\tilde m _K$ and $\tilde M_K$ are bounded since the function $g$ is bounded on a compact set by the continuity of $g$, so $|\tilde M_K - \tilde m _K|$ is at most a constant. 
From~\eqref{eq:message:sandwich},
we conclude that 
$|\hat f_{1}(u)-\hat f_{2}(u)| \leq \tilde M_0-\tilde m_0 \leq C\beta^{\dist(u,\partial\dary)}$. 
\end{proof}

The previous lemma establishes the exponential contraction of messages for the complete $d$-ary tree when all incoming (i.e., leaves) messages lie in a fixed compact interval above $1$. 
We extend this fact now to $\Delta$-regular trees. These only differ from the $d$-ary tree at one vertex, and the message function there has a bounded Lipschitz constant so up to a change of constant, we obtain the following (the full proof is given in Section~\ref{subsec:extensions}). 

\begin{corollary}\label{cor:uc}
The conclusion of Lemma~\ref{lem: uniform contraction} also holds if the complete $d$-ary tree $\dary$ is replaced by the finite rooted $\Delta$-regular tree $\dreg$. 
\end{corollary}

We are now ready to prove Theorem~\ref{thm wsm}, which follows from Lemma~\ref{lem:uniform-positive} and Corollary~\ref{cor:uc}.

\begin{proof}[Proof of Theorem \ref{thm wsm}]
Fix a vertex $v\in V(\dreg)$ and consider the rooted tree $\dregv$. 
For ease of notation let $h_v = \dist(v,\partial \dregv)$.
Let $\mathcal{T}^v_{k} := \{u \in V(\dregv) : \dist(v, u) \leq k\}$ with $k =  h_v - \sqrt{h_v}$. Note that $\mathcal{T}^v_{h_v}$ is a complete $\Delta$-regular tree of height $h_v$.
By Lemma~\ref{lem:uniform-positive}, for all 
$x \in \partial \mathcal{T}^v_{k}$
we have
$1 + \eps \leq f^{\xi}_{\dregv}(x) \leq M$ with probability 
at least $1-\exp\big( - \frac{\theta d^{\sqrt{h_v}}}{8}+h_v\log d\big)$ for a suitable $\varepsilon > 0$ and $M > 0$ sufficiently large, and, by the same argument, the same holds for 
$f^{1}_{\dregv}(x)$.
Applying Corollary \ref{cor:uc} to $\mathcal{T}^v_{k}$ we obtain for  suitable $\beta_0 \in (0, 1)$ and $C > 0$ that
$$
|f^{1}_{\dregv}(v) -  f^{\xi}_{\dregv}(v)| \leq C \beta_0^{k}.
$$

We next translate this bound into an analogous one for the difference in probability between the events $v \sim \mathcal C_1(\xi)$ under $\xi$ and $v\sim \partial\dreg$ under the wired boundary. With a little algebra, one sees that for two single-component boundary conditions $\xi_1,\xi_2$, we have 
\begin{align}\label{eq:message to probability}
    |\mu^{\xi_1}_{\dregv}(v\sim \mathcal{C}_1(\xi_1)) \!- \!\mu^{\xi_2}_{\dregv}(v \sim \mathcal{C}_1(\xi_2))|
= \frac{q| f^{\xi_1}_{\dregv}(v)  \!- \! f^{\xi_2}_{\dregv}(v)|}{ (f^{\xi_1}_{\dregv}(v) + q - 1)(f^{\xi_2}_{\dregv}(v) + q - 1)}   \le \frac{1}{q} |f^{\xi_1}_{\dregv}(v) \!- \!f^{\xi_2}_{\dregv}(v)| 
\end{align}
where the last inequality used that $f_{\dregv}^{\xi_1}, f_{\dregv}^{\xi_2}\ge 1$. 
By~\eqref{eq:message to probability}
\begin{align*} 
|\mu^1_{\dregv}(v\sim \partial\dreg) - \mu^{\xi}_{\dregv}(v \sim \mathcal{C}_1(\xi))| \le \frac{1}{q} \cdot |f^1_{\dregv}(v)-f^{\xi}_{\dregv}(v)|
\leq  \frac{C}{q} \beta_0^{k}.
\end{align*}

For a large enough constant $\gamma$, 
there exists a sufficiently small constant $a > 0$ such that for $h_v \ge \gamma$ we have
$h_v -\sqrt{h_v} \ge a h_v$; thus
for a suitable constant $\beta \in (0, 1)$:
\begin{equation} \label{eq:main:wsm}
|\mu^1_{\dregv}(v\sim \partial\dreg) - \mu^{\xi}_{\dregv}(v \sim \mathcal{C}_1(\xi))| \leq \frac{C\beta^{h_v}}{q}.
\end{equation}
Finally, for $\gamma$ large enough we have for all $h_v \ge \gamma$, 
$
\frac{\theta d^{\sqrt{h_v}}}{8} - h_v\log d \ge \frac{\theta d^{\sqrt{h_v}}}{16},
$
so that \eqref{eq:main:wsm} holds with probability at least $1 -\exp\big( -\frac{\theta d^{\sqrt{h_v}}}{16}\big)$
as claimed.
\end{proof}

\begin{remark}\label{rem:choice-of-beta}
    It is clear going through the above proofs that if we define $\beta_* := g'(y^*)$, then the exponential decay rate of Theorem~\ref{thm wsm} can be identified as exactly $\beta_*$. More precisely, we have the bound 
    \begin{align*}
        \big| \mu_{\dreg}^1 (v\sim \partial\dreg) - \mu_{\dreg}^\xi (v\sim \mathcal C_1(\xi))\big| \le C (\beta_* + o(1)) ^{\dist(v,\partial \dreg) - \sqrt{\dist(v,\partial \dreg)}}\,.
    \end{align*}
\end{remark}

We now state two direct corollaries of Theorem~\ref{thm wsm}. They both establish 
analogous decay bounds; the first
does it for the difference on edge probabilities and the second for the case when the root vertex is wired to the boundary component as part of the boundary condition. 
Both facts will be used in subsequent proofs.

In what follows we use $A$ as the random subgraph sampled from a random-cluster distribution; e.g., $e \in A$ is the event that edge $e$ is part of the random-cluster configuration. 

\begin{corollary}
    \label{cor: edge WSM}
    Under the same assumptions as in Theorem \ref{thm wsm}, 
    with the same probability, 
    there exist constants
    $C > 0$, $\beta \in (0, 1)$, and 
    $\gamma > 0$  depending only on $\Delta$, $p$, $q$, and $\theta$ such that
    for any edge $e \in E(\dreg)$ with $\dist(e, \partial\dreg) \ge \gamma$  we have 
    \[
    |\mu^1_{\dreg}(e \in A) - \mu^\xi_{\dreg}( e \in A)| \leq C \beta ^{\dist(e, \partial\dreg)}.
    \]    
    Here, the constants $\beta$ and $\gamma$ are the same as those in Theorem~\ref{thm wsm}.
\end{corollary}
\begin{proof}
For an edge $e=\{u,v\}$ with $\dist(e,\partial\dreg)\ge \gamma$, reveal all edges except $e$, and consider the monotone coupling $\mathbf{P}$ of $\mu^1_{\dreg}$ and $\mu^\xi_{\dreg}$. Conditional on the revealed configuration, the law of $e$ depends only on the indicators of the events that $u$ and $v$ are connected to the wired component in their respective subtrees, and this dependence is monotone in these indicators. Therefore, by monotonicity, the value of $e$ in the coupling can differ only if the wired status of at least one of $u$ or $v$ differs between the two configurations. Therefore,
\[
|\mu^1_{\dreg}(e\in A) - \mu^\xi_{\dreg}( e\in A)| \leq \mathbf{P}(\mathbf 1_{\{u\sim \partial\dreg\}}\neq \mathbf 1_{\{u\sim \mathcal C_1(\xi)\}}) + \mathbf{P}(\mathbf 1_{\{v\sim \partial\dreg\}}\neq \mathbf 1_{\{v\sim \mathcal C_1(\xi)\}}),
\]
and, by monotonicity, the right-hand side of this inequality is equal to
\[
|\mu^1_{\dreg}(v\sim \partial\dreg) - \mu^{\xi}_{\dreg}(v \sim \mathcal{C}_1(\xi))| +|\mu^1_{\dreg}(u\sim \partial\dreg) - \mu^{\xi}_{\dreg}(u \sim \mathcal{C}_1(\xi))|,
\]
and the result follows.
\end{proof}

Let $\mu^{1,\circlearrowleft}_{\dreg}$ and $\mu^{\xi,\circlearrowleft}_{\dreg}$ denote the random-cluster distribution with boundary condition $1$ and $\xi$, respectively, 
in which the root of $\dreg$ is wired to $\partial\dreg$ and $\mathcal C_1(\xi)$, respectively.

\begin{lemma} \label{lemma: root decay}
    Under the same assumptions as in Theorem \ref{thm wsm}, 
    with the same probability, 
    there exist constants
    $C > 0$, $\beta \in (0, 1)$, and 
    $\gamma > 0$  depending only on $\Delta$, $p$, $q$, and $\theta$ such that
    for any $v \in V(\dreg)$ with $\dist(v, \partial\dreg) \ge \gamma$, we have 
\[
\left| \mu^{1,\circlearrowleft}_{\dreg}(v\sim \partial\dreg) 
- \mu^{\xi,\circlearrowleft}_{\dreg}(v \sim \mathcal C_1(\xi)) \right|
\leq C \beta^{\dist(v, \partial\dreg)}.
\]
Here, the constants $\beta$ and $\gamma$ are the same as those in Theorem~\ref{thm wsm}. Moreover, the same holds for the complete $d$-ary tree $\dary$.
\end{lemma}

We prove Lemma~\ref{lemma: root decay} in Section~\ref{subsec:extensions}, where several extensions of the WSM property from Theorem~\ref{thm wsm} to allow for $O(1)$ additional wirings to the boundary are proved in a unified manner. 

\subsection{Uniqueness of the wired tree measure for $p > \ps$} \label{sec: uniqueness}

We will now establish our Theorem~\ref{thm:uniqueness}. 
While the random-cluster model on finite graphs with boundary conditions is perfectly well-defined per~\eqref{eq:random-cluster-with-bc}, more care is needed for the infinite-volume measures. We are especially interested in infinite-volume measures on the wired tree $\mathcal T_\infty = (V(\mathcal T_\infty),E(\mathcal T_\infty))$. Let us recall the Dobrushin-Lanford-Ruelle formalism of infinite-volume Gibbs measures (we will do so specifically in the context of random-cluster measure on the wired tree). 

\begin{definition}[DLR Condition]
\label{def:dlr}
A probability measure $\nu$ on $\{0,1\}^{E(\mathcal T_\infty)}$ is called an infinite-volume measure 
for the random-cluster model on the wired tree at parameters $p,q$ if its conditional probabilities satisfy the following. 

For every finite edge set $\Lambda \subset E(\mathcal{T}_\infty)$ and for every 
configuration $\eta \in \{0,1\}^\Lambda$, we have
\[
\nu(A(\Lambda) = \eta \mid A({E(\mathcal T_\infty) \setminus \Lambda})) =
\mu^{1_\infty \sqcup A({E(\mathcal T_\infty)\setminus \Lambda})}_{\Lambda, p, q}(\eta)
\]
with $\nu$-probability $1$ over $A({E(\mathcal T_\infty) \setminus \Lambda})$.
Here the boundary conditions induced by $1_\infty \sqcup A(E(\mathcal T_\infty) \setminus \Lambda)$ is such that every infinite connected component in $A(E(\mathcal T_\infty)\setminus \Lambda)$ is considered connected up to each other ``at infinity". 
\end{definition}

Unlike the traditional spin system setting, the set of wired tree measures does not consist of all limits of finite tree measures with boundary configurations and wirings ``at infinity" because the limit measure of the all-empty boundary condition will always not see the wirings at infinity even at large $p$ and will just be a product measure $\bigotimes \text{Ber}(\hat p)$. This restricted class, however, has been the primary object of study for the random-cluster phase transition on the tree (see~\cite{grimmett2006random}, Section 10) as it is what captures the static and dynamical transitions of the model. Jonasson~\cite{Jonasson} introduced a formalism using a ghost vertex whose connection parameter to every other vertex is some parameter $r\downarrow 0$, so that in that framework subsequential limits of finite-volume random-cluster measures with all boundary conditions yield infinite-volume measures on the wired tree. 

To show Theorem~\ref{thm:uniqueness}, we show that any Gibbs measure induces, from sufficiently far away, a boundary condition that is $\theta$-wired with high probability for some $\theta>0$. The weak spatial mixing result Theorem~\ref{thm wsm} then implies that its local marginals must agree with those of the wired limit.

\begin{proof}[Proof of Theorem~\ref{thm:uniqueness}]
Let $\mu^\mathrm{wired}$ be the infinite-volume measure obtained as the following limit. 
\begin{align*}
    \mu^\text{wired}: = \lim_{h\to\infty} \mu_{\dreg}^1\,.
\end{align*}
The existence of this limit is standard due to monotonicity, and the fact that it satisfies the DLR conditions for the wired tree specifically was shown e.g., in Theorem 2.9 of~\cite{Jonasson} (see also the remark before it on page 342 of~\cite{Jonasson}). By this limit, we mean that for any finite set $\Lambda$, the marginal $\mu^{\text{wired}} \vert_\Lambda:= \mu^{\text{wired}}(A(\Lambda)\in \cdot) $ is given by $\lim_{h\to\infty} \mu_{\dreg}^1 \vert_\Lambda$. 

    Let $\nu$ be an arbitrary infinite-volume Gibbs measure for the random-cluster model. It suffices to show that for any finite edge subset $\Lambda \subseteq E(\mathcal{T}_\infty)$, the marginals coincide, that is, $
    \nu\vert_\Lambda = \mu^\mathrm{wired}\vert_\Lambda$. 

    For a fixed finite edge subset $\Lambda \subseteq E(\mathcal{T}_\infty)$, consider a sufficiently large $h$ such that $\Lambda \subset E(\dreg)$ with $\dist(\Lambda, \partial{\dreg}) > \gamma$ for some constant $\gamma > 0$. By the DLR condition (Definition~\ref{def:dlr}), the restriction of $\nu$ to $\Lambda$ is obtained by averaging over the boundary conditions on $\partial \dreg$  induced by the configuration outside $\dreg$. Let $\sigma$ be a configuration from $\nu$, and let $\xi$ be a boundary condition on $\partial \dreg$ induced by $\sigma$, where all infinite components are considered wired together. Then, we have
    \[
        \nu \vert_\Lambda = \E _{\sigma \sim \nu} [\mu^\xi_{\dreg} \vert_\Lambda].
    \]

    Since $\nu$ is a DLR measure, by~\eqref{eq:Glauber-update-rule}, $\sigma$ must stochastically dominate Bernoulli($\hat p$) bond percolation. Since $p>\ps$ implies $\hat p>1/d$, and all infinite components of $\sigma$ are considered wired together in $\xi$, the random boundary condition $\xi$ follows a $\theta$-wired distribution for some $\theta > 0$.

    Let $r_h = \dist(\Lambda, \partial \dreg)$. By Corollary \ref{cor: edge WSM}, for an edge $e$ in the induced subgraph by $\Lambda$, with probability at least $1 - \exp(-c d^{\sqrt{r_h}})$ for some constant $c > 0$, 
    \[
        \big|\mu^1_{\dreg}(e\in A) - \mu^{\xi}_{\dreg}(e\in A)\big| \leq C \beta^{r_h} 
    \]
    for some constants $C > 0$ and $\beta \in (0, 1)$. Since $\Lambda$ is finite, applying a union bound over the edges in the induced subgraph by $\Lambda$ shows that the above display holds simultaneously over all edges in the induced subgraph, with probability at least 
    $1 - |\Lambda|\exp(-c d^{\sqrt{r_h}})$. Furthermore, this only changes the constant factor in the bound to $C_\Lambda > 0$. Therefore, with this probability, we have 
    \[
        \|\mu^{\xi}_{\dreg}\vert_\Lambda - \mu^{1}_{\dreg}\vert_\Lambda\|_{\textsc{tv}} \leq C_\Lambda \beta^{r_h} 
    \] for some constant $C_\Lambda > 0$ and $\beta < 1$. Taking the expectation over $\sigma \sim \nu$ and using that $\E[X]\le \E[|X|]$, we have
    \[
        \|\nu\vert_\Lambda - \mu^1_{\dreg}\vert_\Lambda\|_{\textsc{tv}} \leq \E_{\sigma \sim \nu} 
        \left[\|\mu^{\xi}_{\dreg}\vert_\Lambda - \mu^{1}_{\dreg}\vert_\Lambda\|_{\textsc{tv}}\right] \leq C_\Lambda \beta^{r_h} + |\Lambda| \exp(-c d^{\sqrt{r_h}}). 
    \] 
    The right-hand side goes to 0 as $h \to \infty$.
    Furthermore, we have
    \[
        \|\mu_{\dreg}^1\vert_\Lambda - \mu^\mathrm{wired}\vert_\Lambda\|_{\textsc{tv}} \to 0
    \] 
    from the convergence in the finite-dimensional distribution sense. 
    By the triangle inequality,
    \[
        \|\nu\vert_\Lambda - \mu^\mathrm{wired}\vert_\Lambda\|_{\textsc{tv}} \leq \lim_{h\to\infty} \left(\|\nu\vert_\Lambda-\mu_{\dreg}^1 \vert_\Lambda\|_{\textsc{tv}} + \|\mu_{\dreg}^1 \vert_\Lambda-\mu^\mathrm{wired}\vert_\Lambda\|_{\textsc{tv}}\right) \to 0,
    \] as $h \to \infty$.     This holds for any finite edge subset $\Lambda$; therefore, $\nu = \mu^\mathrm{wired}$.
\end{proof}

\section{Mixing time when $p>\ps$ for trees}
\label{sec: fast mixing}

Our main aim in this section is to establish Theorem~\ref{thm:intro-wired-tree-mixing-time}. 
We write $\dary$ for the $d$-ary tree of height $h$. For a vertex $v \in \dary$, we denote by $\mathcal{T}_{v,h(v)} \subseteq \dary$ the subtree of height $h(v)$ rooted at $v$ consisting only of the descendants of $v$ in $\dary$.
In particular, $\mathcal{T}_{v,h(v)}$ contains only the portion of $\dary$ lying below $v$. In general, $\mathcal{T}_{v,h}$ is different from the truncated tree $\mathcal{T}_k^v$.

We bound the mixing time of the random-cluster dynamics 
on the $n$-vertex $d$-ary complete tree by
$n \cdot (\log n)^{O(\log\log n)}$ under ``typical'' $\theta$-wired boundary conditions.
To formalize this notion, for fixed constants $C>0$ and $\beta\in(0,1)$, let $\tilde{\mathcal{F}}(h)$ be the class of single-component boundary conditions satisfying the properties that
for any $v\in V(\mathcal{T}_h)$ with $h(v):=\dist(v, \partial\mathcal{T}_h) > h_{\min} :=  (\log_d (h^{1.1}))^2$,
\[
  \max_{\xi\in \tilde{\mathcal{F}}(h)}\left|\mu^1_{\mathcal{T}_{v,h(v)}}(v \sim  \partial \mathcal{T}_{v,h(v)}) - \mu^{\xi}_{\mathcal{T}_{v,h(v)}}(v \sim \mathcal C_1(\xi)\cap \partial \mathcal{T}_{v,h(v)}) \right|
     \leq C \beta ^{h(v)},
\]
and 
\[
 \max_{\xi\in \tilde{\mathcal{F}}(h)}\left|\mu^{1,\circlearrowleft}_{\mathcal{T}_{v,h(v)}}(v \sim  \partial \mathcal{T}_{v,h(v)}) - \mu^{\xi, \circlearrowleft}_{\mathcal{T}_{v,h(v)}}(v \sim \mathcal C_1(\xi)\cap \partial \mathcal{T}_{v,h(v)})\right|
     \leq C \beta ^{h(v)}.
\]
Let $\theta$ be a sufficiently small constant depending on $d,p,q$. Throughout this section, we assume  $C$ and $\beta$ are the constants given by Theorem~\ref{thm wsm} and Lemma~\ref{lemma: root decay} at $\theta$, taking the maximum of the two, and we choose $h_{\min}$ to be $(\log_d (h^{1.1}))^2$ so that these properties can apply to every $v$ with $h(v)>h_{\min}$ via a union bound. 

The following theorem provides a bound for
the mixing time of the random-cluster dynamics on 
the $d$-ary tree with boundary condition from $\tilde {\mathcal F}(h)$.

\begin{theorem}\label{thm:tree-mixing}
    Fix $d \ge 2$, $q > 1$ and $p>\ps$. The random-cluster dynamics on the $d$-ary tree on $n$ vertices, with any boundary condition $\tilde{\mathcal{F}}(h)$, has mixing time $n\cdot (\log n)^{O(\log \log n)}$. 
\end{theorem}

Theorem~\ref{thm:intro-wired-tree-mixing-time} follows immediately from Theorem~\ref{thm:tree-mixing} because the all-wired boundary condition $\xi \equiv 1$ is trivially in $\tilde {\mathcal F} (h)$.
In addition, since a boundary condition on $\mathcal{T}_h$ generated by a $\theta$-wired distribution falls inside $\tilde{\mathcal{F}}(h)$ with high probability, the same upper bound extends to the mixing time under these boundary conditions.
\begin{corollary}
    \label{cor:theta-wired-bc}
    Fix $d \ge 2$,  $q > 1$ and $p > \ps$. Let $\mathsf{M}$ be a $\theta$-wired distribution over single-component boundary conditions on $\partial \mathcal{T}_h$, and let $\xi \sim \mathsf{M}$. 
    Then $\xi \in \tilde{\mathcal{F}}(h)$  
    with probability $1-\exp(-\Omega(h^{1.1}))$.
\end{corollary}
\begin{proof}
Fix $h$ sufficiently large. 
Let $C,\gamma>0$, $\beta\in(0,1)$, and $c>0$ be the constants from Theorem~\ref{thm wsm}.
For any vertex $v$ with $h(v)=\dist(v,\partial\mathcal{T}_h)\ge \gamma $, Theorem~\ref{thm wsm} implies that, with probability at least $1-\exp(-cd^{\sqrt{h(v)}})$ over $\xi\sim \mathsf{M}$,
\[
\left|\mu^1_{\mathcal{T}_{v,h(v)}}(v \sim  \partial \mathcal{T}_{v,h(v)}) - \mu^{\xi}_{\mathcal{T}_{v,h(v)}}(v \sim \mathcal C_1(\xi)\cap \partial \mathcal{T}_{v,h(v)})\right|
     \leq C \beta ^{h(v)}.
\]
The same bound holds for the measures $\mu^{1,\circlearrowleft}_{\mathcal{T}_{v,h(v)}}$ and $\mu^{\xi,\circlearrowleft}_{\mathcal{T}_{v,h(v)}}$ by Lemma~\ref{lemma: root decay}, with the same probability.
There are $d^{h- h(v)}$ such vertices in the same level.
Note that if $h(v)\ge h_{\min}$, then $cd^{\sqrt{h(v)}} \ge ch^{1.1} $ which is at least $2h\log d $ when $h$ is sufficiently large.
Taking a union bound over all vertices of height $h(v) \ge h_{\min} \ge \gamma$ and both estimates gives
\begin{align*}
\pr\left(\xi\in\tilde{\mathcal{F}}(h)\right) &\ge 1 - \sum_{x=\lceil h_{\min} \rceil}^h  2d^{h-x}\exp(-cd^{\sqrt{x}})
= 1 - \sum_{x=\lceil h_{\min} \rceil}^h 2\exp\left(\log d \cdot (h-x) - cd^{\sqrt{x}}\right) \\
&\ge 1 -  \sum_{x=\lceil h_{\min} \rceil}^h 2\exp \left( - \frac{c}{2}d^{\sqrt{x}}\right) 
\ge  1 -  2h \exp\left( -\frac{c}{2} h^{1.1}\right)
\end{align*}
which is at least $1-\exp(-\Omega(h^{1.1}))$ for all $h$ sufficiently large.
\end{proof}

\subsection{Mixing Preliminaries}
Let \(P\) be a finite discrete-time ergodic reversible Markov chain on a state space \(\Omega\) with stationary distribution \(\pi\). For \(\varepsilon \in (0,1)\), we write \(T_{\mix}(P,\varepsilon)\) for the mixing time defined in \eqref{eq:tmix}, with \(1/4\) replaced by \(\varepsilon\).
It is a standard result that for any \(\varepsilon \le 1/4\),
\begin{equation}
\label{eq: tmix extend}
T_{\mix}(P,\varepsilon)\le \left\lceil \log_2 \varepsilon^{-1}\right\rceil T_{\mix}(P).
\end{equation}

In this section, we present a set of general comparison lemmas relating the block random-cluster dynamics to the random-cluster dynamics on a graph $G = (V, E)$.  
To define a general block dynamics, suppose $\mathcal{B} := \{B_i\}_{i=0}^{k-1}$ is a collection of $k$ subsets of edges, which together cover $E$.  
Let $P_{\mathcal{B}}$ denote the block dynamics that, at each step, selects a block uniformly at random from $B\sim \mathcal{B}$ and resamples the entire configuration on the block $B$, conditionally on the configuration on $E\setminus B$. 
Sometimes it is convenient to consider a related time-inhomogeneous Markov chain $\hat{P}_{\mathcal{B}}$ called the \emph{systematic scan block dynamics} in which the blocks are updated in a sequential and deterministic order, specifically, at step $t\ge 1$, the chain updates $B_{t \mod k}$ according to the Gibbs measure.

When analyzing a specific block $B_i$, it is useful to distinguish the boundary conditions localized to $B_i$ from the global configuration outside of $E$.  The global boundary condition on $G$, denoted by $\xi$, is assumed to be given and remain fixed throughout this section.
We let $\Omega({B_i^-})$ denote the set of configurations on $E \setminus B_i$ (each when appended to $\xi$ induces a boundary condition on $B_i$); these local boundary conditions vary during the execution of the block dynamics.  
For each $B_i$, we consider a set $\Gamma_i\subseteq \Omega({B_i^-})$ of ``good'' induced boundary conditions.
For given small $\varepsilon\in (0,1/4]$ and big constant $N>0$, define
\begin{equation}
    \label{eq: Bi mixing time}
    T^*(\varepsilon, N):=
    \left\lceil
    \max_{i=0,\dots, k-1}
\max_{\eta \in \Gamma_i} \frac{|E|}{|B_i|}
\max \left\{ 3T_{\mix}(P_{B_i}^{\xi\sqcup  \eta}, \varepsilon), N \right\} 
    \right\rceil
,
\end{equation}
and 
\[
\delta(N) := \exp\left(-2N/9\right).
\]
(The constant $N$ is just in place of proving a lower bound on the mixing time, which should hold generically anyways.) 
We first present a comparison lemma  between the (single-edge) random-cluster dynamics and the systematic scan block dynamics. 
Our comparison framework based on the collection of sets $\{\Gamma_i\}$ is largely inspired by Proposition 4.10 of \cite{GL19}, but in discrete time instead of continuous-time. Consequently, the estimate in \eqref{eq: Bi mixing time} incurs an additional proportionality factor to account for discrepancies in block sizes and to overcome for the fluctuations in the number of updates experienced by each block, we adopt a slightly stronger assumption than that in \cite{GL19}.
\begin{lemma}
    \label{lem: simulate systematic dynamics}
    Suppose $\{X_t\}_{t\ge 0}$ is a systematic random-cluster dynamics with blocks $\{B_i\}_{i=0}^{k-1}$ starting at an initial state $X_0= \omega_0$.
    Let $\Gamma_0 \subseteq \Omega(B_0^-), \dots, \Gamma_{k-1} \subseteq \Omega(B_{k-1}^-)$ be $k$ sets of good boundary conditions, and we define
    \[
    \rho:=\max_{\omega_0 } \max_{t\ge 0} \pr(X_t(E\setminus B_{(t+1\mod k)})  \sqcup \xi \notin \Gamma_{(t+1 \mod k)} \mid X_0=\omega_0).
    \]
    For any $\varepsilon \in (0,1/4]$, if there exist $\varepsilon'\in(0,1/4]$ and $N>0$ such that 
    \begin{equation} 
        \label{eq: 3 in 1 condition}
        3 T_{\mix}(\hat{P}_{\mathcal{B}}^\xi, \varepsilon') \cdot  (\varepsilon' + \delta(N)+ \rho) \le \frac{\varepsilon}{2|E|}\,,
    \end{equation}
    then  we have $T_{\mix}(P^{\xi}, \varepsilon) \le 2 T^*(\varepsilon', N) T_{\mix}(\hat{P}_{\mathcal{B}}^\xi, \varepsilon')$.
\end{lemma}

The condition~\eqref{eq: 3 in 1 condition} is to ensure that bad boundary conditions are never encountered throughout the entirety of the systematic random-cluster dynamics process (except with probability at most $\eps$). 

The proof of Lemma~\ref{lem: simulate systematic dynamics} is deferred to the Appendix~\ref{sec: simulating BD}.
In the special case where the blocks \(\{B_i\}_{i=1}^k\) are disjoint and independent, the following standard result provides a sharper upper bound on the mixing time of the random-cluster dynamics than that obtained from the general block-dynamics comparison.
\begin{lemma}
    \label{lem: simulating BD 3}
    If $\{B_i\}_{i=1}^k$ are disjoint and independent blocks (i.e., the configuration on one does not change the distribution on the other), then we have
      \[
    T_{\mix}(P, \varepsilon) \le  \max_{i=0,\dots, k-1} 
     \max_{\eta\in \Omega({B_i^-})}\left[ T_\mix \big(P^{\xi\sqcup \eta}_{B_i}, \frac{\varepsilon}{3k} \big) + \log(\frac{2k}{\epsilon}) \right]\cdot \frac{3|E|}{|B_i|}.
    \]
\end{lemma}
The proof of Lemma~\ref{lem: simulating BD 3} is straightforward: since the distribution is a product measure, we apply a Chernoff bound to control the number of updates inside each block and then take a union bound over the total-variation distances of the $k$ blocks.

\subsection{Fast mixing on sufficiently wired trees}

This section will develop a recursive approach to bound the mixing time of the random-cluster dynamics on trees with boundary conditions in $\tilde {\mathcal F}(h)$ that will enable us prove Theorem~\ref{thm:tree-mixing}.
For the rest of this section, we fix $h^*>0$ sufficiently large and we set $r^*:=(1.1\log_d (h^*))^2 $.
For any $h\in (r^*, h^*]$, let $\mathcal{F}(h)$ 
denote the collection of boundary conditions such that $\xi\in \mathcal{F}(h)$ if for all $v\in V(\mathcal{T}_h)$ with $h(v):= \dist(v,\partial\mathcal{T}_h)>r^*$,
\[
    \left|\mu^1_{\mathcal{T}_{v,h(v)}}(v \sim  \partial \mathcal{T}_{v,h(v)}) - \mu^{\xi}_{\mathcal{T}_{v,h(v)}}(v \sim \mathcal C_1(\xi)\cap \partial \mathcal{T}_{v,h(v)})\right|
     \leq C \beta ^{h(v)},
\]
and 
\[
    \left|\mu^{1,\circlearrowleft}_{\mathcal{T}_{v,h(v)}}(v \sim  \partial \mathcal{T}_{v,h(v)}) - \mu^{\xi, \circlearrowleft}_{\mathcal{T}_{v,h(v)}}(v \sim \mathcal C_1(\xi)\cap \partial \mathcal{T}_{v,h(v)})\right|
     \leq C \beta ^{h(v)}.
\]
We first note that $\mathcal{F}(h^*)=\tilde{\mathcal{F}}(h^*)$ and for $h\in(r^*, h^*]$, $\tilde{\mathcal{F}}(h)\subseteq \mathcal{F}(h)$.
Now fix $h\in (r^*,h^*]$ and $\xi\in \mathcal{F}(h)$. For each vertex $v$ at depth $h-h(v)$ in $\mathcal{T}_h$, with $r^*<h(v)\le h^*$, let $\xi_{(v)}$ be the boundary condition on $\partial \mathcal{T}_{v,h(v)}$ induced by $\xi$.
Because $\xi$ is single-component, the induced boundary condition $\xi_{(v)}$ is also single-component. In particular, $\xi_{(v)}\in \mathcal{F}(h(v))$.

Since our Markov chain operates on the set of edges, we use $|\mathcal{T}_h|$ to denote the number of edges in $\mathcal{T}_h$. 
(As $\mathcal{T}_h$ is a tree, this is exactly one less than its number of vertices.)
Also, for any $U\subseteq V(\mathcal{T}_h)$, we use $|U|$ to denote the number of edges in $U$. 
Moreover, given a graph $H$ and a boundary condition $\xi$ on $\Lambda\subseteq V(H)$, we write $P_H^\xi$ for the random-cluster dynamics on $E(H)$ with its boundary condition on  $\Lambda$ fixed to be $\xi$, 
and $T_\mix(P_H^\xi)$ for the mixing time of the random-cluster dynamics on $U$ under $\xi$.
Let
\[
\tau_\mix(h) =
\max_{\xi  \in \mathcal{F}(h)} \max\{T_\mix(P^\xi_{\mathcal{T}_{h}}),  T_\mix(P^{\xi, \circlearrowleft}_{\mathcal{T}_{h}}
)\};
\]
recall that $\circlearrowleft$ correspond to the case where the root of $\mathcal{T}_{h}$ is wired to $\mathcal C_1(\xi)$.
When $h$ is clear from context, we use $P^\xi$ for $P_{\mathcal T_h}^\xi$. 

We will prove Theorem~\ref{thm:tree-mixing} via a recursion of $\Th(h)$ on the height $h$. To formulate the recursion, we decompose a tree of height $h$ as follows. 
For a fixed  $\delta \in (0,1/4)$ that we later choose to be sufficiently small, and for any sufficiently large integer $h$, one could find three integers $h_0, h_1$ and $h_2$  such that
$h_0 + h_1 + h_2 = h$ and  $ h_0 = h_2 = \delta h$. 
In the decomposition above, to avoid unnecessary notational burden, we are suppressing rounding issues for non-integer heights, as the proof is unaffected by the choice of floors or ceiling implicitly made.

For an integer $h$ with its decomposition $(h_0, h_1, h_2)$, we show the following recursion. 

\begin{lemma}
    \label{lem: mixing recursion}
    If $p>\ps$,
    then for any $h\in (2r^*, h^*]$ sufficiently large with its decomposition $(h_0, h_1, h_2)$, there exists a constant $M>0$ such that
    the random-cluster dynamics satisfies 
    \begin{equation}
        \label{eq: mixing recursion}
        \Th(h)\le M\log^{2} |\mathcal{T}_h|\cdot \max\left\{\Th({h_0+h_1}) \cdot \frac{|\mathcal{T}_h|}{|\mathcal{T}_{h_0+h_1}|}, \quad d^{h_0} \cdot h^4\cdot \Th({h_1+h_2})\right\}.
    \end{equation}
\end{lemma}

Before proving Lemma~\ref{lem: mixing recursion}, 
we first show how  
Theorem~\ref{thm:tree-mixing} follows from 
Lemma~\ref{lem: mixing recursion}. 
\begin{proof}[Proof of Theorem~\ref{thm:tree-mixing}]
    A $d$-ary tree with $n$ vertices has height $h = O(\log_{d}n)$.  
By a standard canonical paths argument (see Lemma 17 in~\cite{BG24-PTRF} that applies to arbitrary single-component boundary condition), one obtains the bound
$\Th(h) = \exp(O(h))$.
We assume that $h^*$ is sufficiently large such that Lemma~\ref{lem: mixing recursion} can apply to any $h$ that is at least $3r^* = \Omega((\log h^*)^2)$; for smaller values of $h^*$, the theorem follows from the standard bound aforementioned by uniformly adjusting the constant hidden in the big-$O$ notation. 

Define a sequence $(h^{(i)})_{i\ge 1}$ by 
$h^{(1)}=\delta h $ and 
$h^{(k+1)}= \delta (h- \sum_{i=1}^k h^{(i)})$ for all $k>1$.
Note that $h^{(k+1)} =\delta (1-\delta)^k h$, which converges to $0$ and moreover $\sum_{k\ge 1}h^{(k)}\le h$ follows from the construction.
In iteration $k$, we choose $h_0= h_2= h^{(k+1)}$ and $h_1 = (1-2\delta)(h- \sum_{i=1}^k h^{(i)})$.
By iterating \eqref{eq: mixing recursion} from Lemma~\ref{lem: mixing recursion} for $s\ge1$ steps, we obtain that
\begin{align}
    \Thh &\le \Th\Big({h- \sum_{i=1}^s h^{(i)}}\Big) \cdot (M (h\log d)^2)^{s}\cdot 
    \prod_{j=1}^s\Big[ d^{h^{(j)}}\cdot \big(h- \sum_{i=1}^{j-1} h^{(i)}\big)^{4} \Big] \nonumber\\
    &=\Th\Big({h- \sum_{i=1}^s h^{(i)}}\Big)\cdot (M (h\log d)^2)^{s}\cdot d^{\sum_{j=1}^s h^{(j)}} \cdot \prod_{j=1}^s \Big[  \big(h- \sum_{i=1}^{j-1} h^{(i)}\big)^{4} \Big] \nonumber \\
    &\le \Th\Big({h- \sum_{i=1}^s h^{(i)}}\Big)\cdot d^{\sum_{j=1}^s h^{(j)}}
    \cdot (M (\log d)^2 h^6)^s.\label{eq: s recursion mixing}
\end{align}
Since $(1-\delta) (h-\sum_{i=1}^k h^{(i)}) =  h- \sum_{i=1}^{(k+1)} h^{(i)} $,
there exists $s=M_2 \log  h$ with a suitably chosen $M_2=M_2(\delta, d)$ such that
\[
3r^* \le  h- \sum_{i=1}^s h^{(i)} \le 24r^*.
\]
Then for each term in \eqref{eq: s recursion mixing} we have an upper bound: 
\begin{enumerate}
    \item $ \Th\bigl({h- \sum_{i=1}^s h^{(i)}}\bigr) \le \max_{h'\in [3r^*,24r^*]} \Th(h') 
    =\exp\bigl(O((\log \log n)^2)\bigr) = (\log n)^{O(\log \log n)}$;
    \item $d^{\sum_{j=1}^s h^{(j)}} \le   d^{h} = O(n)$;
    \item $(M (\log d)^2 h^6)^s = (M (\log d)^2 h^6)^{M_2 \log \log n} \le  (M (\log d)^2 \log  n)^{6 M_2 \log \log n}$.
\end{enumerate}
The theorem therefore follows from \eqref{eq: s recursion mixing} and these bounds.
\end{proof}

\subsection{Mixing time recursion on trees}

We proceed with the proof of Lemma~\ref{lem: mixing recursion}.  
For this, we 
consider the random-cluster dynamics restricted to two overlapping blocks. 
Let $B_0 = \mathcal{T}_{r, h_0+h_1}$ from the root $r$, and $$B_1 = (V(\mathcal{T}_h \setminus \mathcal{T}_{r, h_0-1}), E(\mathcal{T}_h)  \setminus E(\mathcal{T}_{r, h_0})).$$  
With a slight abuse of notation we shall use $B_0$ and $B_1$
for the edge sets $E(B_0)$ and $E(B_1)$.
Define $\Omega_{B_0\setminus B_1}$ as the set of all configurations on $E(B_0) \setminus E(B_1)$; a configuration $\eta\in \Omega_{B_0\setminus B_1}$,  together with the boundary condition $\xi$ induce a boundary condition denoted $\xi \sqcup \eta$ on $B_1$.
We derive the following recursion for the mixing time 
of the random-cluster dynamics on $\mathcal T_h$
in terms of the mixing times of the same Markov chain on $B_0$ and $B_1$. This will be proved in Section~\ref{sec: mixing time BD}. 
\begin{lemma}
    \label{lem: block recursion}
    Let $p>\ps$.
    There exists a constant $M_1>0$ such that
      for any $h\in (2r^*, h^*]$ sufficiently large with its decomposition $(h_0, h_1, h_2)$, 
    the random-cluster dynamics satisfies 
    \begin{equation}
        \label{eq: block recursion}
        \Thh \le 
        M_1 \log^{2} |\mathcal{T}_h|  
         \max \Bigg\{ \max_{\xi \in \mathcal{F}( h_0+h_1)}  \max\{T_\mix(P^{\xi}_{B_0}),  T_\mix(P^{\xi, \circlearrowleft}_{B_0})\}  \frac{|\mathcal{T}_h|}{|B_0|},  \max_{\substack{\xi\in \mathcal{F}(h) \\\eta\in \Omega_{B_0\setminus B_1}}} T_\mix(P_{B_1}^{\xi\sqcup  \eta})  \frac{|\mathcal{T}_h|}{|B_1|}\Bigg\}.
    \end{equation}
\end{lemma}
Note that $B_0$ is itself a tree of smaller height with single-component boundary conditions, so the first $\max$ in the braces on the right-hand side of \eqref{eq: block recursion} is simply $\Th({h_0+h_1})$.  
The second term, however, requires a more delicate treatment,
as it correspond to a forest with some of its tree roots 
connected by the configuration  on $E(B_0) \setminus E(B_1)$.
We make this more precise with the following notation.

\begin{definition}
    For fixed $\eta \in \Omega_{B_0 \setminus B_1}$, the block $B_1$ decomposes into at most $|\partial \mathcal{T}_{h_0}|$ 
disjoint connected components where consider two subtrees as being part of the same component if their roots are connected by $\eta$.  
Let $\mathcal{C}(\eta)$ denote the components $B_1$ induced by $\eta$.   
Each $\kappa \in \mathcal{C}(\eta)$ consists of $d_\kappa$ trees of height $h_1 + h_2$ rooted at $r_1, \dots, r_{d_\kappa}$ with their roots wired together and with their leaves wired according to the induced boundary conditions $\xi_{(r_1)},\dots,\xi_{r_{d_\kappa}} \in \mathcal{F}(h_1+h_2)$.  \end{definition}

Instead of explicitly writing all boundary conditions $\xi_{(r_j)}$, we will simply use 
$P_{\kappa}^{\xi}$ to denote the random-cluster dynamics on the component $\kappa$ under the combined boundary condition 
$\xi_{(r_1)} \sqcup  \dots \sqcup  \xi_{(r_{d_\kappa})}$.  
Among the components in $\mathcal{C}(\eta)$, there may be at most one component whose roots are additionally wired to $\mathcal{C}_1(\xi)$; for such a component we write $P_{\kappa}^{\xi,\circlearrowleft}$.
In Section~\ref{sec: wired components}, 
we prove the following bound for the mixing time of the random-cluster dynamics on any such component.

\begin{lemma}
\label{lem: 3rd bound kappa}
Let $p>\ps$.
For every $\eta \in \Omega_{B_0\setminus B_1}$,  for every component $\kappa \in \mathcal{C}(\eta)$ and every boundary condition $\xi \in \mathcal{F}(h)$,
\[
\max\{T_{\mix}(P_{\kappa}^{\xi}), T_{\mix}(P_{\kappa}^{\xi, \circlearrowleft})\}
    \;\le\;
    O\!\left(d_\kappa \log 2d_\kappa\right) \cdot O\!\left((h_1 + h_2)^{2}\right)
    \cdot \Th({h_1 +h_2}).
\]
\end{lemma}

With all these ingredients in place, we now proceed to prove Lemma~\ref{lem: mixing recursion}. For brevity, in what follows we let $h_3 := h_1 + h_2$. 

\begin{proof}[Proof of Lemma~\ref{lem: mixing recursion}]
By applying Lemma~\ref{lem: block recursion}, we obtain that
\begin{equation}
      \Thh \le M_1 \log^{2} |\mathcal{T}_h| \cdot \max \Bigg\{
      \Th({h_0 + h_1})\cdot \frac{|\mathcal{T}_h|}{|B_0|},
     \max_{\substack{ \xi\in \mathcal{F}(h) \\ 
     \eta\in \Omega_{B_0\setminus B_1}} 
     }
     T_\mix(P_{B_1}^{\xi\sqcup  \eta})\cdot  \frac{|\mathcal{T}_h|}{|B_1|}
    \Bigg\}  \label{eq: two block decomposition}
\end{equation}
 Since the first terms in \eqref{eq: two block decomposition} and \eqref{eq: mixing recursion} coincide, it remains to bound 
$T_\mix(P_{B_1}^{\xi\sqcup  \eta})$ for any $\xi \in \mathcal{F}(h) $ and $ \eta\in \Omega_{B_0\setminus B_1}$.
 Moreover, since dynamics on disjoint components of $\mathcal C(\eta)$ are independent, by Lemma~\ref{lem: simulating BD 3} and \eqref{eq: tmix extend}, we have 
    \begin{equation}
        \label{eq: component disjointness}
         T_\mix(P_{B_1}^{\xi\sqcup  \eta}) = \max_{\kappa \in \mathcal{C}(\eta)}  
         \frac{|B_1|}{|\kappa|}\max\{T_{\mix}(P_{\kappa}^{\xi}), T_{\mix}(P_{\kappa}^{\xi, \circlearrowleft})\} \cdot O(\log{|\mathcal{C}(\eta)|})\,.
    \end{equation}
Then, by Lemma~\ref{lem: 3rd bound kappa} 
\begin{align*}
     T_\mix(P_{B_1}^{\xi\sqcup  \eta}) &= \max_{\kappa \in \mathcal{C}(\eta)}  
         \frac{O( d_\kappa \log 2d_\kappa) \cdot O(h^2) \cdot \Th({h_3})}{|\kappa|/|B_1|} \cdot O(\log{|\mathcal{C}(\eta)|}) \\
         &= \max_{\kappa \in \mathcal{C}(\eta)}  
         \left\{
         O( d_\kappa \log 2d_\kappa) \cdot O(h^2) \cdot \Th({h_3}) \cdot  \frac{|\partial \mathcal{T}_{h_0}| \cdot |\mathcal{T}_{h_3}| }{d_\kappa \cdot |\mathcal{T}_{h_3}|} \cdot O(\log{|\partial \mathcal{T}_{h_0}|})
         \right\} \\
         &= \max_{\kappa \in \mathcal{C}(\eta)}  
         O\left( h^2 \log (2d_\kappa) \right)
      \cdot \Th({h_3}) \cdot O(|\partial \mathcal{T}_{h_0}|\log{|\partial \mathcal{T}_{h_0}|})
\end{align*}
    Then, since $|\partial \mathcal{T}_{h_0}| \le O(d^{h\delta})$ we have
\begin{align*}
    T_\mix(P_{B_1}^{\xi\sqcup  \eta}) 
      &= h^2 \max_{\kappa \in \mathcal{C}(\eta)}  
         O\left( \log (2d_\kappa) \right)
      \cdot \Th({h_3}) \cdot O(d^{h\delta} \cdot h\delta \log d) \\
      &=O(h^{4})\cdot d^{\delta h} \cdot  \Th({h_3}).
\end{align*}
Since $|\mathcal{T}_h|/|B_1|=O(1)$,
combining this estimate with~\eqref{eq: two block decomposition} completes the proof of the lemma.
\end{proof}

\subsection{Reduction via block dynamics: Proof of Lemma~\ref{lem: block recursion}}
\label{sec: mixing time BD}
In this subsection, we prove Lemma~\ref{lem: block recursion} by analyzing an auxiliary block dynamics.  
To set up this dynamics, we recall that each  \(h > 0\) is associated with three parameters \(h_0, h_1,\) and \(h_2\) satisfying 
$
h_0 + h_1 + h_2 = h$ and $h_0 = h_2 =  \delta h$.
We define \(P_{\textsc{bd}}\) to be the \emph{systematic scan block dynamics} that at time $t$ performs a heat-bath update (i.e., resamples the configuration conditional on the configuration in the complement) on the block $U_t:=B_{t\!\!\mod \!\!~2}$, where recall $B_0 = \mathcal{T}_{h_0 + h_1}$ and
\[
B_1 
    = \bigcup_{v \in \partial \mathcal{T}_{h_0}} \mathcal{T}_{v,\,h_1 + h_2}\,.
\]

\begin{lemma}
    If $p>\ps$,
    for any $h\in (2r^*, h^*]$ and any $\xi \in \mathcal{F}(h)$,
    there exists an $M_0=M_0(d,p,q)$ that is independent of $h$ such that $T_{\mix}(P_{\textsc{bd}}^{\xi})\le M_0$ and $T_{\mix}(P_{\textsc{bd}}^{\xi, \circlearrowleft})\le M_0$.
    \label{lem: coupling time 2 block dynamics}
\end{lemma}
\begin{proof}
We analyze the mixing time of the systematic scan block dynamics via coupling. 
Formally, for a fixed boundary condition $\xi$, we couple two instances $\{X_t^+\}_{t \ge 0}$ and $\{X_t^-\}_{t \ge 0}$ of $P_{\textsc{bd}}^{\xi}$ where the former starts from the all-in configuration and the latter from the all-out configuration.  
Since the random-cluster model is a monotone system, 
under the coupling in which both chains update the same block $U_t$ with the optimal (monotone) coupling, we have 
\[
T_{\mix}(P_{\textsc{bd}}^{\xi}) \le \tau_{\textsc{cp}}(P_{\textsc{bd}}^{\xi}) := \min\big\{t>0 : \pr(X_t^+ = X_t^-) \ge 3/4 \big\}.
\]
Thus, we will establish an upper bound on $\tau_{\textsc{cp}}(P_{\textsc{bd}}^{\xi})$.
The proof would be general enough to extend to obtain an upper bound of $\tau_{\textsc{cp}}(P_{\textsc{bd}}^{\xi, \circlearrowleft})$ as well.
Moreover, we may assume that $h$ is larger than some sufficiently large constant $H>0$, since for every fixed $h \le H$ the coupling time admits a uniform upper bound. 
We now consider the case $h > H$.

Consider the following event $E_t$ at any time $t$: 
\begin{align}
    \label{eq: et}
    E_t&:=\{X_t^+(\mathcal{T}_{r, h} \setminus B_0) \sqcup  \xi \in \mathcal{F}( h_0+h_1), X_t^-(\mathcal{T}_{r, h} \setminus  B_0 )\sqcup  \xi \in \mathcal{F}( h_0+h_1)\}  \\
   & \cap \left\{\eta = X_t^{\pm}(\mathcal{T}_{r, h} \setminus B_0) \sqcup  \xi: \mu^1_{B_0}(e\in A) - \mu^\eta_{B_0}( e\in A) \leq C_2 \beta ^{\dist(e, \partial B_0)} ,\forall  e\in E(\mathcal{T}_{r,h_0}) \right\} \label{eq: et part2},
\end{align}
where $C_2$ and $\beta$ are the constants given in Corollary~\ref{cor: edge WSM}.
Observe that $E_t$ is a monotone property in the full configuration, and thus it is well-defined as a property about configurations only on $E(\mathcal{T}_{r, h} \setminus B_0)$ regardless of the configurations  on $E(B_0) \setminus E(B_1)$.

We claim $E_1$ happens with $1-\exp(-h^{1+\Omega(1)})$ probability. 
To show this, we first  prove that after we update $B_1$ at time $t=1$, 
the distributions over boundary conditions on $\partial \mathcal{T}_{r, h_0+h_1}$ induced by configurations $X_1^+(\mathcal{T}_{r, h} \setminus B_0) \sqcup  \xi$ and $X_1^-(\mathcal{T}_{r, h} \setminus B_0) \sqcup  \xi$ are $\theta$-wired for a fixed $\theta>0$.
Indeed, when we update $B_1$,  each node $v\in \partial \mathcal{T}_{r, h_0+h_1}$ connects to $\mathcal{C}_1(\xi)$ with probability at least $\mu_{\mathcal T_{v,h_2}}^{\xi_{(v)}}\bigl(v\sim \mathcal C_1(\xi_{(v)})\bigr)$.
Since $h_2>r^*$ and  $\xi_{(v)}\in \mathcal{F}(h_2)$, 
\begin{equation}
    \label{eq: use WSM for theta-wired}
    \left| \mu_{\mathcal{T}_{v,h_2}}^{\xi_{(v)}}(v\sim \mathcal{C}_1(\xi_{(v)}) ) - \mu_{\mathcal{T}_{v,h_2}}^{1}(v\sim \partial\mathcal{T}_{v,h_2})\right| \le C\beta^{h_2}.
\end{equation}
Hence, each  $v\in \partial \mathcal{T}_{r, h_0+h_1}$ is connected to $\mathcal{C}_1(\xi_{(v)})$ with probability at least
\[
 \theta_v:=
 \mu_{\mathcal{T}_{v,h_2}}^{1}(v\sim \partial\mathcal{T}_{v,h_2})  -  C\beta^{h_2}.
\]
Since the trees $\{\mathcal{T}_{v,h_2}\}_{v\in \partial \mathcal{T}_{r,h_0+h_1}}$ are all $d$-ary trees of height $h_2$, the quantity $\theta_v$ is the same for every $v\in \partial \mathcal{T}_{r,h_0+h_1}$. Moreover, the probability $\mu_{\mathcal{T}_{v,h_2}}^{1}(v\sim \partial \mathcal{T}_{v,h_2})$ is decreasing in $h_2$, and hence converges as $h_2\to\infty$. For $p>p_s$, this probability is bounded below by the corresponding connection probability in Bernoulli $\hat{p}$-percolation on the infinite $d$-ary tree with parameter $\hat p>1/d$, which is strictly positive. Therefore, for all sufficiently large $h_2$, there exists a constant $\theta>0$, independent of $h_2$, such that $\theta_v\ge \theta$ for every $v\in \partial \mathcal{T}_{r,h_0+h_1}$.
Moreover, since $\mathcal{T}_{r,h}\setminus B_0$ is a forest whose leaves inherit a single wired boundary class, 
an induced boundary condition on $\partial \mathcal{T}_{r,h_0+h_1}$ must also be a single-component boundary condition.
Hence, the induced boundary conditions on $\partial \mathcal{T}_{r,h_0+h_1}$ are $\theta$-wired.

As a consequence, by invoking Corollary~\ref{cor:theta-wired-bc} for $\theta$-wired distributions,
and the fact that $\tilde{\mathcal{F}}(h_0+h_1)\subseteq \mathcal{F}(h_0+h_1)$,
we will obtain that,  regardless of the configurations on $E(B_0) \setminus E(B_1)$, 
the event in \eqref{eq: et} holds with probability $ 1 - 2\exp(-(h_0+h_1)^{1+\Omega(1)})$.
Furthermore, as long as $h_1 >\gamma$, Corollary~\ref{cor: edge WSM} with a union bound over the edges in $E(\mathcal{T}_{r,h_0})$  implies that
the event in \eqref{eq: et part2} holds with probability at least
\[
1-4d^{h_0}\cdot e^{ - c d^{ \sqrt{h_1}}}
\ge 1-4\exp\bigl(-\frac{c}{2} d^{\sqrt{h_1}} \bigr).
\]
Therefore, with a union bound over these two events, we obtain
\begin{align}\label{eq:good-bc-high-probability}
    \pr( E_1) \ge 1 - \exp(-(h_0+h_1)^{1+\Omega(1)})\,.
\end{align}

Assume that $E_{1}$ holds.  
We now show that 
\begin{equation}
    \label{eq: match top}
    \pr\left(X^+_{2}(E(\mathcal{T}_{r,h_0}))= X^-_{2}(E(\mathcal{T}_{r,h_0})) \mid E_{1} \right) \ge 1-e^{-\Omega(h)}.
\end{equation}
Let $\eta_+= X^+_{1}(\mathcal{T}_{r, h} \setminus B_0)\sqcup  \xi$ and $\eta_-= X^-_{1}(\mathcal{T}_{r, h} \setminus B_0)\sqcup  \xi$ be the two boundary conditions of $B_0$ induced by $X^+_{1}(\mathcal{T}_{r, h} \setminus B_0)$ and $X^-_{1}(\mathcal{T}_{r, h} \setminus B_0)$ respectively.
By definition of that set of boundary conditions and a triangle inequality, for each $e\in E(\mathcal{T}_{r,h_0})$ we have 
\begin{align*}
    \pr(X^+_{2}(e)\neq X^-_{2}(e)\mid E_1)&=\left | \mu_{B_0}^{\eta_+}(e\in A) -\mu_{B_0}^{\eta_-}(e\in A)  \right | \\
    &\le \left | \mu_{B_0}^{\eta_+}(e\in A) -\mu_{B_0}^{1}(e\in A)  \right| + \left| \mu_{B_0}^{1}(e\in A) -\mu_{B_0}^{\eta_-}(e\in A)  \right| \\
   & \le 2C_2 \beta^{h_1}.
\end{align*}
A union bound over $E(\mathcal T_{r,h_0})$ gives
\begin{align*}
    \pr\bigl(X_2^+(E(\mathcal T_{r,h_0}))
    \ne
    X_2^-(E(\mathcal T_{r,h_0}))\mid E_1\bigr)
    &\le
    \sum_{e\in E(\mathcal T_{r,h_0})} \pr(X^+_{2}(e)\neq X^-_{2}(e)\mid E_1) \\
    &\le
    (d^{\delta h})\beta^{(1-2\delta)h}.
\end{align*}
This establishes \eqref{eq: match top} for sufficiently large $h$ and sufficiently small $\delta$.

Finally, if 
\(
X_2^+(E(\mathcal T_{r,h_0}))
    =
    X_2^-(E(\mathcal T_{r,h_0})),
\)
then with probability $1$ we have 
\[
X_3^+=X_3^-,
\]
since once the configurations on $E(\mathcal T_{r,h_0})$ agree, an identical coupling updates $B_1$ to be the same configuration in $X_3^+$ and $X_3^-$.
Therefore, we conclude that for $h>H$,
\[
\tau_{\textsc{cp}}(P_{\textsc{bd}}^{\xi}) \le 3
\quad\text{with probability at least } 1 - e^{-\Omega(h)},
\]
and the proof is complete. 
\end{proof}

We now turn to the proof of Lemma~\ref{lem: block recursion}. 
We simulate the previously defined block dynamics using the random-cluster dynamics. 

\begin{proof}[Proof of Lemma~\ref{lem: block recursion}]
Fix $\xi\in \mathcal{F}(h)$.
We prove an upper bound on $T_{\mix}(P^\xi)$ by comparing $T_{\mix}(P^\xi)$ and $T_{\mix}(P_{\textsc{bd}}^\xi)$ the simulation technique; the same bound applies to $T_{\mix}(P^{\xi,\circlearrowleft})$. 
Let $\{X_k\}_{k\ge 0}$ be the systematic scan block dynamics $P_{\textsc{bd}}^\xi$ started from an initial configuration $\omega_0$. Let $\Gamma_1 = \Omega_{B_0\setminus B_1}$ and let $\Gamma_0$ be the following set
\[
\left\{ \omega\in \Omega_{B_1 \setminus B_0}: \omega\sqcup  \xi\in \mathcal{F}(h_0+h_1),  \mu^1_{B_0}(e\in A) - \mu^{\omega\sqcup  \xi}_{B_0}( e\in A) \leq C_2 \beta ^{\dist(e, \partial B_0)} ,\forall  e\in E(\mathcal{T}_{r,h_0})\right\}.
\]  
By definition, $X_k(\mathcal{T}_{r,h}\setminus B_1) \in \Gamma_1$ holds with probability $1$ for all $k$.  
Moreover, by~\eqref{eq:good-bc-high-probability}, we know that when $\omega_0$ is an extreme configuration, for $k=1$, 
\[
\pr\bigl(X_k(\mathcal{T}_{r,h}\setminus B_0)\notin\Gamma_0 \mid X_0 = \omega_0 \bigr)\le|\mathcal{T}_h|^{-1-\Omega(1)},
\]
as $\delta$ is sufficiently small and $h$ is sufficiently large;  the same argument extends to all $k\ge 1$ and to arbitrary initial configurations. 
We choose $N = \Theta(\log |\mathcal{T}_h|)$ sufficiently large so that 
$
\delta(N)=\exp(-2N/9)<(60M_0 |\mathcal{T}_h|)^{-2}
$,
and set $\varepsilon'=(60M_0 |\mathcal{T}_h|)^{-2}$, where $M_0$ is the constant in  Lemma~\ref{lem: coupling time 2 block dynamics}. Define
\[
T(\varepsilon') := \max\left\{
    \max_{\eta\in \Gamma_1} 
    T_{\mix}(P_{B_1}^{\xi\sqcup \eta},\varepsilon')\frac{|\mathcal{T}_h|}{|B_1|},
    \;
    \max_{\xi' \in \Gamma_0} 
    \max\!\left\{
        T_{\mix}(P_{B_0}^{\xi'},\varepsilon'),\;
        T_{\mix}(P_{B_0}^{\xi',\circlearrowleft},\varepsilon')
    \right\}\frac{|\mathcal{T}_h|}{|B_0|}
\right\}
\] and let $T^*(\varepsilon' ,N) = \max \{ T(\varepsilon'), N\}$.
By our choice of $\varepsilon'$ and $N$, we have
$T^*(\varepsilon' ,N) =O(T(\varepsilon'))$.
By Lemma~\ref{lem: coupling time 2 block dynamics} and \eqref{eq: tmix extend}, 
\[
 3 T_{\mix}(P_{\textsc{bd}}^\xi, \varepsilon') \cdot  (\varepsilon' + \delta(N)+ \rho) 
 \le 
4M_0 \cdot O(\ln |\mathcal{T}_h|) \cdot \left(\frac{2}{(60 M_0 |\mathcal{T}_h|)^2} +  \frac{1}{|\mathcal{T}_h|^{1+\Omega(1)}}  \right) < \frac{1}{8|\mathcal{T}_h|}.
\]
Applying \eqref{eq: tmix extend}, Lemmas~\ref{lem: simulate systematic dynamics} and \ref{lem: coupling time 2 block dynamics}  then yields
\[T_{\mix}(P^\xi) \le 2 T_{\mix}(P_{\textsc{bd}}^\xi, \varepsilon') \cdot T^*(\varepsilon' ,N)
=  O(\ln^{2} |\mathcal{T}_h|) \cdot O(T(\varepsilon')).
\]
The desired upper bound for $\Thh$ follows by choosing the constant $M_1$ sufficiently large.
\end{proof}

\subsection{Mixing time of the random-cluster dynamics on connected sets of trees}
\label{sec: wired components}
We now provide the proof of Lemma~\ref{lem: 3rd bound kappa}. 
This lemma concerns the mixing time of random-cluster dynamics on a \emph{wired component} $\kappa$ that consists of 
$d_\kappa$ disjoint $d$-ary trees of height $h_3$, whose leaves are wired by boundary conditions from $\mathcal{F}(h_3)$ and whose roots are also wired together by the boundary condition. 

\begin{proof}[Proof of Lemma~\ref{lem: 3rd bound kappa}]
        Fix $\kappa \in \mathcal{C}(\eta)$ and fix a sufficiently large integer $h_3$. Suppose that $d_\kappa$ $d$-ary trees of height $h_3$ are wired at their roots, with respective leaf boundary conditions drawn from the family $\mathcal{F}( h_3)$.  Recalling that $h_3 = h_1 + h_2$, we will show that 
\begin{equation}
    \label{eq: star recursion 2}
    T_{\mix}(P_{\kappa}^{\xi}) \le O( d_\kappa \log 2d_\kappa) \cdot O(h_3^2) \cdot \Th({h_3}),
\end{equation}
and, in fact, the same upper bound for $T_{\mix}(P_{\kappa}^{\xi, \circlearrowleft})$ follows by an identical (or even simpler) argument. If $d_\kappa=1$, the claim follows directly from the definition of $\Th(h_3)$. Thus we assume throughout that $d_\kappa\ge 2$.

Let $\{X_t\}_{t\ge 0}$ denote the uncensored random-cluster dynamics on $\kappa$ with stationary distribution $\pi_\kappa^\xi$. To bound this, note that under the grand monotone coupling of random-cluster dynamics, for every component $\kappa \in \mathcal C(\eta)$ and every boundary condition $\xi \in \mathcal F(h)$, we have 
\begin{align*}
    \max_{x_0} \| \mathbb P (X_t \in \cdot \mid X_0 = x_0) - \pi_\kappa^\xi\|_{\textsc{tv}} & \le \sum_{e} \big(\mathbb P (e\in X_t \mid X_0 = 1) - \mathbb P(e\in X_t \mid X_0 = 0)\big)\\
    & \le 2|E(\kappa)| \max_{\eta \in \{0,1\}}\|\mathbb P (X_t  \in \cdot \mid X_0 = \eta) - \pi_\kappa^{\xi}\|_{\textsc{tv}}
\end{align*}

Let $\{Y_t^1\}$ denote the random-cluster dynamics $P_{\kappa}^{\xi}$ starting from the all-wired configuration under a censoring scheme that we will shortly specify, and let $\{Y_t^0\}$ denote the same dynamics initialized from the all-free configuration under the same censoring scheme.
We  run an uncensored ``burn-in'' phase of  $\tau_{1}=10\log (4/c_*) d_\kappa \Th({h_3})$ steps, where the constant $c_*=c_*(p,q,d)>0$ will be defined below.  
Then we censor the updates to $\mathcal{T}_{r_1} := \mathcal T_{r_1, h_3}$ for the next $\tau_2:=200 \log |E(\kappa)|  d_\kappa \log 2d_\kappa \cdot \tau_\mix(h_3)$ steps. Finally 
we censor the updates to $\kappa \setminus\mathcal{T}_{r_1}$ for $\tau_3=200 \log |\mathcal{T}_{r_1}| d_\kappa \tau_{\mix}(h_3)$ steps. For every time chunk $\tau:= \tau_1 + \tau_2 + \tau_3$ thenceforth, the same censoring scheme is repeated.  By the censoring inequality of~\cite{PW} (see Lemma~\ref{lem: censoring} for a precise statement), we have
\begin{align*}
    \max_{x_0} \| \mathbb P(X_t \in \cdot \mid X_0 = x_0) - \pi_\kappa^\xi\|_{\textsc{tv}} \le 2d^{h_3} \max_{\eta \in \{0,1\}}  \| \mathbb P (Y_t^\eta  \in \cdot) - \pi_\kappa^\xi\|_{\textsc{tv}} \le 2 d^{h_3}  \mathbb P( Y_t^1   \ne Y_t^0)\,.
\end{align*} where $Y_t^1, Y_t^0$ are coupled via the grand monotone coupling of random-cluster dynamics. 

We will show that in the grand monotone coupling $(Y_t^1, Y_t^0)$ of these two Markov chains, at time $\tau$,  
$$
\pr(Y_\tau^1 \ne  Y_\tau^0)  \le 1- c_*/5
$$
as then, by standard boosting under the grand monotone coupling, we get $\pr(Y_{k\tau}^1 \ne Y_{k\tau}^0) \le (1-c_*/5)^k$, which will be $o(d^{-h})$ after $\tau \cdot O(h_3)$ steps. 

We first show that after time $\tau_1$, there is a constant probability that in $Y_{\tau_1}^0$ (and therefore by monotonicity also $Y_{\tau_1}^1$), $\mathcal T_{r_1}$ has a root-to-boundary connection. This will decouple all the trees $\kappa \setminus \mathcal T_{r_1} = (\mathcal T_{r_i})_{i=2}^{d_\kappa}$ for the time period $[\tau_1,\tau_1+\tau_2]$ where $\mathcal T_{r_1}$ updates are censored.  
By the monotonicity of the random-cluster measure in each tree $\mathcal{T}_{r_j}:=\mathcal{T}_{r_j,h_3}$ the probability that $r_j$ connects to $\mathcal{C}_1(\xi_{(r_j)})$ is lower bounded by the same probability in the tree assuming that the root has not been wired to the boundary component. 
For $p>\ps$, the latter probability is guaranteed to be at least $c_* = c_*(p, q, d)>0$ provided the boundary condition $\xi_{(r_j)}$ on $\partial \mathcal{T}_{r_j}$ is from $\mathcal{F}(h_3)$; this follows from the argument surrounding \eqref{eq: use WSM for theta-wired}.

By Chernoff's inequality, by the time $\tau_1$,  there are at least $\log(4/c_*)\Th({h_3})$ single-edge updates inside $\mathcal{T}_{r_1}$ with probability $1-\delta_3$, where $\delta_3< c_* /3$. 
Let $\{Y_t'\}$ denote the same random-cluster dynamics starting at $Y_0^0$, but we explicitly require that in the underlying graph, the root $r_1$ is not wired to anything. 
Under the monotone coupling $(Y_t^0, Y_t')$ such that  $Y_t^0\succeq Y_{t}'$ for any $t$.
If there are at least $\log(4/c_*)\Th({h_3})$ single-edge updates inside $\mathcal{T}_{r_1}$ by $\tau_1$,  then by \eqref{eq: tmix extend} the distribution of $Y_{\tau_1}'$ has  equilibrated $\mathcal{T}_{r_1}$ except with a total-variation error of $\delta_2<c_*/3$.
Thus there is a random-cluster induced path between  root $r_1$ and $\mathcal{C}_1(\xi_{(r_1)})$  in $Y_{\tau_1}'$ with probability at least 
\[
    c_* - \delta_3 - \delta_2 > \frac{c_*}{3}.
\]
If we denote by $\mathcal{W}_1$ this event that in $Y_{\tau_1}^0$, then we have with probability at least $c_*/3$ that $Y_{\tau_1}^0,Y_{\tau_1}^1 \in \mathcal W_1$. 

On the event $\mathcal{W}_1$,
since updates to $\mathcal{T}_{r_1}$ are censored in the period $[\tau_1, \tau_1+\tau_2]$ and $r_1$ is connected to $\xi_{(r_1)}$ in $\mathcal T_{r_1}$,
the stationary distribution of $Y_t^0,Y_t^1$ given $Y_{\tau_1}^0(\mathcal T_{r_1}),Y_{\tau_1}^1(\mathcal T_{r_1})$ are the same product distribution over the remaining $d_\kappa-1$ trees with each of their roots wired to the boundary (because the roots are wired to $r_1$, which is wired to the boundary through $\mathcal T_{r_1}$), i.e.,  $\bigotimes_{j=2}^{d_\kappa}\mu_{\mathcal{T}_{r_j}}^{\xi_{(r_j)}, \circlearrowleft}$.

By definition, the mixing time of the random-cluster dynamics on each of these trees is at most $\tau_{\mix}(h_3)$; thus by Chernoff's inequality, Lemma~\ref{lem: simulating BD 3}, and \eqref{eq: tmix extend} with a union bound to move from mixing time to coupling time under a monotone coupling, one has that 
\[
\Pp\left[Y_{\tau_1+\tau_2}^0(\mathcal{T}_{r_j}) = Y_{\tau_1+\tau_2}^1(\mathcal{T}_{r_j}), ~\forall 2 \le j \le d_\kappa  \mid \mathcal{W}_1\right] \ge  \bigl(1-\frac{1}{6d_\kappa}\bigr)^{d_\kappa-1} \ge 
\frac{5}{6}\,.
\]
Let $\mathcal{W}_2$ be the event that the coupling succeeds in all these trees $\{\mathcal{T}_{r_{j}}\}_{j=2}^{d_\kappa}$.

Finally, on $\mathcal W_1\cap\mathcal W_2$, since only updates to $\mathcal{T}_{r_1}$ in the period $[\tau_1+\tau_2, \tau_1+\tau_2+\tau_3]$ are uncensored, the random-cluster dynamics $Y_t^0$ and $Y_t^1$ are  either both converging to  $\mu_{\mathcal{T}_{r_1}}^{\xi_{(r_1)}}$ or both to   $\mu_{\mathcal{T}_{r_1}}^{\xi_{(r_1)}, \circlearrowleft}$. 
In either case, the mixing time is upper bounded by $\tau_{\mix}(h_3)$, so we have 
\[
\Pp\left[ Y_{\tau}^0(\mathcal{T}_{r_1}) =Y_{\tau}^1(\mathcal{T}_{r_1}) \mid \mathcal W_1\cap\mathcal W_2 \right] \ge 
\frac{5}{6}.
\]
Combining the last two estimates, we get
\[
    \pr[Y_\tau^0=Y_\tau^1\mid \mathcal W_1]
    \ge
    1-\frac{1}{6}-\frac{1}{6}
    =
    \frac{2}{3}.
\]
Therefore, by a union bound we have 
\[
\Pp\left[ 
    Y_\tau^0 = Y_\tau^1
\right]
\ge 
\Pp\left[ 
    Y_\tau^0 = Y_\tau^1 \mid \mathcal{W}_1
\right] \cdot \Pp\left[ \mathcal{W}_1 \right]
\ge \frac{c_*}{3} \cdot \frac{2}{3}
= \frac{2c_*}{9},
\]
which establishes the desired claim.
\end{proof}

\section{Mixing on random regular graphs for $p > \ps$}\label{sec: rrg mixing}

The main aim of this section is to prove the mixing time bound on random regular graphs throughout the low-temperature regime of $p>\ps$ so long as $q = \Omega(\log d)$. This is done by combining a burn-in, the weak spatial mixing estimates on trees from Section~\ref{sec:trees-uniqueness-WSM} with $\theta$-wired boundary conditions for $\theta>0$, and the mixing time estimates of Section~\ref{sec: fast mixing} on trees with $\theta$-wired boundary conditions. 

This proof will be completed modulo extensions of those WSM and mixing time bounds from trees with single-component boundary conditions to treelike graphs with $O(1)$ tree excess, and boundary conditions that may additionally have $O(1)$ many wirings disjoint from the single giant component. 
Those extensions are themselves non-trivial since at large $p$, small graph modifications can have large effects on connection probabilities, but are less central to the main results of this paper and are therefore deferred to Section~\ref{sec:mixing on treelike graphs}. 

\subsection{Properties implying fast mixing on random regular graphs}
For some fixed $\alpha<\frac{1}{4}$, let
$
R = \alpha \log_{d} n.
$
We set $a=6/\alpha$ and let $C_0=C_0(a)$ be given by Theorem~\ref{thm: sharper bound}.
For any edge $e\in E(G)$ and $r\ge 1$, we consider the ball of radius $r$ centered at the lexicographically smaller vertex $o\in e$ and denote this ball by $B_r(e)$. 
The following two lemmas summarize the structural properties of typical random regular graphs and the behavior of i.i.d. edge percolation on them in the low-temperature regime needed in our analysis.

\begin{lemma}\label{lem:G1}
    Let $\Delta \ge 3$ and let $G$ be a random $\Delta$-regular graph on $n$ vertices. With probability $1-o(1)$ over the choice of $G$, the following holds. 
    \begin{enumerate}
        \item[\textup{(G1)}]
        For every $e\in E(G)$, either $B_R(e)$ is a tree, or there exists an edge $(u,w) \in E(B_R(e))$ such that $B_R(e)\setminus \{(u,w)\}$ is a tree.
    \end{enumerate}
\end{lemma}

On the high probability event of Lemma~\ref{lem:G1}, for each $e\in E(G)$, define the ball $B(e):=B_R(e)$. 

\begin{lemma}\label{lem:P1-P2}
    Let $\Delta \ge 3$, $q\ge C_0 \log d$, $p>\ps$, and let $G$ be a random $\Delta$-regular graph on $n$ vertices. With probability $1-o(1)$ over the choice of $G$, the following holds.  If $A$ is an i.i.d.\ edge percolation on $G$ with parameter $\hat p$ from~\eqref{eq:Glauber-update-rule}, then with probability $1-O(n^{-10})$, for every $e\in E(G)$, the boundary conditions $\xi$ induced by $A(B(e)^c)$ on $B(e)$ satisfy 
        \begin{enumerate}
        \item[\textup{(P1)}]
        There exists a constant $C_3>0$ such that
        $
        \bigl| \mu^1_{B(e)}(A(e) = 1) - \mu^\xi_{B(e)}(A(e) = 1) \bigr|
        \le C_3 d^{-2R/\alpha};
        $ 
        \item[\textup{(P2)}]
        The random-cluster dynamics on $B(e)$ with boundary condition $\xi$ or all-wired has mixing time
        \[
        |B(e)| \cdot (\log |B(e)|)^{O(\log \log |B(e)|)}\,.
        \]
    \end{enumerate}
\end{lemma}

\begin{proof}[Proof of Theorem~\ref{thm:intro-rrg-mixing}]
    Fix a graph $G$ for which Lemma~\ref{lem:G1} and Lemma~\ref{lem:P1-P2} hold.
    We analyze the mixing time of the random-cluster dynamics on this graph $G$. 
    For each $e\in E(G)$, let $B(e)$ be the ball defined in \textup{(G1)} in Lemma~\ref{lem:G1}.
        Let the constant $C_3$ be as given by \textup{(P1)} in Lemma~\ref{lem:P1-P2}. 
    By the existence of a monotone grand coupling, it suffices to upper bound the coupling time $\min\{ t : \pr \bigl(X_t^1 \neq X_t^0 \bigr) < \tfrac{1}{4} \}$, 
    where $\{X_t^1\}_{t\ge 0}$ and $\{X_t^0\}_{t\ge 0}$ denote the random-cluster dynamics started from the all-wired and all-free configurations, respectively.
 Let $C_2>0$ be a large constant and let $T_{\textsc{Burn}}:=C_2 n \log n$. 
     Define two auxiliary chains $\{Y_t^1\}_{t\ge 0}$ and  $\{Y_t^0\}_{t\ge 0}$ 
jointly evolving with $\{X_t^1\}_{t\ge 0}$ and $\{X_t^0\}_{t\ge 0}$ 
on $[0, T_{\textsc{Burn}}]$ as follows: 
    \begin{itemize}
        \item let $Y_0^0 = X_0^0$, and at any time  $t\in [0, T_{\textsc{Burn}}]$, let $Y^0_t$ select an edge $e_t$ and update it according to an independent Bernoulli random variable $\mathcal{O}_t$ with parameter $\hat{p}$, that is, $Y^0_{t+1}(e_t) = \mathcal{O}_t$, and for every edge $g\neq e_t$, $Y^0_{t+1}(g)= Y^0_{t}(g)$;
        \item for all $t\in [0, T_{\textsc{Burn}}]$, we set $Y_t^1 = X_0^1$ (equivalently it updates with $\text{Ber}(1)$ random variables);
        \item at step $t>0$, the four chains $\{X_t^1\}_{t\ge 0}$, $\{X_t^0\}_{t\ge 0}$, $\{Y_t^1\}_{t\ge 0}$ and  $\{Y_t^0\}_{t\ge 0}$  select the same edge $e_t$, and on the resampling step of $e_t$, conduct their updates jointly via a monotone coupling.
    \end{itemize}
    By monotonicity of the random-cluster dynamics, and update rule~\eqref{eq:Glauber-update-rule}, for all $t\in [0,T_{\textsc{Burn}}]$,  
    \begin{equation}
        \label{4 chains comp}
               Y_t^0 \preceq X_t^0 \preceq X_t^1 \preceq Y_t^1\, .
    \end{equation}
    Suppose $C_2>0$ is large enough so that with probability at least $1-o(1)$, every edge of $G$ is selected to update at least once during $[0, T_{\textsc{Burn}}]$.
This follows from a standard coupon collector's argument. 
If this happens, the distribution of $Y^0_{T_{\textsc{Burn}}}$ is the i.i.d.\ bond-percolation measure on $E(G)$ with parameter $\hat p$.
    Let $T := T_{\textsc{Burn}}$, and 
    let $\mathcal{A}_T$ be the event that $A=Y^0_T$ satisfies the properties stated in Lemma~\ref{lem:P1-P2}, and we know that $\pr(\mathcal{A}_T^c)  =O(n^{-10})$.
    Also, we claim that $A' = Y^1_T$ also satisfies these properties even though it does not have parameter $\hat p$. 
    Indeed, with high probability $G$ is such that the complement of every $B(e)$ is connected, in which case since $A'$ is the all-wired configuration, the all-wired boundary is induced on $B(e)$. This all-wired boundary condition then satisfies the property \textup{(P1)} trivially, and the property \textup{(P2)} by assumption.
    
    Fix an edge $e \in E(G)$.
       Suppose that after time $T$, $\{Y_t^1\}_{t\ge T}$ and  $\{Y_t^0\}_{t\ge T}$ are the censored random-cluster dynamics in which they still update edges via the monotone coupling but ignore any attempted updates to edges in $E(G)\setminus E(B(e))$.
By the censoring inequality of \cite{PW}, \eqref{4 chains comp} can be extended to any $t\ge 0$, and thus for any $t \ge 0$, it holds that 
    \[
        \pr\bigl(X_t^1(e) \neq X_t^0(e)\bigr)
        \le
        \pr\bigl(Y_t^1(e) \neq Y_t^0 (e)\bigr).
    \]
    We define
    $\eta^0 := Y^0_T(E(G) \setminus E(B(e)))$ and $\eta^1 := Y^1_T(E(G) \setminus E(B(e)))$. Up to discarding the event $\mathcal A_T^c$ which has probability $O(n^{-10})$, we can work on the event that both $\eta^0$ and $\eta^1$ serve as boundary conditions satisfying the properties (P1) and (P2).
    For any $s > 0$, by the triangle inequality, if $\mathcal F_T$ is the sigma-algebra generated by the randomness up to time $T$, 
    \begin{align}
        \label{eq: triangle for Yt}
        \pr\bigl( Y_{T+s}^1(e) \neq Y_{T+s}^0(e) \mid \mathcal F_T \bigr)
        &\le
        \bigl| \mu_{B(e)}^{\eta^0}(e\in A) - \pr(Y^0_{T+s}(e)=1\mid \mathcal F_T) \bigr| \nonumber  \\
        &\quad
        + \bigl| \mu_{B(e)}^{\eta^0}(e\in A) - \mu_{B(e)}^{\eta^1}(e\in A) \bigr| \nonumber \\ 
        & \quad + \bigl| \mu_{B(e)}^{\eta^1}(e\in A) - \pr(Y^1_{T+s}(e)=1\mid \mathcal F_T) \bigr|. 
    \end{align}
    The second term in~\eqref{eq: triangle for Yt} is controlled by \textup{(P1)}:
    \begin{align*}
        \bigl| \mu_{B(e)}^{\eta^0}(e\in A) - \mu_{B(e)}^{\eta^1}(e\in A) \bigr|
        &\le
        \bigl| \mu_{B(e)}^{1}(e\in A) - \mu_{B(e)}^{\eta^1}(e\in A) \bigr|
        + \bigl| \mu_{B(e)}^{1}(e\in A) - \mu_{B(e)}^{\eta^0}(e\in A) \bigr| \\
        &\le
         C_3 d^{-2R/\alpha}
        =  C_3 d^{-2\log_{d} n}
         = o(1/n)\,.
    \end{align*}
    Now set $s=\max_e \{ \tau_{B(e)}\cdot \frac{10\Delta n}{|B(e)|}\}$, where by (P2)
      \[
        \tau_{B(e)} := |B(e)| (\log |B(e)|)^{O(\log \log |B(e)|)}.
    \]
    By a Chernoff bound, except with probability at most $O(n^{-10})$,  by time $T+s$ the number of updates within $B(e)$ is at least
    $T_{\mix}(P^{\eta^0}_{B(e)}, n^{-10})$. Therefore, we have 
    \[
        \bigl\| \mu_{B(e)}^{\eta^0} - \pr(Y^0_{T+s} \in\cdot  ) \bigr\|_{\textsc{TV}}
        \le O(n^{-10})\,,
    \]
    which yields the desired bound on the first term in~\eqref{eq: triangle for Yt}.
    Similarly, after
    $T_{\mix}(P^{\eta^1}_{B(e)}, n^{-10})$ updates within $B(e)$, the third term in~\eqref{eq: triangle for Yt} is also upper bounded by $O(n^{-10})$.

    Hence, we have
    \[
    \pr\bigl(X_{T+s}^1(e) \neq X_{T+s}^0(e) \bigr) \le 
        \pr\bigl( Y_{T+s}^1(e) \neq Y_{T+s}^0(e)  \bigr) \le \pr\bigl( Y_{T+s}^1(e) \neq Y_{T+s}^0(e) \mid \mathcal A_T \bigr) + \mathbb P(\mathcal A_T^c) \,.
    \]
    Taking a union bound over the at most $\Delta n$ edges of $G$ yields
    \[
        \pr\bigl(X_{T+s}^1 \neq X_{T+s}^0 \bigr)  \le 
       \Delta n \max_e  \mathbb P(Y_{T+s}^1(e) \neq Y_{T+s}^0(e) \mid \mathcal A_T ) + \Delta n \pr\bigl(\mathcal{A}_T^c\bigr) =  o(1)\,.
    \]
    Therefore, the coupling time of the  random-cluster dynamics on $G$ is at most
    $$
   T + \max_e \bigl\{ \tau_{B(e)}\cdot \frac{10\Delta n}{|B(e)|}\bigr\}
    = C_2 n \log n+
    10\Delta n \max_e \bigl\{(\log |B(e)|)^{O(\log \log |B(e)|)} \bigr\}\,,
    $$ 
    which is at most $n(\log n)^{O(\log \log n)}$, 
    and it upper bounds the mixing time.
\end{proof}

\subsection{Treelike nature of random regular graphs: establishing (G1)}

We include here a lemma that random regular graphs are locally treelike in the sense that all their $(1/4 - o(1))\log_{d} n$ balls are at most unicyclic. This is standard (see e.g.,~\cite[Fact 2.17]{gheissari2025rapid}), but we include a proof for self-containedness.

\begin{proof}[Proof of Lemma~\ref{lem:G1}]
    By contiguity of the configuration model to $\mathcal G_\Delta(n)$ it suffices to work with the configuration model where every vertex has $\Delta$ half-edges, and the (multi-)graph is given by a uniform at random matching of those half-edges. By a union bound, it suffices to fix a vertex $v$ and show that the probability that $B_R(v)$ contains at least two edges that close cycles is $o(1/n)$. This probability can be obtained by revealing $B_R(v)$ in a breadth-first manner, starting from matching the $\Delta$ half-edges of $v$, then those of the vertices it connected to, and so on.  

    At each matching step of this revealing process, the probability that a pair of half-edges get matched between vertices already revealed to belong to $B_R(v)$, conditional on any previous history of the revealing procedure up to that step, will evidently be bounded by $\frac{\Delta (\Delta-1)^{R}}{n}$. 
    In order to have tree excess greater than one, in the $\Delta(\Delta-1)^R$ many attempts at such an event, it must occur at least twice. Therefore, the probability of $B_R(v)$ having tree excess greater than one is at most 
    \begin{align*}
        \mathbb P\Big( \text{Bin}(\Delta(\Delta-1)^R, \tfrac{\Delta(\Delta-1)^R}{n})\ge 2\Big)\,.
    \end{align*}
    In turn, since $\Delta(\Delta-1)^R =  \Delta n^{\alpha}$, this is bounded by $\Delta^4 n^{2\alpha}  n^{2\alpha -2}$. If $\alpha<1/4$ the exponent on $n$ is strictly less than $-1$ as desired. 
\end{proof}

\subsection{Percolation on random regular graphs: establishing (P1)--(P2)}

Recall that after its burn-in, the random-cluster dynamics from any initialization is above a parameter-$\hat p$ Bernoulli percolation on the random regular graph. Moreover, when $p>\ps$, this Bernoulli percolation is supercritical. The two properties, (P1)--(P2) in Lemma~\ref{lem:P1-P2} were about local spatial and temporal mixing of balls with the boundary conditions induced by the supercritical $\hat p$ percolation on $E(G)$. Towards that, we have the following. 

\begin{lemma}\label{lem:rrg-percolation-properties}
    Fix $\Delta\ge 3$ and $\hat p> 1/d$ and $\alpha<1/2$. Let $G$ be a random $\Delta$-regular graph on $n$ vertices, and let $A$ be obtained by i.i.d.\ edge percolation with parameter $\hat p$. There exist constants $L,K > 0$ such that with probability $1-o(1)$ , $G$ is such that all of the following hold: 
    \begin{itemize}
        \item With probability $1-O(n^{-10})$, $A$ has a giant component $|\mathcal C_1(A)| \ge K^{-1} n$, and all other components of $A$ have size at most $K\log n$. Moreover, the same is true of $A \setminus B(e)$ for every $e\in E(G)$.
        \end{itemize}
     Furthermore, for all $v\in V(G)$, and $r = \alpha\log_{d} n$, 
        \begin{itemize}
        \item Let $(\eta_w)_{w\in \partial B_r(v)}$ be indicator random variables of $w$ belonging to a component of size greater than $K \log n$ in $B_r(v)^c$. There exists a set $H$ of size at most $L$ such that $(\eta_w)_{w\in \partial B_r(v)\setminus H}$ stochastically dominates a product measure of independent $\text{Ber}(1/K)$. 
        \item With probability $1-O(n^{-10})$, $A$ is such that for all $v\in V(G)$, the number of vertices in $\partial B_r(v)$ that are in non-singleton, non-giant boundary components induced by $B_r(v)^c\cap A$ is at most $L$. 
    \end{itemize}
\end{lemma}

\begin{proof}
    The first claim for $A$  is well-known about the supercritical phase of percolation on the random regular graphs, see e.g.,~\cite{PittelRGPercolation,nachmias2010critical} for in fact much more refined estimates. The fact that the supercritical properties also hold for $A\setminus B(e)$ (where we are removing a ball on $n^{\alpha}$ many vertices) follows by the same arguments. For a precise reference, we apply Theorem 22 of~\cite{BollobasRiordan} to $\hat p$-percolation on $G \setminus B(e)$; the exposure of the ball $B(e)$ in the configuration model means changing $O(n^{\alpha})$ many of the degree counts from $\Delta$ to something in $\{0,...,\Delta\}$. This does not affect the convergence of the empirical degree distribution to $\delta_{\Delta}$, so their theorem is applicable to the resulting configuration model for generating $G \setminus B(e)$. Though the specific statement of Theorem 22 of~\cite{BollobasRiordan} gives an exponentially small probability of a linear-sized second component, the same proof gives a polynomially small probability of a $K\log n$ sized second component (e.g., as in~\cite[Lemma 11]{MolloyReed}).   
    
    For the second item, fix a $\delta$ small such that $\alpha + \delta <1/2$. By the same proof as in Lemma~\ref{lem:G1} (see e.g., Fact 2.3 of~\cite{BG21}), with probability $1-o(1)$ the graph $G$ is such that for all $v$, the ball $B_{(\alpha + \delta) \log_{d} n}(v)$ has tree excess at most $L$ for a fixed constant $L = L(\alpha + \delta)$. Fix such a graph $G$. For each $v$, let $H(v)$ be the set of $w\in \partial B_r(v)$ such that the part of the breadth-first-search (BFS) representation of $B_{(\alpha + \delta) \log_{d} n}(v)$ below $w$ is \emph{not} a complete $d$-ary tree of depth $\delta \log_{d} n$ (disjoint from those of all other vertices on $\partial B_r(v)$).  
    
     For every $w\in \partial B_r(v)\setminus H$, let $T_{w,h}$ be that tree of depth $h= \delta \log_{d} n$. Consider the event, measurable with respect to the percolation configuration on $E(T_{w,h})$ that the connected component of $w$ within $T_{w,h}$ consists of at least $(1+\eta/2)^{h}$ many vertices, where $(1+\eta) = \hat p d$ so $\eta>0$. This has $\theta = \theta(\eta)>0$ probability by supercriticality of the resulting $\text{Bin}(d,\hat p)$ Galton--Watson tree and a Hoeffding bound: see for example the proof of~\cite[Fact 16]{BG24-PTRF}. Moreover, this event is independent of the same event for all the other vertices because of disjointness of $T_{w,h}$ from $T_{w',h}$ for $w,w'\in \partial B_r(v) \setminus H$. Taking $K$ a large constant depending on $\hat p,d$, it will satisfy $\theta> 1/K$ yielding the proof. 

    We finally turn to the third item. For this item, it is helpful to compute the probability jointly under $G$ and $A$. By contiguity of the configuration model and the random $\Delta$-regular graph, we use the configuration model to generate $G$. 
    Fix a vertex $v\in V$ and consider the probability that more than $L$ vertices on $\partial B_r(v)$ are connected through $A \cap B_r(v)^c$ but not connected to $\mathcal C_1(A\setminus B_r(v))$ in $B_r(v)^c$, then union bound over the $n$ possible choices of $v$. To bound this probability, first condition on the ball $B_r(v)$ in $G$, then pick $L$ many vertices $w_1,...,w_L$ that are set to be the vertices in $\partial B_r(v)$ in non-trivial but non-giant components of $A \cap B_r(v)^c$: there are at most $\binom{n^{\alpha}}{L}\le n^{\alpha L}$ many such vertices. We now compute the probability that these $L$ many vertices are in non-trivial non-giant components of $A \cap B_r(v)^c$.
    By item (1) of the lemma, with probability $1-O(n^{-10})$, for all vertices that event can be replaced with the event that $w_1,...,w_L$ are paired through components each of which are of size at most $K \log n$. It hence suffices to bound the probability of 
    \begin{align*}
        \mathbb P( w_1,...,w_L \text{ paired through components of size at most $K\log n$ in $A \cap B_r(v)^c$} \mid B_r(v))
    \end{align*}
    where here $w_1,...,w_L$ connected means that for each $i\in [L]$, there exists $j\in [L]$ such that $w_i$ and $w_j$ are connected in $A \cap B_r(v)^c$. 

    For each $w_i$, let $\mathcal A_i$ be the event that it is connected to some $(w_j)_{j\ne i}$ through a component of $A \cap B_r(v)^c$ of size at most $K \log n$. We reveal certain edges of the graph, and the value of $A$ on those edges simultaneously to bound the probabilities of $\mathcal A_i$. To describe the iterative process, suppose that we have revealed the edges $F_l$ together with their percolation statuses, $A(F_l)$, and let $\mathcal{F}_l$ be all information revealed at step $l$, i.e., $B_r(v)$ as well as $(F_l,A(F_l))$. Then, we bound the probability of 
    \begin{align}\label{eq:conditional-probability-of-A-l+1}
        \mathbb P(\mathcal A_{l+1} \mid (F_l,A(F_l)), B_r(v))
    \end{align}
    as follows. Consider vertex $w_{l+1}$; if it has already been shown to belong to an edge in $A(F_l)$, then $(F_l, A(F_l))$ implies $\mathcal A_{l+1}$ and this probability is $1$. Else, initialize $E_{l+1,0}$ as the un-matched (in $B_r(v)$ together with $F_l$) half-edges of $w_{l+1}$, and while $j \le \Delta K \log n$, 
    \begin{enumerate}
        \item Let $x_{l+1,j}$ be the vertex incident to the next half-edge $\hat e\in E_{l+1,j}$, chosen in an arbitrary ordering. Match $\hat e$ to a uniformly chosen unmatched half-edge other than itself, and denote the vertex incident to its partner by $u_{l+1,j}$. This produces the edge $e={x_{l+1,j},u_{l+1,j}}$.
        \item  Flip an independent $\text{Ber}(\hat p)$ coin for the value $A$ will take on that edge $A(e)$; 
        \item If $A(e) = 1$ and $u_{l+1,j}$ is in a revealed connected component of $\{w_1,...,w_L\}$, terminate. 
        \item Else, form $E_{l+1,j+1}$ by removing $\hat e$, and if $A(e) =1$, adding all un-matched half-edges of the other endpoint of vertex $u_{l+1,j}$. Then increment $j\gets j+1$. 
    \end{enumerate}
    Let $\mathcal B_{l+1}$ be the  event that the procedure above terminates on Step 3 (rather than on $j$ reaching $\Delta K \log n$ or the un-matched half-edge set hits empty).
    Some observations are in order. Firstly, if the process does not terminate in step 3 for any $j\le \Delta K \log n$, then necessarily its component either does not intersect any $(w_k)_{k\in [L]}$ or its component is of size at least $K \log n$, and in both cases $\mathcal A_{l+1}$ does not happen. Therefore, if $w_{l+1}\notin A(F_l)$, then the probability in~\eqref{eq:conditional-probability-of-A-l+1} is at most the probability of $\mathbb P(\mathcal B_{l+1} \mid (F_l,A(F_l)), B_r(v))$. That probability is in turn upper bounded by a union bound, by  
    \begin{align*}
        \Delta K \log n \cdot \frac{\Delta L K \log n}{\Delta n - \Delta\sqrt{n} - \Delta K L\log n} = O\Big(\frac{(\log n)^2}{n}\Big)\,.
    \end{align*}
    Secondly, for at least $L/2$ many of the $l\in [L]$, we actually incur this probability for~\eqref{eq:conditional-probability-of-A-l+1}, because the revealing procedure for $w_i$ can cover at most one new vertex in $(w_j)_{j>i}$ (as soon as it shows an edge in $A(F_l)$ hitting some $w_j$ it terminates). Putting these two properties together and iterating~\eqref{eq:conditional-probability-of-A-l+1}, we arrive at 
    \begin{align*}
        \mathbb P(\mathcal A_1,...,\mathcal A_L \mid B_r(v)) \le \Big(\frac{(\log n)^2}{n}\Big)^{L/2}\,.
    \end{align*}
    This gets multiplied by $n^{\alpha L}$ for the choices of $w_1,...,w_L$ and thus the probability of the bad event for the boundary vertices of $\partial B_r(v)$ is at most $\tilde O(n^{\alpha L - L/2})$. Since $\alpha<1/2$, there exists $L = L(\alpha)$ large but $O(1)$ such that this probability is at most $o(n^{-11})$, and a union bound over the $n$ many vertices concludes the proof. 
\end{proof}

Observe that on the event of item (1), there is a giant component and all other components are $O(\log n)$-sized, the indicators $(\eta_w)_{w\in \partial B_r(v)}$ are equivalent to belonging to a single boundary component in the boundary conditions induced by $A\cap B_r(v)^c$. As such we obtain that up to $O(1)$ modifications, the boundary conditions induced by the percolation on $B(e)$ for all $e\in E(G)$ are $\theta$-wired for $\theta>0$.  

\begin{definition}\label{def:theta-Q-wired}
A distribution $\mathsf{M}$ over boundary conditions on $\partial B$ is \emph{$(\theta,Q)$-wired} if it admits the following coupling. First sample a subset $J\subseteq\partial B$ by including each vertex independently with probability $\theta$. Conditional on $J$, sample an exceptional set $U\subseteq\partial B$ and a boundary condition $\xi$, with marginal law $\xi\sim\mathsf{M}$, such that almost surely $|U|\le Q$ and there exists a set $W = W(\xi)$ satisfying
$$
J\setminus U\subseteq W\subseteq\partial B\setminus U
$$
for which the restriction of $\xi$ to $\partial B\setminus U$ consists of the component $W$, and singleton blocks for all remaining vertices. The vertices in $U$ may be wired arbitrarily in $\xi$.
\end{definition}

\begin{corollary}\label{cor:burn-in-induced-bc-are-good}
    There exists $Q>0$ such that with probability $1-o(1)$, $G$ is such that if $A$ is Bernoulli $\hat p$-percolation on $G$, then the law of $A$ is within $1-O(n^{-10})$ of a distribution that for every $e\in E(G)$, it induces a $(\theta,Q)$-wired boundary condition on $B(e)$. 
\end{corollary}
\begin{proof}
    Let $\mathsf{M}$ be the distribution on boundary conditions induced on each $B(e)$ that wires all vertices of $\partial B(e)$ belonging to a component of size greater than $K \log n$ in $B(e)^c$ for $K$ from Lemma~\ref{lem:rrg-percolation-properties}, and sets the set $U \subset \partial B(e)$ to be the set $H$ from the second item in Lemma~\ref{lem:rrg-percolation-properties} together with the set of at most $L$ vertices from the third item in that lemma.  

    On the probability $1-O(n^{-10})$ event of the first item in Lemma~\ref{lem:rrg-percolation-properties}, all the vertices getting wired in $\mathsf{M}$ must also get wired in the induced boundary condition by $A$, and by the third item, no others are in non-singleton non-giant components of $A$. Therefore the joint distribution induced by $\mathsf{M}$ and $A$ on all the boundary conditions for $(B(e))_{e\in E(G)}$ are within $o(1)$ of one another. 
\end{proof}

With Corollary~\ref{cor:burn-in-induced-bc-are-good} in hand, the missing ingredients are the following two lemmas, which are extensions of Theorem~\ref{thm wsm} and Theorem~\ref{thm:tree-mixing} respectively from $\theta$-wired boundary conditions on trees to $(\theta,Q)$-wired boundary conditions to treelike graphs with at most one cycle.  

\begin{lemma}\label{lem:thetaQ-wired-implies-P1}
    Assume (G1) holds and a law $\mathsf{M}$ over the boundary conditions on $B(e)$ is $(\theta,Q)$-wired. If $p>\ps$ and $q\ge C_0 \log d$, then with probability $ 1- O(n^{-10})$, (P1) holds.  
\end{lemma}

\begin{lemma}\label{lem:thetaQ-wired-implies-P2}
    Assume (G1) holds and a law $\mathsf{M}$ over the boundary conditions on $B(e)$ is $(\theta,Q)$-wired. If $p>\ps$ and $q > 1$, then with probability $1- O(n^{-10})$, (P2) holds.  
\end{lemma}

 Lemma~\ref{lem:P1-P2} follows directly from Corollary~\ref{cor:burn-in-induced-bc-are-good} and Lemmas~\ref{lem:thetaQ-wired-implies-P1}--\ref{lem:thetaQ-wired-implies-P2}. The next section will establish these two remaining lemmas by adjustments of the arguments of previous sections.  

\section{WSM and mixing time extensions for treelike graphs}\label{sec: treelike extensions}

In this section, we prove the WSM and mixing time extensions on treelike graphs, deferred from the previous section. These will be technical generalizations of the arguments of Sections~\ref{sec:trees-uniqueness-WSM} and~\ref{sec: fast mixing} respectively.  

\subsection{WSM for treelike graphs}
\label{subsec:extensions}

In our proofs we encounter small, local perturbations of the tree structure. These include, for instance, vertices with $d\pm 1$ children, the presence of a single cycle in the BFS tree, or an 
additional non-boundary vertex being wired by the boundary condition.
We give all of these modifications a unified treatment next, showing that the conclusion of
Theorem~\ref{thm wsm}
is stable and continues to hold even under $O(1)$ many such perturbations.

We focus on a family of treelike graphs consisting of trees and unicyclic graphs.
The trees of interest are what we call \emph{almost-$\Delta$-regular trees} and include the following families of trees:
\begin{enumerate}[(i)]
    \item The root may have degree either $\Delta$ or $\Delta-1$, and exactly one non-root vertex has degree $\Delta-1$, with all other internal vertices having degree $\Delta$, and all the leaves are at the same distance from the root.
    \item The root has degree $\Delta-2$,  with all other internal vertices having degree $\Delta$, and all the leaves are at the same distance from the root.
\end{enumerate} 
For the unicyclic case, we call our family 
\emph{almost-$\Delta$-regular unicyclic},
and it includes the graphs $G=(V,E)$ satisfying all of the following properties:
\begin{enumerate}
    \item There exists a distinguished vertex $\rho\in V$ (the \emph{root}) such that the BFS tree $\mathcal{T}$ rooted at $\rho$ has all leaves at the same distance from the root.
    \item There is at most one edge $\{u,w\}\in E$ that does not belong to $\mathcal{T}$.
      \item 
      The root may have degree either $\Delta$ or $\Delta-1$
      while all internal vertices have degree $\Delta$. 
\end{enumerate}
We denote by $\mathcal{G}(h)$ the class of graphs that results from the union of
    complete $\Delta$-regular and $d$-ary trees of height $h$, the class of almost-$\Delta$-regular trees of height $h$, and
    the class of almost-$\Delta$-regular unicyclic graphs of height $h$.
Note that every graph in $\mathcal{G}(h)$ is treelike with at most a single cycle.
If $G$ is an almost-$\Delta$-regular tree, then let $v^*$ denote the non-root vertex with degree $\Delta-1$ if it exists, or let $v^*=\rho$ if the root has degree $\Delta-2$.
We will prove the following WSM result for any graph in $\mathcal{G}(h)$.
For a $(\theta, Q)$-wired boundary condition $\xi$
we extended the notation $\mathcal C_1(\xi)$ to denote the largest boundary component in $\xi$.

\begin{theorem}\label{thm wsm theta,Q wired}
Let $\Delta \ge 3$, $q > 1$, $Q\ge 0$ and $p > \ps$. 
Let $G \in \mathcal{G}(h)$ be a graph with its BFS tree $\mathcal{T}$. 
Let $\mathsf M$ be a $(\theta, Q)$-wired distribution on  boundary conditions on $\partial\mathcal{T}$. 
Then there exist constants $c > 0$, $C>0$, and $\gamma > 0$ depending only on $\Delta$, $p$, $q$, $\theta$ and $Q$, such that
for any vertex $v\in V(G)$ with $\dist(v,\partial \mathcal T) \ge \gamma$, with probability at least $1-e^{ - c d^{ \sqrt{\dist(v,\partial \mathcal T)}}}$, $\xi\sim \mathsf{M}$  is such that 
\begin{align}
\label{eq:wsm:general}
    |\mu_{ G}^1 (v\sim \mathcal C_1(1)) - \mu_{G}^\xi(v\sim \mathcal C_1(\xi))|\le C (\beta_* + o(1))^{\dist(v,\partial \mathcal T)/2}, 
\end{align}
where $\beta_* = g'(y^*)$.
Moreover,~\eqref{eq:wsm:general} also holds for the case when the root of $\mathcal T$ and/or $v^*$ are wired to $\mathcal C_1(\cdot)$ and for the event $\{e\in A\}$
for any edge $e \in E(G)$ with $\dist(e, \partial\mathcal{T}) \ge \gamma$.
\end{theorem}

The proof of this theorem will be based on the following technical lemmas.

\begin{lemma}\label{lem:generalized-positive}
      Let $\Delta \ge 3$, $q > 1$, and $p > \ps$.
Let $G \in \mathcal{G}(h)$ be a graph with its BFS tree $\mathcal{T}$ rooted at $v$. 
Let $\mathsf M$ be a $\theta$-wired distribution on  boundary conditions on $\partial\mathcal{T}$
and let $\xi$ be the single-component boundary condition that results from sampling a boundary condition from $\mathsf M$
and unwiring all the vertices in a fixed set $U$ of size $Q$.
    Then for any sufficiently large integer $\ell \ge 1$, there exist constants $\varepsilon > 0$ and sufficiently large $M > 0$ such that for all vertices $x \in V(\mathcal T)$ with $\dist(v, x) = \dist(v, \partial\mathcal T) - \ell$, 
    with probability at least $1-\exp\big( - \frac{\theta d^{\ell}}{16}+\log d \cdot \dist(v,\partial \mathcal T)\big)$ over $\xi$, we have $1 + \eps \leq f^{\xi}_{\mathcal T}(x) \leq M$.   
\end{lemma}

\begin{lemma}\label{lem: message contraction purturbed}
Let $\Delta \ge 3$, $q > 1$, and $p > \ps$.
Let $\mathcal T^v$ be a finite tree rooted at $v$. Let $I_0 = [a_0,b_0] \subset (1, \infty)$ be a compact interval. 
Suppose we are given two assignments of leaf messages on $\partial \mathcal T^v$, denoted by $f_1,f_2 : \partial \mathcal T^v \to [1,\infty)$, such that for all $x\in \partial\mathcal T^v$ we have $ f_1(x),f_2(x) \in I_0$.

Let $\hat f_1,\hat f_2 : V(\mathcal T^v)\to [1,\infty)$ be the corresponding message functions obtained by propagating these boundary values to the interior using the recursion in~\eqref{eq: function g} with $d$ children at all vertices except at those in a set $\mathcal S$ where $|\mathcal S| = O(1)$.
For each vertex $u \in \mathcal S$, we assume that the messages at the children of $u$ under $\hat f_1$ and $\hat f_2$ are in a compact interval $I_u \subset (1, \infty)$. Suppose now that there exists a compact interval $J_u \subset (1, \infty)$ such that the message is propagated instead using a map $F_u: I_u^{|N(u)|}\rightarrow J_u$ whose Lipschitz constant with respect to the $\ell^\infty$ norm is bounded.
Then, for every vertex $w \in V(\mathcal T^v)$, there exists a constant $C>0$ depending on $p$, $q$, $d$, $I_0$, $\{I_u, J_u\}_{u \in S}$ and 
the Lipschitz constants of the $F_u$'s
such that
\[
|\hat f_1(w)-\hat f_2(w)|
\le C\,(\beta_* + o(1))^{\dist(w, \partial \mathcal T^v)}.
\]
\end{lemma}

\begin{lemma}\label{lem: cycle gadget lipschitz}
Let $G$ be an almost-$\Delta$-regular unicyclic graph of height $h$ and let $e = \{u,v\}$ be the edge that does not belong to the BFS tree $\mathcal T$ of $G$.
Let $w$ be the least common ancestor of $u$ and $v$ and let $m$ be the number of children of $w$ in $\mathcal T$. Let $l + 1$ be the cycle length and $\mathbf{R}= (r_1, \dots, r_l)$ denote the incoming messages from the rooted subtrees attached to the vertices on the cycle other than $w$, after removing the two cycle edges incident to it. Let $\mathbf{T} = (t_1, \dots, t_{m-2})$ denote the messages from the remaining $m-2$ children of $w$.

Assume that all coordinates of $\mathbf{R}$ and $\mathbf{T}$ are in a fixed compact interval $[a, b] \subset (1, \infty)$. Then the message at $w$ can be written as $f(w)=F(\mathbf{R},\mathbf{T}) = F(r_1,\dots,r_l,t_1,\dots,t_{m-2})$
for a function $F:[a, b]^{l + m - 2} \to [a', b']$ where $[a', b'] \subset (1,  \infty)$ is a compact interval depending only on $p, q, \Delta, a, b$. Moreover, $F$ has Lipschitz constant with respect to the $\ell^\infty$ norm on $[a, b]^{l + m - 2}$ bounded by a quantity depending on $p, q, \Delta$ and linearly on the cycle length $l$.

Moreover, let $s$ denote the incoming message at $w$ from its parent, with this coordinate omitted when $w$ is the root. Assuming also that $s\in[a,b]$, for every edge $e'$ of the cycle its open-edge marginal can be written as $\mu_G^\xi(e'\in A)=F_{e'}(s,\mathbf R,\mathbf T)$, where $F_{e'}$ has Lipschitz constant $O(l)$ on the corresponding compact domain, with the implicit constant depending only on $p,q,\Delta,a,$ and $b$.
\end{lemma}

\begin{proof}[Proof of Theorem~\ref{thm wsm theta,Q wired}]

Fix a vertex $v \in V(G)$ 
with $\dist(v,\partial \mathcal T) \ge \gamma$
and view $\mathcal T$ as rooted at $v$. 
For a $(\theta, Q)$-wired boundary condition $\hat\xi$ on $\partial \mathcal T$,
note that $|\mathcal C_1(\hat \xi)|= \Omega(d^h)$ with
probability $1-e^{-O(d^h)}$ by a Chernoff bound.
Let $\xi$ be the single-component boundary condition that results from unwiring all the vertices in the set $U$ where $\hat\xi$ has arbitrary wirings (see Definition~\ref{def:theta-Q-wired}).
Then,
by monotonicity, it suffices to prove the claim for $\xi$.

Let $r = \dist(v,\partial \mathcal T)$ and $\mathcal T_l := \{u \in V(\mathcal T) : \dist(v, u) \leq l\}$ for some $r/2 \leq l \leq r - \sqrt{r}$ we choose later. By Lemma \ref{lem:generalized-positive}, there exists $\eps > 0$  such that with probability at least 
$$1 - \exp\left(-\frac{\theta d^{r - l}}{16} + r\log d \right),$$ we have
$
f^{\xi}_{\mathcal T}(x),f^{1}_{\mathcal T}(x) \in [1 + \eps, M]$, for all $x \in \partial \mathcal T_l$. 
(Note that we apply Lemma \ref{lem:generalized-positive} to the wired boundary condition with $\theta=1$ and $U=\emptyset$.)
Write $\mathcal{E}$ for this event. For a sufficiently large constant $\gamma$ the above probability is at least $1 - \exp(-c d^{r - l})$ for some constant $c > 0$ depending only on $\Delta, p, q, \theta$ and $Q$. 

For an integer $m \geq 1$, define the $m$-dimensional update map 
\begin{equation}\label{eq: phi_d}
\Phi_m(t_1, t_2, \dots, t_m) = \prod_{i=1}^m \Phi(t_i)    
\end{equation}
where $\Phi$ is defined in $\eqref{message recursion}$.
Note that in our tree recursion, $\Phi_d$ represents the update map with $d$ children by $\Phi$. Moreover, the one-step update at a vertex with $m \neq d$ children is given by the map $\Phi_m$ in \eqref{eq: phi_d}. For any compact domain $[a, b]^m \subset (1, \infty)^m$, $\Phi_m$ has a finite Lipschitz constant. Moreover, since $\Phi(t)>1$ for every $t>1$, the image of $\Phi_m$ on $[a,b]^m$ is contained in a compact interval $[a_m, b_m] \subset (1,\infty)$ depending only on $p,q,m,a,b$. 

We now distinguish two cases according to the graph $G$.

\medskip
\noindent \textbf{Case 1}: $G$ is unicyclic, and the cycle is fully contained in the top half of $\mathcal T$. 
In this case, we fix $l = r - \sqrt{r}$. The subgraph $G_l$ of $G$ above the level $\partial \mathcal T_l$ contains the unique cycle. They share the same vertex set, but $E(G_l) = E(\mathcal T_l) \cup \{u, w\}$ for some edge $\{u, w\} \notin \mathcal T_l$.

Let $L + 1$ be the length of the cycle, $z$ be the least common ancestor of $u, w$ and $m = |N(z)|$. Denote the cycle by $(z, z_1, \dots, z_L)$ where $z_1, \dots, z_L$ are vertices on the cycle. For each $i \in \{1, \dots, L\}$, let $\mathcal{T}_i$ be the rooted subtree hanging from $z_i$ after removing the two cycle edges incident to $z_i$. Let $\mathbf{R} \in [1, \infty)^L$ be the $L$-dimensional vector whose coordinates represent the incoming messages at each $z_i$ on $\mathcal T_i$ for $i \in \{1, \dots, L\}$. Let $\mathbf{T} \in [1, \infty)^{m-2}$ be the $(m-2)$-dimensional vector whose coordinates represent the incoming messages at each children of $z$ that is disjoint from the cycle.

Conditioned on the event $\mathcal E$, let $\mathbf{R}_1, \mathbf{R_\xi} \in [1 + \eps, M]^L$ be the corresponding $L$-dimensional vectors under the two boundary conditions. 
Furthermore, there exists $C_R > 0$ depending on $p, q, d, \eps, M$ such that
\begin{equation}\label{eq:bound on R}
    \|\mathbf R_1 - \mathbf R_\xi\|_\infty \leq C_R (\beta_* + o(1))^{\min\limits_{1\leq i\leq L}\dist(z_i, \partial\mathcal T_i)} \leq C_R (\beta_* + o(1)) ^{r/2 - \sqrt r}
\end{equation}
by Lemma \ref{lem: message contraction purturbed}. The second inequality holds because $\dist(z_i, \partial \mathcal T_i) \geq r/2 - \sqrt r$ for every $i$ as the cycle lies in the top half of $\mathcal T$.

Conditioned on the event $\mathcal E$, let $\mathbf{T}_1, \mathbf{T}_{\xi} \in [1 + \eps, M]^{m-2}$ be the corresponding $(m-2)$-dimensional vectors under the two boundary conditions. We also have that they are in $[1 + \eps, M]^{m-2}$ with the same argument as $\mathbf{R}_1$ and $\mathbf{R}_\xi$.
Moreover, since $t_j^1$ and $t_j^\xi$ are the messages at the roots of $m-2$ subtrees of $z$ with height at least $r / 2$, then we have
\begin{equation}\label{eq:bound on T}
    \|\mathbf T_1 - \mathbf T_\xi\|_\infty \leq C_T(\beta_* + o(1))^{r/2 - \sqrt r}
\end{equation} by Lemma \ref{lem: message contraction purturbed} for some $C_T > 0$ depending on $p, q, d, \eps, M$.

Combining \eqref{eq:bound on R} and \eqref{eq:bound on T} gives
\begin{equation}\label{eq:bound on Linf}
    \|(\mathbf R_1 , \mathbf T_1) - (\mathbf R_\xi , \mathbf T_\xi)\|_\infty \leq \Gamma (\beta_* + o(1))^{r/2 - \sqrt r}
\end{equation}
for some constant $\Gamma > 0$ depending on $p, q, d, \eps, M$.

By Lemma \ref{lem: cycle gadget lipschitz}, we have the update map $F: [1 + \eps, M]^{L +m-2} \to [a', b'] \subset (1,\infty)$ at $z$ such that $F$ has Lipschitz constant bounded by $C_F \cdot L$ for some constant $C_F > 0$ depending on $p, q, d, \eps, M$.
Therefore, we have
\[
|F(\mathbf R_1 , \mathbf T_1) - F(\mathbf R_\xi , \mathbf T_\xi) | \leq C_FL \|(\mathbf R_1 , \mathbf T_1) - (\mathbf R_\xi , \mathbf T_\xi)\|_\infty \leq C_F L \Gamma (\beta_* + o(1))^{r/2 - \sqrt r}.
\]
Now since $L \leq r$, and the bound on $(\beta_* + o(1))^{r/2 - \sqrt r}$ decays exponentially in $r$, the factor $O(L)$ can be absorbed. Therefore, under the given assumption, the update at $z$ may in fact be regarded as having a  Lipschitz constant bounded by a constant depending only on $p, q, d, \eps, M$. 

Apart from the cycle gadget, the recursion is the usual tree recursion $\Phi_d$ in \eqref{eq: phi_d}, except possibly at $O(1)$ exceptional vertices where the number of children differs from $d$. For these exceptional vertices, the update maps are of the form $\Phi_m$ for some $m$, thus the corresponding update maps have Lipschitz constants bounded by a constant depending only on $p, q, d$. Hence, every exceptional update map within $\mathcal T_l$ has a compact domain contained in $(1, \infty)^m$ where $m$ is the number of children, a compact image interval contained in $(1,\infty)$, and a bounded Lipschitz constant. Therefore the hypotheses of Lemma~\ref{lem: message contraction purturbed} are satisfied.

Therefore, conditioned on the event $\mathcal E$ above, all messages on $\partial \mathcal T_l$ lie in $[1 + \eps,M]$ under both boundary conditions, Lemma \ref{lem: message contraction purturbed} applies directly and yields for some $C>0$ that 
\[
|f_G^1(v)-f_G^\xi(v)| \leq C(\beta_*+o(1))^l.
\]

\medskip
\noindent \textbf{Case 2}: The top half of $G$ is acyclic. Fix $l = r / 2$.
This includes both the case that $G$ has no cycle and the case that $G$ is unicyclic but the unique cycle is not contained in the top half. In either case, the subgraph of $G$ induced by the vertices of distance at most $l$ is identical to $\mathcal T_l$, which is either a $\Delta$-regular or an almost-$\Delta$-regular tree.

If $G$ is unicyclic, the bottom half may contain the cycle and, after conditioning on the configuration below $\partial\mathcal T_l$, may induce one additional boundary block $\{a,b\}\subseteq\partial\mathcal T_l$ disjoint from $\mathcal C_1$.  Couple the configurations below the cut under the two boundary conditions monotonically.  Let $\zeta_\xi^-$ be the boundary condition obtained from the induced boundary condition in the lower chain by splitting $a$ and $b$ into singletons, and let $\zeta_1^+$ be obtained from that in the upper chain by wiring $a$ and $b$ to $\mathcal C_1$. Then $\zeta_\xi^-\preceq \zeta_\xi \preceq \zeta_1\preceq \zeta_1^+$. Hence, for every increasing observable $\mathcal O$ supported above the cut, $0\le \mu_G^1(\mathcal O)-\mu_G^\xi(\mathcal O) \le \mu_{\mathcal T_l}^{\zeta_1^+}(\mathcal O) - \mu_{\mathcal T_l}^{\zeta_\xi^-}(\mathcal O)$. 

Both extreme boundary conditions on the right are single-component boundary conditions. The freeing or wiring of $a$ and $b$ produces at most two exceptional incoming coordinates. We absorb their first updates into the corresponding exceptional maps; by finite energy these maps have compact image and uniformly bounded Lipschitz
constant.  On the event $\mathcal E$, all remaining incoming messages belong to the common compact interval.  Lemma~\ref{lem: message contraction purturbed} therefore gives $\mu_{\mathcal T_l}^{\zeta_1^+}(\mathcal O) - \mu_{\mathcal T_l}^{\zeta_\xi^-}(\mathcal O) \le C(\beta_*+o(1))^l$ for $\mathcal O=\{v\sim\mathcal C_1\}$, with the loss of the fixed number of exceptional levels absorbed into $C$.

Hence, from the level $\partial \mathcal T_l$ up to $v$, the recursion is the usual tree recursion $\Phi_d$ in \eqref{eq: phi_d}, except possibly at $O(1)$ exceptional vertices where the number of children differs from $d$. For these exceptional vertices, the update maps are also $\Phi_m$ for some $m$, and the corresponding update maps have Lipschitz constants bounded by a constant depending only on $p, q, d$. Hence, every exceptional update map within $\mathcal T_l$ has a compact domain contained in $(1, \infty)^m$ where $m$ is the number of children, a compact image interval contained in $(1,\infty)$, and a bounded Lipschitz constant. Therefore the hypotheses of Lemma~\ref{lem: message contraction purturbed} are satisfied.

Conditioned on the event $\mathcal E$ above, all messages on $\partial \mathcal T_l$ lie in $[1 + \eps,M]$ under both boundary conditions, Lemma \ref{lem: message contraction purturbed} applies directly and yields
\[
|f_G^1(v)-f_G^\xi(v)| \leq C(\beta_*+o(1))^l
\]
for some constant $C > 0$.

Thus, for both cases, with probability at least $1 - \exp(-cd^{\sqrt{\dist(v, \partial \mathcal T)}})$, we have
   $ |f_G^1(v)-f_G^\xi(v)| \leq C(\beta_*+o(1))^l \leq C(\beta_*+o(1))^{r/2}$
and hence as in~\eqref{eq:message to probability}, 
\[
    \mu_{ G}^1 (v\sim \mathcal C_1(1)) - \mu_{G}^\xi(v\sim \mathcal C_1(\xi)) \leq \frac{1}{q}|f_G^1(v)-f_G^\xi(v)| = C(\beta_*+o(1))^{r/2}.
\]

The same argument also applies when the root or the special vertex $v^*$ is wired under both boundary conditions. Indeed, wiring these vertices only pins one incoming coordinate in their local update maps on $\mathcal T$ to a fixed value $\Phi(\infty) = (1-p)^{-1}$. The following fact about Lipschitz functions implies that the pinned update maps have the Lipschitz constant.
\begin{fact}\label{fact: pinned coordinate lipschitz}
    Let $M > 0$ be a sufficiently large constant and $F: [1,M]^{m(F)} \to [1,M]$ be an  
    $L(F)$-Lipschitz function with respect to the $\ell^\infty$ norm. For $I \subseteq \{1, \dots, m(F)\}$, and $\gamma = (\gamma_i)_{i \in I} \in [1, M]^I$, define 
    $
        F_{I,\gamma} : [1,M]^{m(F)-|I|} \to [1, M]
    $
    by fixing the coordinates indexed by $I$ to be equal to $\gamma$.
    Then the function $F_{I, \gamma}$ is also $L(F)$-Lipschitz. 
\end{fact}

Hence,
\[
|\mu_G^{1,\circlearrowleft}(v\sim \mathcal C_1(1))
-\mu_G^{\xi,\circlearrowleft}(v\sim \mathcal C_1(\xi))|
\le C(\beta_*+o(1))^{r/2}.
\]
and the same holds with the same probability if $v^*$ is wired to $\mathcal{C}_1(\xi)$.

It remains to verify the claim for the observable $\{e\in A\}$. If the edge $e$ lies in the acyclic part above the chosen cut, the preceding monotone sandwich applies with $\mathcal O=\{e\in A\}$.  On a tree, the open-edge marginal is a smooth rational function of the two cavity messages adjacent to $e$; on the common compact message domain its Lipschitz constant is uniformly bounded.  Lemma~\ref{lem: message contraction purturbed} therefore gives the same bound 
\[
   |\mu_G^1(e\in A)-\mu_G^\xi(e\in A)|
   \le C(\beta_*+o(1))^{\dist(e,\partial\mathcal T)/2}.
\]

If the edge $e$ belongs to the cycle gadget, we instead use the edge-marginal part of Lemma~\ref{lem: cycle gadget lipschitz}. The additional parent-side coordinate is controlled by the same message-contraction estimate as the other incoming coordinates, and the $O(l)$ Lipschitz factor is absorbed into the exponential decay exactly as in \eqref{eq:bound on T} and \eqref{eq:bound on Linf}. These two cases cover every edge of $G$.
\end{proof}

We note that the variants of WSM in Corollary \ref{cor:uc} and
Lemmas~\ref{lemma: root decay} are simply special cases and thus follow immediately. These tree variants retain the full-distance decay: the $\Delta$-ary root update changes only the multiplicative constant, and the pinned version follows directly from Fact~\ref{fact: pinned coordinate lipschitz}.

\subsubsection{Proofs of supporting Lemmas \ref{lem:generalized-positive}, \ref{lem: message contraction purturbed}, and~\ref{lem: cycle gadget lipschitz}}

\begin{proof}[Proof of Lemma~\ref{lem:generalized-positive}]
The proof follows the argument in the proof of Lemma~\ref{lem:uniform-positive} with minor modifications. Consider the truncated tree $\mathcal{T}^v_{k} := \{u \in \mathcal{T}^v : \dist(v, u) \leq k\}$ rooted at $v$ with $k = \dist(v, \partial\mathcal T^v) - \ell$.
As in~\eqref{eq:m:pr1} and~\eqref{eq:prob:bound}, by Chernoff and union bounds,
for all $x \in \partial \mathcal{T}^v_{k}$, the number of vertices wired by $\xi$ in $\partial\mathcal T^{v, x}$ is at least $\frac{\theta |\partial\mathcal T^{v, x}|}{2}$
with probability at least
\begin{equation}
     1 - \exp\Big(-\frac{\theta (d^{\ell}-Q)}{8} + \log d \cdot \dist(v,\partial\mathcal{T}^v)\Big) \ge  1 - \exp\Big(-\frac{\theta d^{\ell}}{16} + \log d \cdot \dist(v,\partial \mathcal{T}^v)\Big)
\end{equation}
for $\ell$ large enough. Here we used the fact that the number of vertices wired by $\xi$ in $\partial\mathcal T^{v, x}$
stochastically dominates a $\mathrm{Binomial}(|\partial\mathcal T^{v, x}|-Q, \theta)$ random variable where $Q$ is a constant.

For the probability that $x \in \partial \mathcal{T}^v_k$ is connected to one of the vertices wired by $\xi$ in $\partial \mathcal T^{v,x}$, we note that $\mathcal T^{v,x}$ is almost a $d$-ary tree, with all but one vertex $u$ having branching $d$, and $u$ having branching $d-1$ or $d-2$.
For the $d$-ary tree the probability that $x$ is connected to $\partial \mathcal T^{v, x}$ is bounded below by some constant $\varphi_x>0$ since $\hat p > 1/d$ when $p>\ps$~ \cite{MR3616205}, and the different branching factor at a single vertex $u$ only affects this 
probability by a constant factor. 
Then, the probability that $x$ is connected to the wired component in $\partial \mathcal T^{v, x}$ is bounded below by $\frac{\varphi_x \theta}{2}$ by symmetry.
This probability can be rewritten in terms of the message functions $f^\xi_{\mathcal{T}^v}(x)$ at $x$ as in~\eqref{eq:m:pr} and the result follows. 
\end{proof}

\begin{proof}[Proof of Lemma \ref{lem: message contraction purturbed}]
We show the claim from the $d$-ary tree by replacing the update at vertices in $\mathcal S$, and each such replacement affects the message contraction only by a multiplicative constant.
Enumerate the vertices of $\mathcal S$ as $\mathcal S=\{x_1,\dots,x_k\}$ and $\dist(v, x_i) \ge \dist(v, x_j)$ if $i \leq j$. Thus, the replacements are performed from the leaves upward.

For $0 \leq i \leq k$, let $\hat f^{(i)}_1, \hat f^{(i)}_2$ denote the two message functions obtained as follows.
For vertices $x \in \{x_1,\dots,x_i\}$, we use the exceptional update maps $F_{x}$, and at every other vertex we use the usual update map $\Phi_d$ in \eqref{eq: phi_d} with $d$ children. 
For a vertex $x \in \{x_{i  +1}, \dots, x_k\}$ with $m$ children $\{y_1, \dots, y_m\}$, we use the auxiliary $d$-dimensional update map
\[
\hat\Phi_x(f(y_1),\dots,f(y_m)) := 
\begin{cases}
\Phi_d\bigl(f(y_1),\dots,f(y_m),1,\dots,1\bigr), & m<d,\\
\Phi_d\bigl(f(y_1),\dots,f(y_d)\bigr), & m\ge d.
\end{cases}
\]

For a vertex $u \in V(\mathcal T^v)$, we write
\[
    \delta_i(u) := \left| \hat f^{(i)}_1(u) - \hat f^{(i)}_2(u) \right| .
\]
Let
\[
\tilde K := \max_{F_{x_i}: 1\leq i\leq k}K(F_{x_i}),
\]
where $K(F)$ is the Lipschitz constant of $F$ with respect to the $\ell^\infty$ norm. Since $k=O(1)$, after enlarging the relevant intervals if necessary, we may fix a common compact interval $I=[a,b]\subset(1,\infty)$ which contains $I_0$, the fixed point $y^*$, all exceptional input and output intervals $I_u,J_u$ for $u \in \mathcal S$, and the images of the auxiliary maps $\widehat\Phi_{x_i}$ on the relevant domains. We may further choose  $I$ so that $g(I)\subseteq I$. The interval $I$ depends only on the compact intervals appearing in the statement and is independent of $i$, $z$, and the height of the tree.

We prove by induction on $i$ that there exists a constant $C_i > 0$ such that
\begin{equation}\label{eq: induction}
\delta_i(u) \leq C_i (\beta_* + o(1))^{\dist(u, \partial \mathcal T^v)}
\end{equation}
for every $u \in V(\mathcal T^v).$ 

For $i = 0$, the update at every vertex $u \in V(\mathcal T^v)$ is given by the homogeneous auxiliary update map, so Lemma \ref{lem: uniform contraction} and Remark \ref{rem:choice-of-beta} gives
\[
    \delta_0(u) \leq C_0 (\beta_* + o(1))^{\dist(u, \partial \mathcal T^v)}
\]
for all $u$.

Now assume that \eqref{eq: induction} holds for $i < k$, and consider the next replacement at the vertex $x_{i + 1}$. Since the replacement at the $(i + 1)$-th step is performed at the vertex $x_{i+1}$, the only vertices whose messages may change are $x_{i+1}$ and its ancestors. In particular, if a vertex $y$ is not an ancestor of $x_{i+1}$, then
\[
    \hat f^{(i + 1)}_j(y) = \hat f^{i}_j(y) \implies \delta_{i + 1}(y) = \delta_i(y)
\]
for every non-ancestor $y$ of $x_{i + 1}$ and $j = 1, 2$. Moreover, every strict ancestor of $x_{i+1}$ still uses the usual update map $\Phi_d$.

Suppose that the vertex $x_{i + 1}$ has $m_{i + 1}$ children $\{x_{i+1, 1}, x_{i+1, 2}, ..., x_{i+1, m_{i + 1}}\}$ and the update at $x_{i+1}$ is performed by $F_{x_{i+1}}$. We first consider the vertex $x_{i + 1}$ itself. Then, we have
\begin{align*}
\delta_{i + 1} (x_{i + 1}) &= |F_{x_{i+1}}\bigl(\hat f^{(i+ 1)}_1(x_{i+1,1}),\dots, \hat f^{(i+ 1)}_1(x_{i+1,m_{i + 1}})\bigr) \\
& \qquad - F_{x_{i+1}}\bigl(\hat f^{(i+ 1)}_2(x_{i+1,1}),\dots,\hat f^{(i+ 1)}_2(x_{i+1,m_{i + 1}})\bigr)|.    
\end{align*}
Since $F_{x_{i+1}}$ has Lipschitz constant at most $\tilde K$, the induction hypothesis gives 
\[
\delta_{i + 1}(x_{i + 1}) \leq K(F_{x_{i + 1}}) \max_{1 \leq j \leq m_{i + 1}} \delta_{i + 1}(x_{i + 1, j}) \leq  \tilde K \max_{1 \leq j \leq m_{i + 1}} \delta_i(x_{i + 1, j}) \leq \tilde K C_i (\beta_* + o(1))^{\dist(x_{i+1}, \partial \mathcal T^v) - 1}.
\]
Since $\beta_*>0$, the missing factor of $\beta_*+o(1)$ can be absorbed into the constant. Thus, after increasing the constant, we have
\[
\delta_{i + 1}(x_{i + 1}) \leq C_i' (\beta_* + o(1))^{\dist(x_{i+1}, \partial \mathcal T^v)}.
\]
where $C_i'$ depends only on $C_i$ and $\widetilde K$.

Now let $z$ be an ancestor of $x_{i + 1}$, and $l := \dist(z, x_{i + 1})$. Let $\mathcal T^z_l$ be the subtree of $\mathcal T^v$ rooted at $z$ and truncated at depth $l$. Thus, $x_{i + 1}$ is a leaf of $\mathcal T^z_l$. All internal vertices of this truncated tree use the homogeneous auxiliary update map: vertices outside $\mathcal S$ use the usual map $\Phi_d$, and vertices of $\mathcal S$ that have not yet been replaced use the auxiliary map $\widehat\Phi_x$, which is defined from the same $d$-dimensional homogeneous update. 

For each leaf $y \in \partial \mathcal T^z_l$, $y$ is not an ancestor of $x_{i + 1}$, and hence
\[
\delta_{i+1}(y) = \delta_i(y) \leq C_i (\beta_* + o(1))^{\dist(y,\partial \mathcal T^v)}
\]
by the induction hypothesis. If $y = x_{i + 1}$, then we have 
\[
\delta_{i + 1}(x_{i + 1}) \leq  C_i' (\beta_* + o(1))^{\dist(x_{i+1}, \partial \mathcal T^v)}
\]
from the argument above. Therefore, for every $y \in \partial \mathcal T^z_l$, 
\[
\delta_{i+1}(y) \leq C_i'' (\beta_* + o(1))^{\dist(z,\partial \mathcal T^v)-l},
\] where $C_i'' = \max\left(C_i, C_i'\right)$ depending only on $C_i$ and $\tilde K$.

Consider now $\mathcal T^z_l$ as a finite $d$-ary tree with leaf messages given by $\hat f^{(i+ 1)}_1$ and $\hat f^{(i+ 1)}_2$ on $\partial \mathcal T_l^z$. For $y \in \partial \mathcal T_l^z$, let $a_y := \hat f^{(i+1)}_1(y)$ and $b_y = \hat f^{(i+1)}_2(y)$ and write 
\[
\mathbf a := (a_y)_{y \in \partial \mathcal T_l^z}, \quad\mathbf b := (b_y)_{y \in \partial \mathcal T_l^z}.
\]
Note that all coordinate of $\mathbf a$ and $\mathbf b$ lie in a compact interval $I_{i+1}^z \subset (1, \infty)$ and 
\begin{equation}\label{eq: bound of a-b}
    \|\mathbf a -\mathbf b\|_\infty \leq C_i'' (\beta_* + o(1))^{\dist(z,\partial \mathcal T^v)-l}.
\end{equation}
    
Let $G_l$ denote the update map from the messages on $\partial T_l^z$ to the message at $z$.  All messages occurring in this recursion belong to the common compact interval $I$ fixed above.

Set $A_0=I$ and $A_{r+1}=\bigl[g(\inf A_r),g(\sup A_r)\bigr]$. For $r\ge 0$, define $\lambda_r = \sup_{\mathbf t\in A_r^d}\sum_{j=1}^d\partial_j\Phi_d(\mathbf t)$. By the mean value theorem,
\[
    |\Phi_d(\mathbf s)-\Phi_d(\mathbf t)|
    \le
    \lambda_r\|\mathbf s-\mathbf t\|_\infty,
\]
where $\mathbf s,\mathbf t\in A_r^d$. The intervals $A_r$ converge geometrically to $\{y^*\}$. Hence, for some $\kappa\in(\beta_*,1)$,
\[
    \lambda_r
    =
    \sum_{j=1}^d
    \partial_j\Phi_d(y^*,\ldots,y^*)
    +O(\kappa^r)
    =
    \beta_*+O(\kappa^r),
\]
and therefore
\[
    \prod_{r=0}^{s-1}\lambda_r
    \le C_0\beta_*^s.
\]
There are at most $k$ auxiliary update vertices in $T_l^z$. Splitting the recursion at these vertices and using their uniformly bounded Lipschitz constants changes only the multiplicative constant; the loss of at most $k$ regular levels is also absorbed into that constant.  Consequently, 
\[
    |G_l(\mathbf a)-G_l(\mathbf b)|
    \le
    \widehat C(\beta_*+o(1))^l
    \|\mathbf a-\mathbf b\|_\infty,
\]
where $\widehat C$ is independent of $i$, $z$, $l$, and the height of the tree.

Note that $\delta_{i+1}(z) = |\hat f_1^{(i+1)}(z) - \hat f_2^{(i+1)}(z)| = |G_l(\mathbf a) - G_l(\mathbf b)|$. Therefore, with \eqref{eq: bound of a-b}, we have
\[ 
\delta_{i+1}(z) \leq \hat C (\beta_* + o(1))^l C_i'' (\beta_* + o(1))^{\dist(z,\partial\mathcal T^v)-l} = C_{i + 1}(\beta_* + o(1))^{\dist(z,\partial\mathcal T^v)}.
\]

Combining this with the fact that $\delta_{i+1}(u)=\delta_i(u)$ for every non-ancestor vertex $u$ of $x_{i+1}$, after increasing $C_{i+1}$ if necessary, we conclude that
\[
\delta_{i+1}(u) \leq C_{i+1}(\beta_* + o(1))^{\dist(u,\partial \mathcal T^v)}
\]
for every $u \in V(\mathcal T^v)$. This proves \eqref{eq: induction} for $i+1$.
\end{proof}

\begin{fact} \label{fact: derivativelog}
Let $\gamma_{\mathbf{x}} \geq 0$ for all $\mathbf{x} \in \{0, 1\}^{l}$, not all zero, and $H(\mathbf{R}) := \sum_\mathbf{x} \gamma_\mathbf{x} \prod_{i=1}^{l} R_i^{x_i}$ where $\mathbf{R} \in (0, \infty)^{l}$. Then,
\[
\frac{\partial}{\partial y_i} \log H(e^{\mathbf{y}}) \in [0,1], \quad \forall i
\] where $e^{\mathbf{y}}$ is applied entrywise and $y_i := \log R_i$.
\end{fact}

\begin{proof}
    Define a probability distribution $\pi_{\gamma, \mathbf{y}}$ such that $\pi_{\gamma, \mathbf{y}} \propto \gamma_\mathbf{x}e^{\langle\mathbf{x}, \mathbf{y}\rangle}$. Since $H(e^{\mathbf{y}}) = \sum_\mathbf{x} \gamma_\mathbf{x} e^{\langle \mathbf{x}, \mathbf{y}\rangle}$, 
    \[
    \frac{\partial}{\partial y_i} \log H(e^{\mathbf{y}}) = \frac{\sum_\mathbf{x}\gamma_\mathbf{x}x_ie^{\langle \mathbf{x}, \mathbf{y}\rangle}}{\sum_\mathbf{x}\gamma_\mathbf{x}e^{\langle \mathbf{x}, \mathbf{y}\rangle}} = \mathbb{E}_{\pi_{\gamma, \mathbf{y}}}[x_i] \in [0, 1].
    \]
\end{proof}

\begin{proof} [Proof of Lemma \ref{lem: cycle gadget lipschitz}]
    Let $w$ be the least common ancestor of $u$ and $v$ in $\mathcal T$. Let $P(w, u)$ and $P(w, v)$ be the two tree paths from $w$ to $u$ and $v$, respectively. Define the cycle gadget
    \[
        H = P(w, u) \cup P(w, v) \cup \{(u, v)\}.
    \]
    
    $H$ is the unique cycle in $G$. Let $l + 1$ be the length of the cycle and denote the cycle by $(w, w_1, w_2, ..., w_l)$ where $w_1, \dots, w_l$ are the vertices on the cycle $H$ other than $w$.

    Consider the descendant subgraph $G_w$ of $G$ rooted at $w$. Since we can obtain a tree by removing the edge $e$ from $G_w$, write $G_w = \mathcal T_w \cup \{e\}$. Exactly two of the children of $w$ share the branches with $H$, and the remaining children have descendant subtrees disjoint from $H$. Therefore, from the perspective of $w$, we show that the message function in \eqref{message recursion} is defined by the product of the usual tree contributions from the children not meeting the cycle, together with one additional contribution from the gadget $H$.
    
    For each $i \in \{1, \dots, l\}$, let $\mathcal T_i$ be the rooted subtree hanging from $w_i$ after removing the two cycle edges incident to $w_i$, and let $\tilde Z_0(w_i), \tilde Z_1(w_i)$ be the corresponding partition functions according to whether $w_i$ is not connected to, or is connected to the wired component on $\partial \mathcal T_i$ through $\mathcal T_i$.
    
    Define a graph $\widetilde{G}$ as the subgraph of $G_w$ obtained by removing the subtrees belonging to children of $w$ not in $H$.
    We will split the contribution to the random-cluster partition function
    on $\widetilde{G}$ 
    among configurations outside of the cycle $H$ that induce the same connectivity profile on the vertices of $H$. For this, for each vector $\mathbf x=(x_1,\dots,x_l)\in\{0,1\}^l$, we interpret $x_i=1$ as the event that $w_i$ is connected to the wired component through $\mathcal T_i$, and $x_i=0$ otherwise.
    
    For $s \in \{0, 1\}$, let $\Psi_s$ denote the restricted partition function on $\widetilde{G}$ from configurations on $\widetilde{G}$ at $w$ corresponding to the event that $w$ is either disconnected ($s=0$) or connected ($s=1$) to the wired component. Then, we have
    \[
    \Psi_s = \sum_{\mathbf{x} \in \{0, 1\}^l} \alpha_{\mathbf{x},s}(p, q) \prod_{i=1}^l \tilde{Z}_{x_i}(w_i)
    \]
    where $\alpha_{\mathbf{x},s}(p, q) \geq 0$ depend only on the internal configurations on the gadget $H$ and on the state $s$.

    With this notation, we can define the partition functions at $w$. Recall that the contribution from the subtrees of $w$ that are disjoint from $\widetilde{G}$ factorizes exactly as in the tree recursion.
    Let $t := p/q + (1-p)$ and $w'$ range over the children of $w$ that are not in the cycle gadget $H$.
    Then, exactly as in the proof of Lemma \ref{lem:recursion}, we have
    \[
    Z_0(w) = q^2 \Psi_{0}\prod_{w'\notin H}\frac{t Z_0(w')+(1-p)Z_1(w')}{q},
    \]
    and
    \[
    Z_1(w) = q\left[    (\Psi_{0}+q\Psi_{1})\prod_{w'\notin H}\frac{t Z_0(w')+Z_1(w')}{q} - \Psi_{0}\prod_{w'\notin H}\frac{t Z_0(w')+(1-p)Z_1(w')}{q}\right]
    \]
    where $Z_0(w)$ and $Z_1(w)$ are the partition functions at $w$ on $G_w$ for configurations disconnected from, and connected, to the wired component in $\mathcal T_w$, respectively.
    
    Consequently, the message at $w$ on $G_w$ can be defined as
    \begin{equation}\label{eq: cycle update map}
    f(w) = q\frac{Z_1(w)}{Z_0(w)}+1 = L \prod_{w'\notin H}\Phi(f(w')), \qquad \text{where} \qquad  L := \frac{\Psi_{0}+q\Psi_{1}}{\Psi_{0}}
    = 1+q\frac{\Psi_{1}}{\Psi_{0}},
    \end{equation}
    and where $\Phi$ is the usual update function defined in \eqref{message recursion}. 
    
    Therefore, the effect of the cycle gadget $H$ on the message at $w$ is entirely captured by the scalar factor $L$, while the remaining subtrees contribute multiplicatively as in the tree case. In particular, the ratio $\Psi_{1}/\Psi_{0}$ depends only on the ratios $\tilde{Z}_1(w_i)/\tilde{Z}_0(w_i)$ at the vertices on the cycle $H$.
    
    Let $\rho(w_i) = \tilde{Z}_1(w_i)/\tilde{Z}_0(w_i)$ and $R(w_i) = 1 + q\rho(w_i)$ for $i = 1,\dots,l$. Moreover, write
    $\boldsymbol\rho=(\rho(w_1),\dots,\rho(w_l))$ and let $\mathbf{R} = (R(w_1), R(w_2), ... , R(w_l)) \in (0, \infty)^l$ be the corresponding $l$-dimensional incoming message vector. Note that $\mathbf R\in[a,b]^l$ implies
    \[
        \boldsymbol\rho
        \in
        \left[\frac{a-1}{q},\frac{b-1}{q}\right]^l
        \subset(0,\infty)^l.
    \]
    
    Factoring out $\prod_{i=1}^l \tilde Z_0(w_i)$ from both $\Psi_0$ and $\Psi_1$, we may write
    \[
    K(\mathbf R) := \frac{\Psi_1}{\Psi_0}
    =
    \frac{\sum\limits_{\mathbf x\in\{0,1\}^{l}} \alpha_{\mathbf x,1}(p,q)\prod_{i=1}^{l} \rho(w_i)^{x_i}}
    {\sum\limits_{\mathbf x\in\{0,1\}^{l}} \alpha_{\mathbf x,0}(p,q)\prod_{i=1}^{l} \rho(w_i)^{x_i}},
    \]
    so that $L(\mathbf{R}) = 1 + qK(\mathbf{R})$.

    We first show that the function $K$, equivalently $L$, is bounded above by some constant depending on $p, q$. Let $e_1, e_2$ be the two edges incident to $w$ in $G_w$. For any configuration $\sigma$ on $G_w$ contributing to $\Psi_1$, let $\sigma'$ be the configuration $\sigma \setminus \{e_1, e_2\}$. Note that $\sigma'$ contributes to $\Psi_0$. The weight ratio between $\sigma$ and $\sigma'$ is bounded by some constant $\alpha > 0$ depending only on $p, q$. Since we have at most three possible configurations, determined by the status of $e_1, e_2$, contributing to $\Psi_1$ being mapped to the same $\sigma'$, 
    \begin{equation} \label{eq: bound on K}
        K(\mathbf{R}) \leq 3 \alpha \qquad \text{for all }\mathbf{R} \in (0, \infty)^l
    \end{equation}

    We now show that the map $\mathbf R\mapsto K(\mathbf R)$ has a bounded Lipschitz constant on the relevant compact domain. Set $\rho_i=e^{y_i}$, equivalently, $R_i=1+q e^{y_i}$ for every coordinate $i$. Since the numerator and denominator in the definition of $K$ are multiaffine polynomials with nonnegative coefficients in $\boldsymbol\rho$, Fact~\ref{fact: derivativelog} gives 
    \[
    \frac{\partial}{\partial y_i}
    \log K(1+q e^{\mathbf y})\in[-1,1].
    \]
    Consequently, for $\mathbf R_1,\mathbf R_2\in[a,b]^l$,
    \[
    \begin{aligned}
    |\log K(\mathbf R_1)-\log K(\mathbf R_2)|
    &\leq
    \sum_{i=1}^l
    \left|
    \log\frac{R_{1,i}-1}{q}
    -\log\frac{R_{2,i}-1}{q}
    \right|\\
    &\leq
    \frac{1}{a-1}\|\mathbf R_1-\mathbf R_2\|_1 \leq
    \frac{l}{a-1}\|\mathbf R_1-\mathbf R_2\|_\infty.
    \end{aligned}
    \]
    
    Furthermore, by the inequality
        $|e^x - e^y| \leq e^{\max(x, y)} |x- y|$
    with $x = \log K(\mathbf{R}_1)$ and $y = \log K(\mathbf{R}_2)$, 
    \[|K(\mathbf{R}_1) - K(\mathbf{R}_2)| \leq C_K' |\log K(\mathbf{R}_1) -\log K(\mathbf{R}_2)| \leq C_K l\|\mathbf{R}_1 - \mathbf{R}_2\|_\infty. \] for some constants $C_K'=\max\{K(\mathbf R_1),K(\mathbf R_2)\}<3\alpha$ and $C_K  = C_K'/(a-1)$ from \eqref{eq: bound on K}. Together, this implies the function $K$, and equivalently the function $L$, has Lipschitz constant bounded by $ C l$ for some constant $C > 0$. 

    Now, returning to \eqref{eq: cycle update map}, we can write the update at $w$ as $f(w) = F(\mathbf{R},\mathbf{T}) = L(\mathbf{R})\Phi_{m-2}(\mathbf{T})$ 
    where $\Phi_{m-2}$ is defined in \eqref{eq: phi_d}. Since the functions $L$ and $\Phi_{m-2}$ are continuous, and the sets $[a,b]^l$ and $[a,b]^{m-2}$ are compact, their images are compact. 
    
    Moreover, since $[a, b] \subset (1, \infty)$, we have $\Phi(t) > 1$ for every $t \in [a, b]$. Thus, $\Phi_{m-2}([a, b]^{m-2}) \subset [a_\Phi, b_\Phi]$ for some compact interval $[a_\Phi, b_\Phi] \subset (1, \infty)$. Similarly, by the definition of $L$ above, we have $L(\mathbf R) > 1$ for every $\mathbf R \in [a, b]^l$. Hence, $L([a,  b]^l) \subset [a_L, b_L]$ for some compact interval $[a_L, b_L] \subset (1, \infty)$. This implies that $F$ maps $[a, b]^{l + m - 2}$ to $[a', b'] \subset (1, \infty)$ for some $a', b' > 1$.
    
    Then, for any $(\mathbf R_1, \mathbf T_1), (\mathbf R_2, \mathbf T_2) \in [a, b]^{l + m -2}$, we have
    \begin{align*}
    |F(\mathbf R_1,\mathbf T_1)-F(\mathbf R_2,\mathbf T_2)|
    &= |L(\mathbf R_1)\Phi_{m-2}(\mathbf T_1) -L(\mathbf R_2)\Phi_{m-2}(\mathbf T_2)| \\
    &\leq |L(\mathbf R_1)-L(\mathbf R_2)||\Phi_{m-2}(\mathbf T_1)| + |L(\mathbf R_2)||\Phi_{m-2}(\mathbf T_1)-\Phi_{m-2}(\mathbf T_2)|.
    \end{align*}

    Since $\Phi_{m-2}$ has bounded Lipschitz constant on $[a,b]^{m-2}$, and it is bounded, we have
    \begin{align*}
    |F(\mathbf R_1,\mathbf T_1)-F(\mathbf R_2,\mathbf T_2)|
    &\leq M_\Phi |L(\mathbf R_1) - L(\mathbf R_2)| + M_L|\Phi_{m-2}(\mathbf T_1)-\Phi_{m-2}(\mathbf T_2)| \\
    &\leq M_\Phi C l \|\mathbf R_1 - \mathbf R_2 \|_\infty + M_L C_\Phi \|\mathbf T_1-\mathbf T_2\|_\infty \\
    &\leq (M_\Phi C l + M_L C_\Phi) \|(\mathbf R_1,\mathbf T_1) - (\mathbf R_2,\mathbf T_2)\|_\infty
    \end{align*}
    where $M_L = \sup_{\mathbf R\in[a,b]^l}|L(\mathbf R)|, M_\Phi = \sup_{\mathbf T\in[a,b]^{m-2}}|\Phi_{m-2}(\mathbf T)|$ and $C_\Phi$ is a Lipschitz constant of $\Phi_{m-2}$. Hence the Lipschitz constant of $F$ on $[a,b]^{l+m-2}$ is bounded by a quantity depending on $p,q,\Delta,a,b$ and linearly on $l$.

    The edge-marginal statement follows from the same argument. Indeed, after fixing the state of $e'$, both the restricted partition function in which $e'$ is open and the unrestricted partition function are polynomials with nonnegative coefficients in the ratios $\rho(w_i)=(R(w_i)-1)/q$ and in the remaining incoming messages. Consequently, Fact~\ref{fact: derivativelog}, applied to the numerator and denominator, shows that the logarithmic derivative of their ratio with respect to each incoming coordinate is uniformly bounded. Since the edge marginal lies in the interval $[\widehat p,p]$, summing over the $l$ cycle coordinates gives an $O(l)$ Lipschitz bound.
\end{proof}

\subsubsection{Stronger spatial mixing rate for large $q$}

In this subsection, we improve the exponential decay rate 
of the WSM property we established in Theorem~\ref{thm wsm}
provided $q$ is logarithmically large in $d$. Specifically, for $p > \ps$, we show that the derivative of the recursion function $g$ from~\eqref{eq: function g} at the fixed point $y^* > 1$ satisfies $g'(y^*) < 1/d^a$ for any constant $a \geq 1$ provided $q \ge C_0(a) \log d$. The same thus holds for $\beta_*$ in the exponential decay of correlations statements. 

\begin{theorem} \label{thm: sharper bound}
    For $p > \ps$, let $y^* > 1$ be the non-trivial fixed point of $g$. For all $a \geq 1$, there exists a constant $C_0(a) > 0$ such that for all $d \ge 2$ and all $q \geq C_0 \log d > 2$, we have $
        g'(y^*) < \frac{1}{d^a}.$
\end{theorem}

With this theorem in hand we can now provide the proof of Lemma~\ref{lem:thetaQ-wired-implies-P1}.

\begin{proof}[Proof of Lemma~\ref{lem:thetaQ-wired-implies-P1}]
By item $\textup{(G1)}$, and by the definition of $B(e)$, we have  $B(e) \in \mathcal G(R)$. 
Consider the  BFS tree $\mathcal T$ corresponding to $B(e)$. 
Let $\xi \sim \mathsf M$ be a $(\theta, Q)$-wired boundary condition on $\partial \mathcal T$. By Theorem~\ref{thm wsm theta,Q wired}, we have with at least probability $1 - \exp(-cd^{\sqrt{\dist(e, \partial \mathcal T)}})$,
\[
\big| \mu^1_{B(e)}(e\in A) - \mu^\xi_{B(e)}(e\in A)\big| \leq C(\beta_*+o(1))^{\dist(e, \partial \mathcal T)/2} \leq C(\beta_*+o(1))^{R/2-1}
\]
for some constant $C> 0$.
Note that $R = \alpha \log _d n$ where $\alpha < 1/4$ and $a = 6/\alpha$. By Theorem \ref{thm: sharper bound}, 
\[
\beta_* = g'(y^*) < d^{-6/\alpha} \implies \beta_* + o(1) \leq d^{-5/\alpha}.
\]
Therefore, we have
\[
\big| \mu^1_{B(e)}(e\in A) - \mu^\xi_{B(e)}(e\in A)\big| \leq  (C/\beta_*) d^{-\frac{5R}{2\alpha}} = O(d^{-2 R/\alpha})
\]
with probability at least $1 - \exp(-cd^{\sqrt{\dist(e, \partial \mathcal T)}}) = 1 - \exp(-cd^{\sqrt{R}}) = 1 - o(n^{-10})$ as $n \to \infty$.
\end{proof}

Going back to the proof of Theorem~\ref{thm: sharper bound}, to better capture the dependence of $g$ on $p$, and with a slight abuse of notation, let us introduce the auxiliary function 
\[
h(p, y) = \Phi(p, y)^d = \left(\frac{y + (q-1)(1-p)}{(1-p)y+p+(q-1)(1-p)}\right)^d
\] where 
\[
\Phi(p, y) =  \left(\frac{y + (q-1)(1-p)}{(1-p)y+p+(q-1)(1-p)}\right);
\]
note $\Phi(p, y)$ corresponds to  $\Phi(y)$ from~\eqref{message recursion} but
here we are emphasizing the dependence on $p$. 
For fixed $p$, we have $g(y) = h(p,y)$. For $p > \ps$, there exists a unique non-trivial fixed point $y^* > 1$ such that $h(p,y^*) = y^*$ (by Theorem~\ref{thm:fixed points:appendix}). We also write $y^* = y^*(p)$ to emphasize the dependence on the parameter $p$.

The following lemma shows that $g'(y^*(p))$ attains its maximum at $p = \ps$ on $p \in [\ps, 1)$ so that we can reduce the proof of Theorem~\ref{thm: sharper bound} to this case. The proof of Lemma \ref{lemma: worst case g'} is deferred to Section \ref{appendix:stronger-decay}.

\begin{lemma}\label{lemma: worst case g'}
    Fix $d \ge 2, q > 2$ and $p \geq \ps$. Then, $g'(y^*(p))$ attains its maximum at $p = \ps$.
\end{lemma}

We are ready to prove Theorem~\ref{thm: sharper bound}.

\begin{proof}[Proof of Theorem~\ref{thm: sharper bound}]
    For the regime $p > \ps$, Lemma \ref{lemma: worst case g'} implies that $g'(y^*(p))\le g'(y^*(\ps))$. Therefore it is enough to prove the desired bound at $p=\ps$.
    At the fixed point $y^*$, we have
    \[
    g'(y^*) = d y^*{^{\frac{d-1}{d}}}\Phi'(y^*).
    \] 
    Setting $p = \ps = \frac{q}{d+q-1}$, a direct computation gives
    $
        \Phi'(y) = \frac{d q^2}{[(1-y) + d(y+q-1)]^2}
    $
    and thus it suffices to show:
    \begin{equation}
        g'(y^*) = d y^{*{\frac{d-1}{d}}}\frac{d q^2}{((1-y^*) + d(y^*+q-1))^2}  < \frac{1}{d^a}
        \label{ineq:y^*}
    \end{equation}

\begin{figure}[t]
    \centering

    \begin{subfigure}[t]{0.48\linewidth}
        \centering
        \begin{overpic}[width=\linewidth]{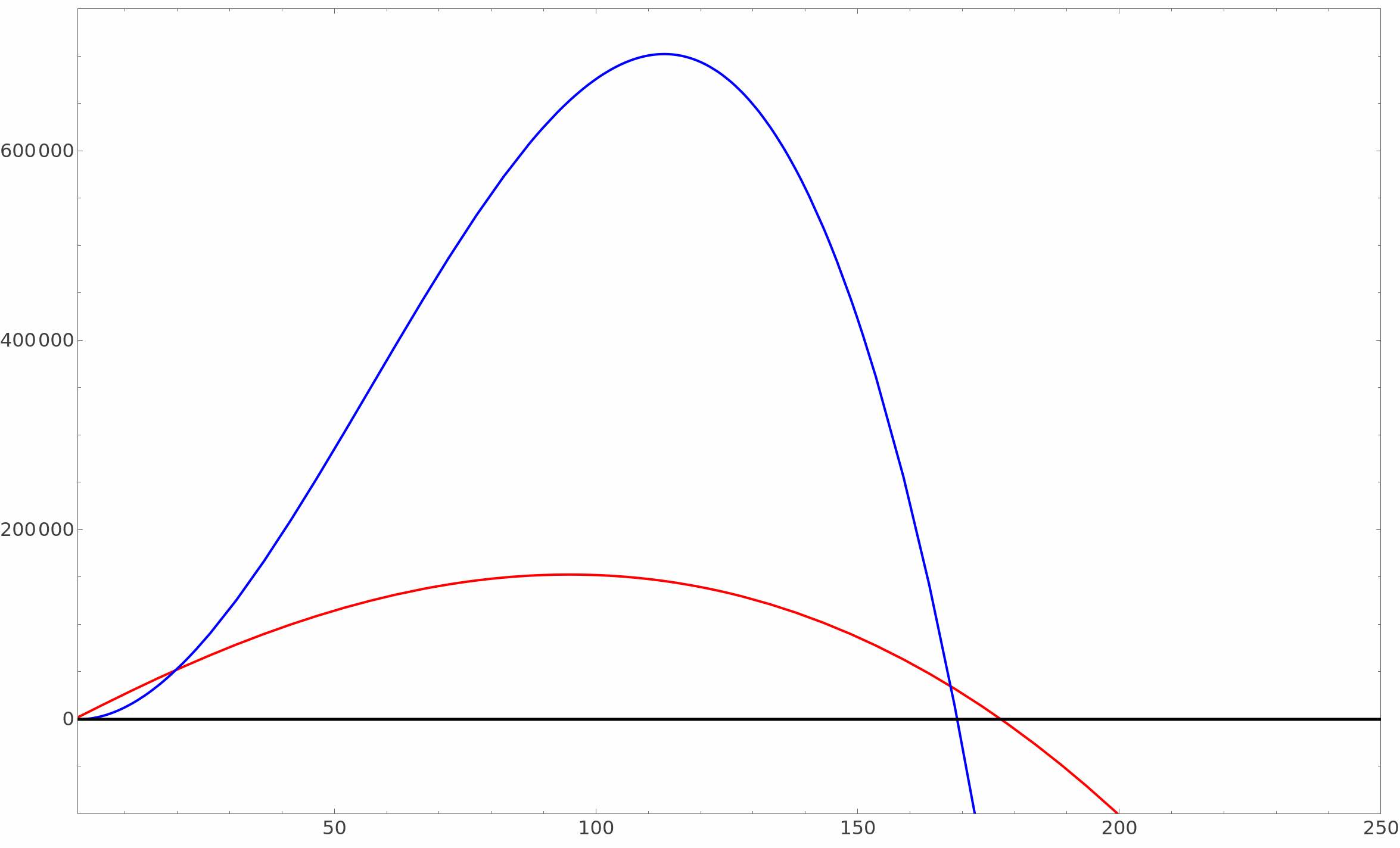}
            \put(50,20){\makebox(0,0)[l]{{\tiny{\color{red}$P(x)$}}}}
            \put(52,42){\makebox(0,0)[l]{{\tiny{\color{blue}$Q(x)$}}}}
        \end{overpic}
        \caption{$q < C_0(a) \log d$}
    \end{subfigure}
    \hfill
    \begin{subfigure}[t]{0.48\linewidth}
        \centering
        \begin{overpic}[width=\linewidth]{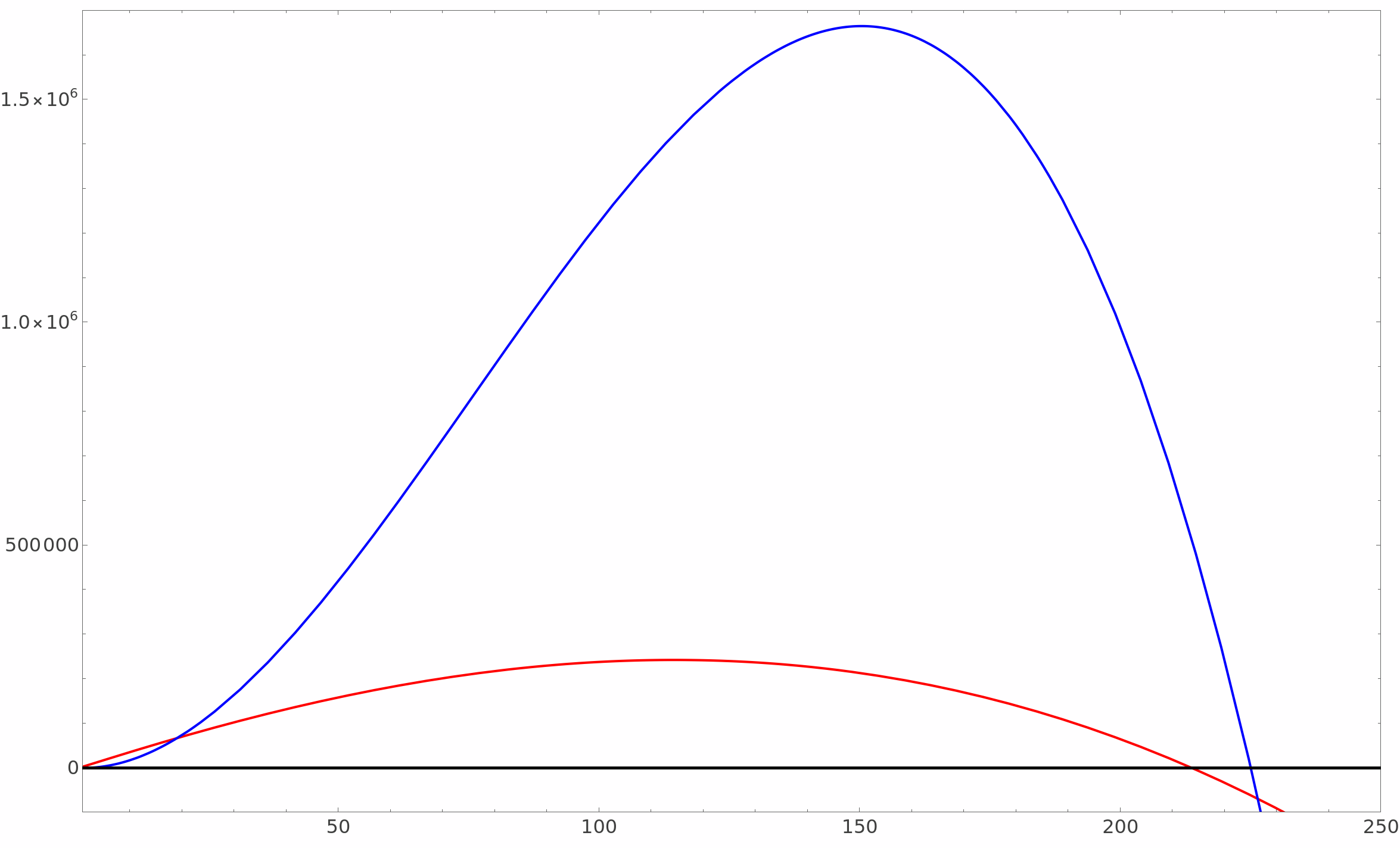}
            \put(54,15){\makebox(0,0)[l]{{\tiny{\color{red}$P(x)$}}}}
            \put(71,39){\makebox(0,0)[l]{{\tiny{\color{blue}$Q(x)$}}}}
        \end{overpic}
        \caption{$ q \geq C_0(a) \log d$}
    \end{subfigure}
    
    \caption{Plots of the function $P(x)$ and $Q(x)$ with the locations of their roots when $q < C_0(a) \log d$ vs. when $q \geq C_0(a) \log d$.}
    \label{fig:graph PQ}
\end{figure}

    We first prove the claim for $d \ge 3$. 
    We define the auxiliary functions:
    $A(y) = (1-y) + d(y+q-1)$, $B(y) = (d-1)(q-1) + (d+q-1)y$, $P(y) = d^{a+2}q^2y - y^{1/d}A(y)^2$ and $Q(y) = B(y)^d - yA(y)^d$.
    At $p = \ps$, we have
    \begin{equation}
        g(y^*) = \left(\frac{y^* + (q-1)(1-\ps)}{(1-\ps)y^* + \ps + (q-1)(1-\ps)} \right)^d = \left(\frac{B(y^*)}{A(y^*)}\right)^d = y^*.
        \label{eq:fixed-point y^*}
    \end{equation}
    Observe that $Q(y^*) = 0$ from \eqref{eq:fixed-point y^*}, and the inequality \eqref{ineq:y^*} is equivalent to $P(y^*) < 0$.
    To establish this we show that both $P$ and $Q$ have unique roots in $(1,+\infty)$, respectively, and that 
    they are both positive in the interval between $1$ and their unique root and negative after that.
    The result follows by noting that the root of $Q$ is larger than that of $P$ when $q \geq 3^{\frac{a+2}{2}} d$; see Figure~\ref{fig:graph PQ}.
    These facts are formally stated as follows.
    
    \begin{fact}\label{fact: one sol P and Q}
    For $d \geq 3$ and $q > 2$, the functions $P$ and $Q$ each have a unique root $r_P, r_Q \in (1,\infty)$, respectively, and $r_Q = y^*$. Moreover,
    $
    P(x)>0$ for  $x\in(1,r_P)$, $P(x)<0$ for $x\in(r_P,\infty)$, and
    $Q(x)>0$ for $x\in(1,r_Q)$, $Q(x)<0$ for $x\in(r_Q,\infty)$.
    \end{fact}

    \begin{fact} \label{fact: q >= 9d}
    For any integers $d \geq 3$ and any real number $a \geq 1$, 
    $r_P < r_Q$ when $q \geq 3^{\frac{a+2}{2}} d$.
    \end{fact}

    \begin{fact}\label{fact: q >= Clogd}
       For any integers $d \geq 3$ and any real number $a \geq 1$,  there exists a constant $C_0(a) > 0$ such that
    $r_P < r_Q$ when $C_0 \log d \leq q \leq 3^{\frac{a+2}{2}} d$.
    \end{fact}

    By Fact \ref{fact: one sol P and Q}, both $P$ and $Q$ have unique roots $r_P$ and $r_Q = y^*$ on $(1, \infty)$, respectively. Furthermore, $r_P < y^* = r_Q$ implies $P(y^*) < 0$. 
    Facts \ref{fact: q >= 9d} and \ref{fact: q >= Clogd} imply that for all $d \ge 3$ and $a \ge 1$, there exists $C_0(a) > 0$ such that $r_P < r_Q$ when $ q \geq C_0(a) \log d > 2.$ This proves \eqref{ineq:y^*}, when $d \ge 3$.

    We now assume that $d = 2$. In this case, $A(y)=2q+y-1, B(y)=(q-1)+(q+1)y.$ At $p=\ps$, the fixed point $y^*$ 
    \[
        \left(\frac{B(y^*)}{A(y^*)}\right)^2 = \left(\frac{(q-1)+(q+1)y^*}{2q + y^* - 1}\right)^2=y^*.
    \]
    Let $t=(y^*)^{1/2}>1$. Then we have
    $
        \frac{(q-1)+(q+1)t^2}{2q+t^2-1}=t,
    $
    which is equivalent to
    \[
        t^3-(q+1)t^2+(2q-1)t-(q-1) = (t-1)^2(t-(q-1))=0.
    \]
    Hence, $y^*=(q-1)^2$, 
    and the inequality $g'(y^*) < 2^{-a}$ in \eqref{ineq:y^*} is equivalent to
    \[
    \frac{2^{a+2}q^2(q-1)}{(2q+(q-1)^2-1)^2} = \frac{2^{a+2}q^2(q-1)}{q^4} = \frac{2^{a+2}(q-1)}{q^2} < 1,
    \]
    which holds whenever $q > 2^{a+2}.$
\end{proof}

The proofs of Facts~\ref{fact: one sol P and Q},~\ref{fact: q >= 9d}, and~\ref{fact: q >= Clogd} are deferred to Appendix~\ref{appendix:stronger-decay}, where they are established by elementary calculus arguments.

\subsection{Mixing time on treelike graphs with $(\theta,Q)$-wired boundary}
\label{sec:mixing on treelike graphs}

Our goal in this section is to establish Lemma~\ref{lem:thetaQ-wired-implies-P2}. For this, let us define a class of boundary conditions associated with the graphs in the class $\mathcal{G}(h)$. Fix an integer $Q\ge 0$, and let $G\in \mathcal{G}(h)$. If $G$ is a tree, set $\mathcal{T}=G$; otherwise, let $\mathcal{T}$ denote the BFS tree of $G$. In both cases, let $\rho$ denote the root of $\mathcal{T}$. All subtrees will be defined with respect to the tree $\mathcal T$. For a vertex subset $\Lambda \subseteq V(G)$, we use $G[\Lambda]$ to denote the subgraph of $G$ induced by $\Lambda$. In this way, when considering a subtree $T$ that contains a non-tree edge from $G$, then the edge is not included in $T$ but is included in $G[T]$.

Recall also the definition of $v^*$ when $G$ is an almost-$\Delta$-regular tree. We define $\tilde{\mathcal{H}}(h,Q)$ to be the set of boundary conditions $\xi$ on $\partial \mathcal{T}$ (or on $\{v^*\}\cup \partial \mathcal{T}$, when applicable) satisfying the following properties:
\begin{enumerate}[(a)]
    \item The total number of vertices contained in non-singleton, non-giant components of $\partial \mathcal{T}$ is at most $2Q$;
    \item If $G$ is an almost-$\Delta$-regular tree  then the vertex $v^*$ may be wired  to $\mathcal{C}_1(\xi)$.
    \item There exist fixed constants
    $\beta\in (0,1)$ and $C>0$ such that for every vertex $v\in \mathcal{T}$ with 
    $h(v):=\dist(v,\partial\mathcal{T}) \ge  (\log_d {h^{1.1}})^2$, we have $|\mathcal C_1(\xi)\cap \partial \mathcal{T}_{v,h(v)}| = \Omega(|\partial \mathcal{T}_{v,h(v)}|)$ and the following bounds hold:
    \[
    \left|\mu^1_{G[\mathcal{T}_{v,h(v)}]}(v \sim  \partial \mathcal{T}_{v,h(v)}) - \mu^{\xi}_{G[\mathcal{T}_{v,h(v)}]}(v \sim \mathcal C_1(\xi)\cap \partial \mathcal{T}_{v,h(v)})\right|
     \leq C \beta ^{h(v)/3},
\]
and 
\[
    \left|\mu^{1,\circlearrowleft}_{G[\mathcal{T}_{v,h(v)}]}(v \sim  \partial \mathcal{T}_{v,h(v)}) - \mu^{\xi, \circlearrowleft}_{G[\mathcal{T}_{v,h(v)}]}(v \sim \mathcal C_1(\xi)\cap \partial \mathcal{T}_{v,h(v)})\right|
     \leq C \beta ^{h(v)/3}.
\]
\end{enumerate}
For this broader class of boundary conditions, we prove the following mixing time bound.
\begin{theorem}\label{thm:mixing unicyclic}
    Fix $\Delta\ge 3, q>1$ and let $Q\ge 0$ be any integer.
    For $p>\ps$, the random-cluster dynamics on any $n$-vertex graph $G\in \mathcal{G}(h)$ with boundary condition from $\tilde{\mathcal{H}}(h, Q)$ has mixing time
    $n\cdot (\log n)^{O(\log \log n)}$.
\end{theorem}
Theorem~\ref{thm:mixing unicyclic} strictly generalizes Theorem~\ref{thm:tree-mixing}. 
Finally, we show that a typical $(\theta, Q)$-wired boundary condition  also falls inside $\tilde{\mathcal H}(h,Q)$.
\begin{corollary}
    \label{cor:theta-Q-wired-bc}
    Fix $\Delta \ge 3$, $q > 1$, $\theta \in (0,1)$, $Q\ge 0$ and $p > \ps$.
    Let $\mathsf{M}$ be a $(\theta, Q)$-wired distribution over boundary conditions on $\ \partial \mathcal{T}$, and let $\xi \sim \mathsf{M}$. Then $\xi \in \tilde{\mathcal{H}}(h, Q)$ with probability $1-\exp(-\Omega(h^{1.1}))$.
\end{corollary}

We can now provide the proof of Lemma~\ref{lem:thetaQ-wired-implies-P2}.

\begin{proof}[Proof of Lemma~\ref{lem:thetaQ-wired-implies-P2}]
The lemma follows immediately from Theorem~\ref{thm:mixing unicyclic} 
and Corollary~\ref{cor:theta-Q-wired-bc}, since $B(e)\in \mathcal{G}(R)$ under assumption~\textup{(G1)}, and $\exp(-\Omega(R^{1.1}))=\exp(-\Omega((\log_d n)^{1.1}))=O(n^{-10})$.
In addition, one can verify that the all-wired boundary condition on $B(e)$ belongs to $\tilde{\mathcal{H}}(h,0)$.
\end{proof}

\subsubsection{Proof of Theorem~\ref{thm:mixing unicyclic} and Corollary~\ref{cor:theta-Q-wired-bc}}
    As in the $d$-ary tree case, in the proof of Theorem~\ref{thm:mixing unicyclic} we will fix $h^*$ sufficiently large and set $r^*:=(\log_d {(h^*)^{1.1}})^2$.
Then, for each $h\in (r^*, h^*]$ we define a class of boundary conditions $\mathcal{H}(h,Q)$, which is the same as $\tilde{\mathcal{H}}(h,Q)$ except that condition~(c) is only required to hold for vertices in distance at least $r^*$ from $\partial\mathcal{T}$.
Since $r^* \ge (\log_d {h^{1.1}})^2$, we have  $\tilde{\mathcal{H}}(h,Q) \subseteq \mathcal{H}(h,Q)$.

For a fixed $Q\ge 0$, we perform a recursion on $h$ for the quantity $\Thh$ defined below, where the maximum is taken over graphs in $\mathcal{G}(h)$ and their compatible boundary conditions in $\mathcal{H}(h,Q)$:
\begin{equation}
    \label{eq: unicyclic recursion}
    \Thh := \max_{\substack{
            G\in \mathcal{G}(h)\\
            \xi\in \mathcal{H}(h,Q)
            }}
    \max\Bigl\{
        T_\mix(P^\xi_G),\ T_\mix(P_G^{\xi, \circlearrowleft})
    \Bigr\}.
\end{equation}
Since $\mathcal{H}(h,0)\subseteq \mathcal{H}(h,1)\subseteq \mathcal{H}(h,2)$, we may assume without loss of generality, when upper bounding $\Thh$, that $Q\ge 2$.

Since $G\in \mathcal{G}(h)$ is either a tree or a unicyclic graph, the number of edges in $G$ is in the same order as the number of edges (or vertices) in $\mathcal{T}_h$. Hence, we use $n_h=|\mathcal{T}_h|$ as an approximate for the size of graphs in $\mathcal{G}(h)$.
Also, recall that a decomposition $(h_0,h_1,h_2)$ of $h$ satisfies $h=h_0+h_1+h_2$ and $h_0=h_2=\delta h$, where $\delta\in(0,1/4)$ is a sufficiently small constant; we implicitly round as needed so that $h_0$, $h_1$, and $h_2$ are integers.
The following recursion will lead to Theorem~\ref{thm:mixing unicyclic}. 
\begin{lemma}
    \label{lem: mixing unicyclic recursion}
Fix $\Delta\ge 3 $, $ q>1$ and $Q\ge 2$.
For sufficiently small constant $\delta\in (0,\frac14)$, the following holds.
If $p>\ps$, then for any sufficiently large $h\in (2r^*,h^*]$ with its decomposition $(h_0,h_1,h_2)$
there exists a constant $M>0$ such that the random-cluster dynamics satisfies
\begin{align}
    \label{eq: mixing unicyclic recursion}
\Th(h) &\le
M \log^2 n_h \cdot \max\!\left\{
\Th(h_0+h_1)\cdot \frac{|\mathcal{T}_{h}|}{|\mathcal{T}_{h_0+h_1}|},
\;
d^{h_0}\cdot {h}^{4} \cdot \Th(h_1+h_2)
\right\}.
\end{align}
\end{lemma}
Lemma~\ref{lem: mixing unicyclic recursion} is the analogue of Lemma~\ref{lem: mixing recursion} for the more general classes of treelike graphs and boundary conditions $\mathcal G(h)$ and $\mathcal H(h,Q)$. 
The idea of proof is the same, but there 
are however several technical differences 
we briefly summarize:
\begin{itemize}
    \item We require the more general weak spatial mixing statement Theorem~\ref{thm wsm theta,Q wired} which applies to the graphs in $\mathcal{G}(h)$ and boundary conditions drawn from a $(\theta, Q)$-wired distribution;
    \item To analyze the root to boundary connection probabilities, we perform modifications to the graph and the boundary condition in order to address irregularities in the graphs and the presence of non-singleton  non-giant boundary components;
    \item The argument in the proof of Lemma~\ref{lem: 3rd bound kappa} must be adjusted to allow for a possible additional edge between the two of the trees in the component.
\end{itemize}

The latter two issues are addressed in Appendix~\ref{app:mixing}.
We provide the proof of Theorem~\ref{thm:mixing unicyclic} next.

\begin{proof}[Proof of Theorem~\ref{thm:mixing unicyclic}]
For small constant $h$,
we use the crude bound 
$\Thh = \exp(O(h))$. This bound can be derived 
from a similar bound for trees (previously derived via 
the canonical path method in~\cite[Lemma 17]{BG24-PTRF}) by 
adding or removing $O(1)$ external wirings (as in \cite[Lemma 24]{BG24-PTRF}). For completeness, we provide a proof of this fact in Appendix~\ref{app:mixing}.
As such, we may assume that $h^*$ is a sufficiently large and therefore any $h \in (3r^*,h^*]$ is also sufficiently large so that 
 Lemma~\ref{lem: mixing unicyclic recursion} can apply to $h$.
We can then iterate the recursion~\eqref{eq: mixing unicyclic recursion} taking the crude estimate as the base case once the height falls below $3r^*$. 
The resulting recurrence can then be handled in the same manner as in the proof of Theorem~\ref{thm:tree-mixing}.
\end{proof}

We conclude this section with proof of Corollary~\ref{cor:theta-Q-wired-bc}.

\begin{proof}[Proof of Corollary~\ref{cor:theta-Q-wired-bc}]
Fix $h$ sufficiently large. We will show that $\pr\left(\xi\in\tilde{\mathcal{H}}(h, Q)\right)\ge 1-\exp(-\Omega(h^{1.1}))$; the claimed mixing time bound then follows from Theorem~\ref{thm:mixing unicyclic}.
Condition~(a) holds by the definition of a $(\theta,Q)$-wired distribution, and condition~(b) is vacuous in the present setting. It therefore remains only to verify condition~(c), with the stated probability bound.

Let $C,\gamma>0$, and $c>0$ be the constants from Theorem~\ref{thm wsm theta,Q wired}, and let $\beta> \beta_*$.
For any vertex $v\in V(G)$  with $h(v)=\dist(v,\partial\mathcal{T})$, by its definition, 
$G[\mathcal{T}_{v, h(v)}]\in \mathcal{G}(h(v))$.
Moreover, since $\xi$ is drawn from a $(\theta, Q)$-wired distribution, the restriction of $\xi$ to $\partial \mathcal{T}_{v, h(v)}$, denoted as $\xi_{(v)}$, also follows a $(\theta, Q)$-wired distribution $M_{(v)}$.
Chernoff's inequality implies that $|\mathcal C_1(\xi_{(v)})| = \Omega(|\partial \mathcal{T}_{v,h(v)}|)$ with probability $1-e^{-\Omega(d^{h(v)})}$.
For any vertex $v$ with $h(v)\ge \gamma $, Theorem~\ref{thm wsm theta,Q wired} implies that, with probability at least $1-2\exp(-cd^{{\sqrt{h(v)}}})$ over $\xi_{(v)} \sim M_{(v)}$,
\[
\left|\mu^1_{G[\mathcal{T}_{v,h(v)}]}(v \sim  \partial \mathcal{T}_{v,h(v)}) - \mu^{\xi}_{G[\mathcal{T}_{v,h(v)}]}(v \sim \mathcal C_1(\xi)\cap \partial \mathcal{T}_{v,h(v)})\right|
     \leq C \beta ^{h(v)/2},
\]
and the same bound holds for the measures $\mu^{1,\circlearrowleft}_{G[\mathcal{T}_{v,h(v)}]}$ and $\mu^{\xi,\circlearrowleft}_{G[\mathcal{T}_{v,h(v)}]}$.
There are $d^{h- h(v)}$ such vertices in the same level.
Note that if $h(v)\ge(\log_d {h^{1.1}})^2$, then $cd^{\sqrt{h(v)}} \ge 2ch^{1.1} \ge 2h\log d $ for a sufficiently large $h$.
Taking a union bound over all vertices of height $h(v) \ge (\log_d {h^{1.1}})^2 \ge \gamma$ and both estimates gives
\begin{align*}
\pr\left(\xi\in\tilde{\mathcal{H}}(h, Q)  \right) &\ge 1 - \sum_{x=\lceil (\log_d {h^{1.1}})^2 \rceil}^h  2d^{h-x} \left[\exp(-cd^{\sqrt{x}}) + \exp(-\Omega(d^{x}))\right] \\
&\ge 1 - \sum_{x=\lceil (\log_d {h^{1.1}})^2 \rceil}^h 4\exp\left(\log d \cdot (h-x) - cd^{\sqrt{x}}\right) \\
&\ge 1 -  \sum_{x=\lceil (\log_d {h^{1.1}})^2 \rceil}^h 4\exp \left( - \frac{cd^{\sqrt{x}}}{2}\right) 
\ge  1 -  4h \exp\left( - ch^{1.1}\right),
\end{align*}
which is at least $1-\exp(-\Omega(h^{1.1}))$ for all $h$ sufficiently large.
\end{proof}

\section*{Acknowledgments}
We thank anonymous referees for their helpful comments.

\section*{AI Disclosure}
We came up with all proofs in the paper without the use of generative AI. The writing was also done entirely without the use of large language models. ChatGPT 6.0 Astra was used for a final proofreading and revisions after referee comments. 

\bibliography{bib}
\bibliographystyle{alpha}

\cleardoublepage
\addtocontents{toc}{\protect\setcounter{tocdepth}{-1}}
\appendix

\section{Tree recursion: Proof of Lemma~\ref{lem:recursion}}
\label{appendix:tree-recursion}

We prove Lemma~\ref{lem:recursion} next; the proof is a slight generalization of the argument in~\cite[Fact 3.2]{blanca2023sampling}.

\begin{proof}[Proof of Lemma~\ref{lem:recursion}]
    Fix a single-component boundary condition $\xi$ on $\partial \mathcal T^v$. For a vertex $u \in V(\mathcal T^v)$, for ease of notation, we write
   $Z_0(u):=Z^{\xi_{(u)},0}_{\mathcal T^{v,u}}$, 
   $Z_1(u):=Z^{\xi_{(u)},1}_{\mathcal T^{v,u}}$, and $f(u) :=f^{\xi}_{\mathcal T^v}(u)$.
    Let $N(u)$ be the set of children of $u$ in the rooted tree $\mathcal T^v$. 
    To obtain the recurrence on $f(u)$, we let 
    \[
    A:=\prod_{w\in N(u)} \frac{tZ_0(w)+(1-p)Z_1(w)}{q}, \quad
    B:=\prod_{w\in N(u)} \frac{tZ_0(w)+Z_1(w)}{q}.
    \] where $t = p/q + (1-p)$.
    From \cite[Fact 3.2]{blanca2023sampling}, we have     $Z_0(u) = q^2 A$ and $Z_1(u) = q(B - A)$ when $\mathcal  C_1(\xi_{(u)})\neq\varnothing$, and $Z_0(u) = q^3 A$ and $Z_1(u) = q^2(B - A)$ when $\mathcal C_1(\xi_{(u)}) = \varnothing$. Indeed, when $\mathcal C_1(\xi_{(u)})\neq\varnothing$, it is already present in every configuration contributing to $Z_0(u)$ or $Z_1(u)$ and is counted once in the random-cluster weight. When $\mathcal C_1(\xi_{(u)})=\varnothing$, no such component is present, so its corresponding factor $q$ is absent. Multiplying both restricted partition functions by $q$ restores the same normalization in the two cases.
    
    By Definition \ref{def:message recursion}, 
    \[
    f(u) = q\frac{Z_1(u)}{Z_0(u)} + 1 = q\frac{B - A}{qA} + 1 = \frac{B}{A}.
    \]
    Also, we have $Z_1(u) = \frac{f(u) - 1}{q} Z_0(u)$. Hence,
    \begin{align*}
        f(u) = \frac{B}{A}&  = \prod_{w \in N(u)} \frac{t Z_0(w) + Z_1(w) }{tZ_0(w) + (1-p)Z_1(w)} = \prod_{w \in N(u)} \frac{t Z_0(w) + \frac{1}{q}(f(w)-1)Z_0(w)}{tZ_0(w) + (1-p)\frac{1}{q}(f(w)-1)Z_0(w)} \\ &= \prod_{w \in N(u)} \frac{tq + (f(w) - 1)}{tq + (1-p)(f(w)  - 1)} = \prod_{w \in N(u)} \frac{f(w) + (q-1)(1-p)}{(1-p)f(w) + p + (q-1)(1-p)}\,,
    \end{align*}
    concluding the proof.
\end{proof}

\section{Fixed point analysis of the tree recursion}
\label{appendix:fp}

In this appendix, we prove Theorem \ref{thm:fixed points:appendix} where we provide a full characterization of the fixed points of the function $g$ in \eqref{eq: function g} for all $\Delta \ge 3$, $q \geq 1$, and $p \in(0, 1)$.

\begin{proof}[Proof of Theorem \ref{thm:fixed points:appendix}]
Note that $g(1) = 1$, which implies that $y=1$ is always a trivial fixed point of $g$. We first introduce some basic properties of $g$. Note that $g$ is twice differentiable on $[1, \infty)$ and we have
\begin{align}\label{eq:g'}
g'(y) = d p\left(\frac{y + (q-1)(1-p)}{(1-p)y + p + (q-1)(1-p)} \right)^{d-1}\frac{1 + (q-1)(1-p)}{\left((1-p)y + p + (q-1)(1-p)\right)^2}.
\end{align}
In particular, $g'(y) > 0$ for all $y \geq 1$. Thus, $g$ is strictly increasing on $[1, \infty)$. Furthermore,
\begin{align}\label{eq:g''}
\begin{split}
    g''(y) =& d p\left(\frac{y + (q-1)(1-p)}{(1-p)y + p + (q-1)(1-p)} \right)^{d-2}\frac{1 + (q-1)(1-p)}{\left((1-p)y + p + (q-1)(1-p)\right)^4}\\& \qquad \cdot [(d - 1)p(1+(q-1)(1-p)) - 2(1-p)(y+(q-1)(1-p))].
\end{split}
\end{align}
Hence, $g''(y) = 0$ at the unique point 
\begin{equation} \label{eq: second derivate is 0}
    y_0 := \frac{(d - 1)p(1+(q-1)(1-p))}{2(1-p)}-(q-1)(1-p),
\end{equation}
and thus, $g''$ changes its sign exactly once from positive to negative on $(1, \infty)$ when $y_0 > 1$, or it is negative on $(1, \infty)$ otherwise.

Define $h(y) := g(y) - y$ on $[1, \infty)$. $h$ is also twice differentiable on its domain. The fixed points of $g$ are exactly the zeros of the function $h$. This implies $h''$ changes its sign at most once on the same domain from positive to negative, $h'$ has at most one local maximum and has at most two zeros. Thus $h$ has at most three zeros in $[1, \infty)$. Equivalently, $g$ has at most three fixed points in $[1,\infty)$.

At $y = 1$, we have
\[
g'(1) = \frac{dp}{p + q(1-p)}.
\]
Thus, $g'(1) = 1$, or $h'(1) = 0$, is equivalent to $p = \frac{q}{d + q - 1} = \ps$. Furthermore, we have $h'(1) < 0$ for $p < \ps$ and $h'(1) > 0$ for $p > \ps$.
Finally, since $g$ is bounded on $[1, \infty)$, we have $h(y) \to -\infty$ as $y \to \infty$.

We now use the characterization of $\pu$ from \cite[Eq. (3.7)]{blanca2023sampling}: 
\begin{equation} \label{eq: pc}
\pu = \sup\{p : \sup_{y > 1} (g (y) - y) \leq 0\} = \sup\{p : \sup_{y > 1} h(y) \leq 0\}.
\end{equation}

Therefore, for Part (i), if $p < \pu$, suppose that $h(\hat y) = 0$, equivalently $g(\hat y ) = \hat y$, for some $\hat y > 1$. By direct computation, we obtain that $\frac{\partial g_p(y)}{\partial p} > 0$ on $y > 1$. Here, we write $g_p$ instead of $g$ to emphasize the dependency on $p$. Then, at $p + \delta$ where $p < p+\delta < \pu$, we have $g_{p+\delta}(\hat y) > g_p (\hat y) = \hat y$, which contradicts $p + \delta < \pu$. Hence, $h(y) < 0$ for every $y > 1$ and there is exactly one fixed point, $y = 1$.

For Part (ii), assume that $q > 2$. If $p \in (\pu , \ps)$, then $h'(1) < 0$, thus with $h(1) = 0$, there exists $y_1 > 1$ sufficiently close to 1 such that $h(y_1) < 0$. Furthermore, since $p > \pu$, there exists $y_2 > y_1$ such that $h(y_2) > 0$ by \eqref{eq: pc}. Also, we have $h(y) \to -\infty$ as $y \to \infty$. Hence, $h$ has at least three zeros on $[1, \infty)$ by the continuity of $h$. Indeed, $y = 1$ from $h(1) = 0$, one of the zeros is in $(y_1, y_2)$ and the other is in $(y_2, \infty)$. Recall that $h$ has at most three zeros in $[1, \infty)$ as established by the second-derivative analysis above, so $h$ has exactly three zeros in this regime. Therefore, $g$ has exactly three fixed points.

For Part (iii), if $p > \ps$, then $h'(1) > 0$. Hence, there exists $y_3 > 1$ close to 1 such that $h(y_3) > 0$. Recall that $h(y) \to -\infty$ as $y \to \infty$. Therefore, $h$ has at least one zero in $(1, \infty)$ by the continuity of $h$. Furthermore, since $h''$ either is negative or changes its sign exactly once from positive to negative on the domain, and $h' (1) > 0$, $\lim_{y\to\infty} h'(y) =\lim_{y\to\infty} g'(y) - 1 =-1$, $h'$ has exactly one zero on $(1, \infty)$. Indeed, $h'$ starts positive, ends negative, so $h'$ can cross 0 only once.

Therefore, $h$ has a single local maximum on $(1, \infty)$. Since $h(1) = 0$, $h(y_3) > 0$ for $y_3 > 1$ close to 1, and $h(y) \to -\infty$, it follows that $h$ has exactly one zero in $(1, \infty)$. This implies
\begin{equation}\label{eq: p > ps}
    h(y) > 0 \text{ for } y \in (1, y^*), \quad h(y) < 0 \text{ for } y \in (y^*, \infty)
\end{equation}
for some $y^* > 1$. Therefore, $g$ has exactly two fixed points $y = 1$ and $y^* > 1$ in $[1, \infty)$.

For Part (iv), assume that $q > 2$ and $p = \pu$. We write $h_p$ instead of $h$ to emphasize the dependence of $h$ on $p$. Since $\pu < \ps$, we have $h'(1) < 0$. Therefore, by the continuity of $h$ in $p$ and $y$, there exists $\delta > 0$ and $\eta > 0$ such that $h_p(y) < 0$ for $p \in [\pu, \pu + \eta]$ and $y \in (1, 1+\delta]$. Consider a decreasing sequence $\{p_n\} \to \pu$ with $p_n \in (\pu, \pu + \eta]$. By Part (ii), the function $h_{p_n}$ has a zero $y_n > 1$ for each $n$. Note that $y_n > 1 + \delta$ by the preceding inequality.

On the other hand, since $h_{\pu} (y) \to -\infty$ as $y \to \infty$, by the continuity of $h$, there exists $R > 1$ depending on $\eta$ such that $h_p(y) < 0$ for $p \in [\pu, \pu + \eta]$ and $y \geq R$. This implies $y_n \leq R$ for all $n$. Therefore, $y_n \in [1+\delta, R]$ and there exists $\tilde y \in [1 + \delta, R]$ and a subsequence $\{y_{n_k}\}$ such that $y_{n_k} \to \tilde y$. By continuity of $h$ in $p$ and $y$, we obtain $h_{\pu}(\tilde y) = \lim_{k \to \infty} h_{p_{n_k}}(y_{n_k}) = 0.$ Hence, $h$ has a zero at some $\tilde y > 1$.

By the monotonicity in $p$ established above and continuity as $p\uparrow\pu$, we have
\[
\sup_{y> 1} h(y) = 0.
\]
This implies $\tilde y$ is a global maximizer of $h$ on $(1, \infty)$. To see uniqueness, since $h''$ either is negative or changes its sign exactly once from positive to negative on the domain, $h'$ has at most two zeros. Therefore, $h$ has at most one local minimum and one local maximum on $(1,\infty)$. Therefore, $\tilde y$ is a local maximum. Hence, $h$ has two zeros at $y = 1$ and $y = \tilde y$. Therefore, $g$ has exactly two fixed points in this case.

Next, assume that $q > 2$ and $p = \ps$. We have $h'(1) = 0$. Recall that $g''(y) = h''(y) = 0$ at $y_0$ from \eqref{eq: second derivate is 0}.
We observe that
\begin{equation} \label{eq: cond y0}
    y_0 \geq 1 \iff p \geq \frac{2}{d+1}.
\end{equation}
Since $q > 2$, and $p = \ps = \frac{q}{d+q-1} > \frac{2}{d+1}$, we have $y_0 > 1$. This implies $h '' (1) > 0$ and $h''$ changes its sign exactly once from positive to negative on the domain $(1, \infty)$. With $h'(1) = 0, \lim_{y \to \infty}h'(y) = -1$, $h'$ has exactly one zero on $(1, \infty)$. Indeed, $h'$ is positive immediately to the right of $1$, eventually becomes negative, and crosses 0 only once on $(1, \infty)$. The remaining part is the same as Part (iii). This implies
\begin{equation}\label{eq: p = ps, q > 2}
    h(y) > 0 \text{ for } y \in (1, y^*), \quad h(y) < 0 \text{ for } y \in (y^*, \infty).
\end{equation}
Therefore, $g$ has exactly two fixed points in $[1, \infty)$.

Finally, for Part (v), assume that $q \leq 2$ and $p = \pu = \ps$. We have $h'(1) = 0$. We want to show that $h(y) < 0$ for every $y > 1$. Here, we have \[p = \ps = \frac{q}{d+q-1} \leq \frac{2}{d+1}\] when $q \leq 2$. From the condition in \eqref{eq: cond y0}, we have $y_0 \leq 1$, and hence, $g''(y) < 0$ for all $y > 1$. Since $h''(y) =g''(y)$, it follows that $h''(y) < 0$ for all $y > 1$. Thus, $h'$ is strictly decreasing on $(1, \infty)$, which implies $h'(y) < 0$ for all $y > 1$ with the fact that $h'(1) = 0$.
Therefore, we conclude that $h(y) < 0$ for all $y > 1$. Hence, $h$ has a unique zero at $y = 1$, which implies $g$ has a unique fixed point in $[1, \infty)$.
\end{proof}

To complete the proof of Lemma ~\ref{lemma:fp:main}
we further show that when $p > \ps$, the nontrivial fixed point of $g$ is attractive and $g'(y^*) \in (0, 1)$.

\begin{proof}[Proof of Lemma \ref{lemma:fp:main}]
Recall the form of $g$ from~\eqref{eq: function g}, and from~\eqref{eq:g'} that $g$ is continuously differentiable and strictly increasing on $[1, \infty)$.
 Moreover, $\lim_{y \to \infty}g(y) = (1-p)^{-d}$, therefore, $g$ is bounded on $[1, \infty)$.
By Theorem \ref{thm:fixed points:appendix}, when $p > \ps$ the map $g$ has exactly two fixed points in $[1, \infty)$, namely 1 and $y^* > 1$. Define $h(y) := g(y) - y$. From \eqref{eq: p > ps} of Theorem \ref{thm:fixed points:appendix}, we have
\[
g(y) > y \text{ for } y \in (1, y^*), \quad g(y) < y \text{ for } y \in (y^*, \infty)
\]
With the initial point $y_0 > 1$, let us denote $y_k = g^{(k)}(y_0)$. Define the recursion as
\[
y_{k+1} = g(y_k)
\]
for $k \geq 0$. If $y_0 = y^*$, then $y_k = y^*$ for all $k \geq 0$ and the claim is trivial.

Secondly, suppose now $y_0 \in (1, y^*)$. Since $h(y_0) > 0$, we have $y_1 = g(y_0) > y_0.$ Also, since $g$ is strictly increasing and $y_0 < y^*$, we have
\[
y_1 = g(y_0) < g(y^*) = y^*.
\]
Therefore, $y_0 < y_1 < y^*$. Inductively, we deduce that $y_k < y_{k + 1} < y^*$ for all $k \geq 0$. Hence, the sequence $\{y_k\}_{k\geq0}$ is increasing and bounded above by $y^*$, so it converges to some $1 < \tilde y \leq y^*$. In other words, $\lim_{k \to \infty} y_k = \tilde y$, and we have $g(\tilde y) = g(\lim_{k \to \infty} y_k) = \lim_{k \to \infty} g(y_k) = \lim_{k \to \infty} y_{k + 1} = \tilde y$. This implies $\tilde y > 1$ is the fixed point, therefore, $\tilde y = y^*$ and $y_k$ converges to $y^*$.

Similarly, suppose now $y_0 \in (y^*, \infty)$. Since $h(y_0) < 0$, we have $y_1 = g(y_0) < y_0.$ Also, since $g$ is strictly increasing and $y_0 > y^*$, we have
\[
y_1 = g(y_0) > g(y^*) = y^*.
\]
Therefore, $y_0 > y_1 > y^*$. Inductively, we deduce that $y_k > y_{k + 1} > y^*$ for all $k \geq 0$. Hence, the sequence $\{y_k\}_{k\geq0}$ is decreasing and bounded below by $y^*$, so it converges to $y^*$ with the same argument as in the case above.

It remains to show that $0<g'(y^*)<1$. Positivity follows from~\eqref{eq:g'}. For the upper bound, recall from the proof of Theorem~\ref{thm:fixed points:appendix} that when $p>\ps$, the function $h$ has a unique local maximum in $(1,\infty)$ and then decreases strictly to $-\infty$. Since $h(y)>0$ on $(1,y^*)$ and $h(y)<0$ on $(y^*,\infty)$, the zero $y^*$ lies to the right of this local maximum. Hence $h'(y^*)<0$. Since $h'(y^*)=g'(y^*)-1$, we obtain $g'(y^*)<1$. Together with $g'(y^*)>0$, this proves $0<g'(y^*)<1$.
\end{proof}

\section{Stronger spatial mixing rate for large $q$: Proof of Theorem \ref{thm: sharper bound}}\label{appendix:stronger-decay}
In this section, we provide proofs of Lemma \ref{lemma: worst case g'}, as well as of Facts \ref{fact: one sol P and Q}, \ref{fact: q >= 9d} and \ref{fact: q >= Clogd} to complete the proof of Theorem \ref{thm: sharper bound}. We start with the following fact about the monotonicity in $p$ of the fixed point.

\begin{fact} \label{fact: fixed point increasing p}
Fix $d \geq 2, q > 1$ and $p > \ps$. Then, 
$\frac{dy^*(p)}{dp} > 0.$
\end{fact}
\begin{proof}
    Let $F(p, y) := h(p, y) - y$. Note that $F$ is continuously differentiable in the domain and $F(p, y^*) = h(p, y^*) - y^* = 0$.
    Using implicit differentiation, 
    \[
    \frac{d y^*(p)}{dp}=-\frac{\partial_p F(p,y^*)}{\partial_y F(p,y^*)}.
    \]
    We show that the numerator is positive and the denominator is negative, which implies $\frac{dy^*(p)}{dp} > 0$.
    Indeed, by Lemma \ref{lemma:fp:main}, we have
    $
    \partial_y F(p,y^*(p)) = \partial_y h(p, y^*(p)) - 1 < 0.
    $
    Moreover, a direct computation gives
    $\partial_p F(p,y)=\partial_p h(p,y)>0$ for every $y>1$.
\end{proof}

We are now ready to prove Lemma \ref{lemma: worst case g'}.

\begin{proof}[Proof of Lemma \ref{lemma: worst case g'}]
    A direct computation yields that
    \begin{equation*}
    \partial_y h(p,y) = dp \frac{1+(q-1)(1-p)}{\bigl((1-p)y + p + (q-1)(1-p)\bigr)^2}    \Phi(p,y)^{d - 1}.
    \end{equation*}
    
    At the fixed point $y^*(p)$ we have $\Phi(p, y^*(p)) = (y^*)^{1/d}$, and therefore
    \begin{equation}\label{eq:gprime-at-fp}
        \partial_y h(p,y^*) = dp \frac{1+(q-1)(1-p)}{\bigl((1-p)y^* + p + (q-1)(1-p)\bigr)^2}
        (y^*)^{\frac{d - 1}{d}}.
    \end{equation}
    Furthermore, 
    \[
    \frac{\Phi(p, y^*) - 1}{y^* - 1} = \frac{p}{(1-p)y^* + p + (q-1)(1-p)}.
    \]
    Substituting $z := \Phi(p, y^*)$ and $y^* = z^d$, the denominator can be written as
    \[
    (1-p)y^* + p + (q-1)(1-p) = \frac{p(y^* - 1)}{z-1} = \frac{p(z^d - 1)}{z-1}.
    \]
    Plugging this into \eqref{eq:gprime-at-fp} gives
    \begin{equation*}
        \partial_yh(p,y^*) = d p (1+(q-1)(1-p))z^{d-1}\Big(\frac{z-1}{p(z^{d}-1)}\Big)^2 = \Big(\frac{1+(q-1)(1-p)}{p}\Big)\Big(\frac{d z^{d-1}(z-1)^2}{(z^{d}-1)^2}\Big).
    \end{equation*}
The function $\frac{1+(q-1)(1-p)}{p}$ is decreasing  in $p$ when $q > 1$. 
Moreover, by Fact \ref{fact: fixed point increasing p}, $y^*(p)$ is strictly increasing in $(\ps, 1)$, and so $z(p) = y^*(p)^{1/d}$ is also increasing in $p$. Thus, it suffices to show that $H(z)= \frac{d z^{d-1}(z-1)^2}{(z^{d}-1)^2}$ is decreasing in $z > 1$.  
    For this, let
    \[
    S(z) = \sum_{k=0}^{d-1}z^k = \frac{z^d-1}{z-1},
    \] for  $z \in (1, \infty)$,
    so that  $H(z) = d z^{d - 1}{S(z)^{-2}}$ and 
    $$
    H'(z) = dz^{d-2}S(z)^{-3}((d-1)S(z) - 2zS'(z)).$$ 
    The sign of $H'$ is then determined by that of
    \[
    (d-1)S(z) - 2zS'(z) = \sum_{k=0}^{d-1} (d-1-2k)z^k.
    \]
    Pairing the $k$-th and $(d-1-k)$-th terms of the sum (for $k < \frac{d-1}{2}$) we observe that
    \[
        (d-1-2k)z^k + (2k-(d-1))z^{d-1-k} = (d-1-2k)(z^k - z^{d-1-k}) < 0,
    \]
    since $d - 1 - 2k > 0$ and $z^k < z^{d-1-k}$ for $z > 1$. If $d$ is odd, the middle term is $0$. Hence $H'(z) < 0$ for all $z > 1$. Since $\partial_y h(p, y^*)$ is the product of two positive decreasing functions in $p$, it is decreasing for $p > \ps$. By continuity, therefore, $\partial_y h(p, y^*)$ attains its maximum over $p\in[\ps,1)$ at $p = \ps$.
\end{proof}

We conclude this section by proving Facts~\ref{fact: one sol P and Q}, \ref{fact: q >= 9d}, and \ref{fact: q >= Clogd}; their proofs rely on elementary calculus arguments.

\begin{proof}[Proof of Fact~\ref{fact: one sol P and Q}]
    Any root $x$ of the function $P$ satisfies $
        x^{1-1/d} = \frac{A(x)^2}{d^{a+2}q^2}$.
    The function    
    $x^{1-1/d}$ is increasing and concave on $[1, \infty)$ while $\frac{A(x)^2}{d^{a+2}q^2}$ is increasing and convex on $[1,\infty)$. Moreover, at $x=1$ the left-hand side is larger, while the right-hand side eventually dominates as $x\to\infty$. Therefore, these functions intersect exactly once on $(1,\infty)$, and thus $P$ has a unique root $r_P \in (1,\infty)$.     
    Since $P$ is continuous and has a unique root $r_P$ in $(1,\infty)$, and $P(1) > 0$, it follows that
    $
    P(x)>0$ when  $x\in(1,r_P)$ and $P(x)<0$ when $x\in(r_P,\infty)$.
    
    For the function $Q$, note that $Q(x) = 0$ is equivalent to
    $\left(\frac{B(x)}{A(x)}\right)^d=x$,
    which is exactly the fixed-point equation for $g$ at $p = \ps$; see \eqref{eq:fixed-point y^*}. By Theorem \ref{thm:fixed points:appendix}, $g$ has a unique fixed point in $(1,\infty)$. Therefore, $Q$ has a unique root $r_Q = y^* \in (1,\infty)$. 
    For the behavior of $Q$, we have
    $Q(x)>0$ for $x\in(1,r_Q)$ and $Q(x)<0$ for $x\in(r_Q,\infty)$ from \eqref{eq: p = ps, q > 2} in Theorem \ref{thm:fixed points:appendix}.
\end{proof}

\begin{proof}[Proof of Fact~\ref{fact: q >= 9d}]
Since $A(x) = (d-1)x + dq - d + 1 > (d-1)x$ for all $x \geq 1$, we have
\[
    P(x) < d^{a+2} q^2 x - (d-1)^2 x^{2 + 1/d}.
\]
The right-hand side vanishes at
\begin{equation}
U := \Big( \frac{d^{a+2} q^2}{(d-1)^2} \Big)^{\frac{d}{d+1}}
\label{eq:bound-rp}
\end{equation}
Hence, $P(x) < 0$ for all $x > U$ by Fact \ref{fact: one sol P and Q}, it follows that $r_P < U$.

We next show that $U < r_Q$, which implies $r_P < U < r_Q$. 
Since $Q(x)>0$ on $(1,r_Q)$ and $Q(x)<0$ on $(r_Q,\infty)$ by Fact \ref{fact: one sol P and Q}, the inequality $Q(U)>0$ implies that $U < r_Q.$
Since 
$
B(x) = (d-1)(q-1) + (d+q-1)x > (d + q - 1)x
$ for all $x \geq 1$, we have
\begin{equation}\label{eq:q-bound}
Q(x) = B(x)^d - xA(x)^d > (d + q - 1)^d x^d - x A(x)^d.
\end{equation}
Evaluating at $x = U$ the inequality becomes
\[
Q(U) > U\big((d + q - 1)^d U^{d-1} - A(U)^d \big).
\]
Therefore, it suffices to show that
$
(d + q - 1)^d U^{\,d-1} > A(U)^d$ or equivalently that
\begin{equation}
\left((d- 1) + q\right) U^{\frac{d-1}{d}} > (d-1)U + \left( d(q-1) + 1\right).
\label{eq:cond-root}
\end{equation}

To verify \eqref{eq:cond-root}, it is enough to show
\[
q\,U^{\frac{d-1}{d}} > (d-1)U \quad \text{ and } \quad
(d-1)U^{\frac{d-1}{d}} > d(q-1)+1.
\]

We first prove the second inequality. Since $d(q-1)+1 < d q$, it is enough to show that
$$(d-1)U^{\frac{d-1}{d}} > d q. $$ Using the definition of $U$, this is equivalent to
\[
\Big( \frac{d^{a+2} q^2}{(d-1)^2} \Big)^{\!\frac{d-1}{d+1}} \geq \Big( \frac{d^{a+2} q^2}{(d-1)^2} \Big)^{\!\frac{1}{2}} =  \frac{d^{\frac{a}{2} + 1}}{d-1}  q > \frac{d}{d-1}  q,
\]
which holds for all $d \ge 3$, $a \geq 1$ and $q > 1$. 
(Note that this inequality requires the $d \ge 3$ assumption.) 
We next prove the first inequality, which is equivalent to 
\[
    q > (d-1) U^{1/d}
   = (d-1)\Big( \frac{d^{a+2} q^2}{(d-1)^2} \Big)^{\!\frac{1}{d+1}}
   = (d-1)^{1-\frac{2}{d+1}} d^{\frac{a+2}{d+1}} q^{\frac{2}{d+1}},
\]
that is,
\[
q^{\frac{d-1}{d+1}} > (d-1)^{\frac{d-1}{d+1}} d^{\frac{a+2}{d+1}}
   \quad\Longleftrightarrow\quad
   q > (d-1)d^{\frac{a+2}{d-1}}.
\]
The function $d^{\frac{a+2}{d-1}}$ is decreasing in $d \ge 3$ for fixed $a \geq 1$, so
$
d^{\frac{a+2}{d-1}} \le 3^{\frac{a+2}{2}}$.
Hence when $q \ge3^{\frac{a+2}{2}}d$ we have
$
q > (d-1)d^{\frac{a+2}{d-1}}$,
and therefore both inequalities above hold. Thus \eqref{eq:cond-root} holds, and so $Q(U)>0$.
\end{proof}

\begin{proof}[Proof of Fact~\ref{fact: q >= Clogd}]
    From \eqref{eq:bound-rp}, we have
    \[
    r_P < U = \Big( \frac{d^{a+2} q^2}{(d-1)^2} \Big)^{\!\frac{d}{d+1}}
    \leq \Big(\frac{9}{4} d^a q^2 \Big)^{\!\frac{d}{d+1}}
    <\frac{9}{4} d^a q^2.
    \]
    Let $C_1 = 3^{{(a+2)}/2}$, using $q \leq C_1 d$, it follows that
    \begin{equation}\label{eq:rpbound}
        r_P < \frac{9}{4} d^a q^2 \le  \frac{9}{4}C_1^2 d ^{a+2}
    \end{equation}
    for all $d\ge3$, $a \ge 1$ and $q > 1$.

    We now find a lower bound of $r_Q$. Let $x_0 := e^{\alpha q}$ for some $\alpha > 0$ to be chosen later. First, we have
    \[
        A(x_0) = (d-1)x_0 + d(q-1) + 1 < (d-1)x_0 + dq
    \]
    Writing $A(x_0) < (d-1)x_0(1+t)$ with $t := \frac{dq}{(d-1)e^{\alpha q}}$, and by \eqref{eq:q-bound}, we obtain
    \[
    Q(x_0) > x_0^d\Big[(d+q-1)^d - x_0 (d-1)^d e^{dt}\Big].
    \]
    Thus, $Q(x_0) > 0$ follows if
    \begin{equation}\label{eq:key-ineq-clean}
    (d+q-1)^d > x_0 (d-1)^d e^{d t} \iff \log\!\left(1+\frac{q}{d-1}\right)
    >
    \frac{\alpha q}{d}
    +
    \frac{dq}{(d-1)e^{\alpha q}}.
    \end{equation}
    
    Let $u:=q/(d-1)$. Then, \eqref{eq:key-ineq-clean} becomes
    \begin{equation}\label{ineq:Log-applied}
    \log(1+u) > \alpha u\frac{d-1}{d} + d u e^{-\alpha q}.
    \end{equation}
    Since $q \leq C_1 d =  3^{\frac{a+2}{2}}d$, we have $u\le C_1\frac{d}{d-1} \leq \frac{3}{2}3^{\frac{a+2}{2}}$ for $d \geq 3, a\geq1$. For all $u \in (0, \frac{3}{2}3^{\frac{a+2}{2}}],$ there exists a constant $b = b(a) > 0$ such that
    $
        \log(1+u)\ge \frac{u}{b}.
    $
Indeed, the function $\frac{u}{\log (1+u)}$ is increasing when $u > 0$, thus the supremum on $(0, \frac{3}{2}3^{\frac{a+2}{2}}]$ is attained at $u =  \frac{3}{2}3^{\frac{a+2}{2}}$, and is 
$$b = \frac{\frac{3}{2}3^{\frac{a+2}{2}}}{\log \left(1 + \frac{3}{2}3^{\frac{a+2}{2}}\right)}.$$
Using $(d-1)/d < 1$ and \eqref{ineq:Log-applied}, it suffices to show that
$
\frac{1}{b} \geq \alpha + d e^{-\alpha q}.
$
Fix $\alpha := \frac{1}{2b}$. Then, it is equivalent to
$
de^{- q/2b} \le \frac{1}{2b}
$
which is satisfied whenever $q \ge C'(a)\log d$ for some sufficiently large constant $C'(a) > 0$. Thus $Q(x_0) > 0$, implying $ r_Q > x_0 = e^{\alpha q}.$

Finally, if $q \geq C''(a) \log d$ for $C'' > 0$ large enough, then
\[
e^{\alpha q} \ge e^{\alpha \cdot C'' \log d} = d^{\alpha C''} >\frac{9}{4} C_1^2 d^{a+2}.
\]
Therefore, by taking $C_0(a) = \max(C'(a), C''(a))$, we obtain $r_P < r_Q$ from \eqref{eq:rpbound} when $ C_0 \log d \leq q \leq C_1d$ for $d \ge 3$ and $a \ge 1$.
\end{proof}

\section{Deferred proofs of mixing time extension to treelike graphs}
\label{app:mixing}

Recall that in Section~\ref{sec:mixing on treelike graphs}, $\mathcal{G}(h)$ is defined as the class consisting of (almost-)$\Delta$-regular trees and unicyclic graphs of height $h$.
The results in this section extend to a slightly larger class of graphs defined below.

A graph $G=(V,E)$ is called \emph{$(L,h)$-treelike} if it satisfies the following properties:
\begin{enumerate}[(i)]
    \item There exists a distinguished vertex $\rho\in V$ (the \emph{root}) such that the breadth-first-search (BFS) tree $\mathcal{T}_\rho$ rooted at $\rho$ has all leaves at the same depth $h$;
    \item There exists a set $H\subseteq E$ of at most $L$ edges such that $\mathcal{T}_\rho = (V, E \setminus H)$;
    \item All vertices of $G$ have degree at most $\Delta$.
\end{enumerate}

Let $\mathcal{G}(h,L)$ denote the class consisting of  all $(L,h)$-treelike graphs. 
Note that $\mathcal{G}(h) \subseteq \mathcal{G}(h,L)$ for every $L\ge1$.
For a graph $G \in \mathcal{G}(h, L)$, let $\Lambda \subseteq V(G)$ and let $\xi$ be a boundary condition on $\Lambda$. 
We say that $\xi$ is a \emph{$Q$-approximate single-component} boundary condition if there exists a subset $\Lambda_Q \subseteq \Lambda$ with $|\Lambda_Q| \le Q$ such that $\Lambda \setminus \Lambda_Q$ consists only of leaves of $\mathcal{T}_\rho$, and the restriction of $\xi$ to $\Lambda \setminus \Lambda_Q$ is a single-component boundary condition.
In particular, the boundary conditions in $\tilde{\mathcal{H}}(h,Q)$,  which are defined for graphs in $\mathcal{G}(h)$ in Section~\ref{sec:mixing on treelike graphs}, are $(Q+1)$-approximate single-component boundary conditions.

In this section we establish a bound on the mixing time of the random-cluster dynamics on graphs in $\mathcal{G}(h,L)$ under $Q$-approximate single-component boundary conditions. 

\begin{lemma}
\label{lem:treelike ball gap}
Let $Q\ge0$ and $L\ge1$ be integers.
The random-cluster dynamics on any graph $G\in\mathcal{G}(h,L)$ with $n$ vertices and any $Q$-approximate single-component boundary condition $\xi$ satisfies
\[ 
\max\Bigl\{
        T_\mix(P^\xi_G),\ T_\mix(P_G^{\xi, \circlearrowleft}) \Bigr\} = \exp(O(h)).
\]
\end{lemma}

Lemma~\ref{lem:treelike ball gap} is a slight generalization of Lemma~24 of \cite{BG24-PTRF}. While the proof is essentially the same, we reproduce it for completeness.

\begin{proof}[Proof of Lemma~\ref{lem:treelike ball gap}]
Let $G \in \mathcal{G}(h,L)$.
     By definition, there exists a set $H$ of at most $L$ edges in $E$ such that the BFS tree rooted at vertex $\rho$ is $(V, E \setminus H)$. We modify the boundary condition $\xi$  by first unwiring all existing connections in $\xi$ involving the $Q$ special vertices, and then wiring these $Q$ vertices to the only non-trivial component of $\xi$. If there is no non-trivial component, we wire them to an arbitrary trivial component. With this modification the boundary condition becomes single-component. In addition, for each $\{u,v\} \in H$, we wire $u$ and $v$ into the same component.      
We denote the resulting boundary condition by $\xi_1$.
 Since $\xi_1$ modifies the boundary condition at only $2L+Q$ locations, by \cite[Lemma 2.2]{BGV20}, we have
\begin{equation}
    \label{eq: GAP comp}
    \gap(P_G^\xi) \;\ge\; \gap(P_G^{\xi_1})\,/\,q^{5(2L+Q)};
\end{equation}
here $\gap(\cdot)$ denotes the spectral gap of the corresponding Markov chain.
Observe that
$
    \mu_G^{\xi_1} = \mu_{\mathcal{T}}^{\xi_1} \otimes \prod_{e \in H} \ber(e),
$
where $\{\ber(e)\}_{e}$ denotes i.i.d. Bernoulli random variables with parameter $p$. Then, \cite[Lemma 2.2.11]{Saloff-Coste1997} implies that
\begin{equation}
    \label{eq: gap decomp}
    \gap(P_G^{\xi_1})
    =
    \Omega\Big( \min \Big\{ \gap(P_{\mathcal{T}}^{\xi_1})\cdot  \frac{|E\setminus H|}{|E|}, \frac{1}{|H|}\cdot \frac{|H|}{|E|} \Big\} \Big).
\end{equation}
Next, to obtain a boundary condition supported only on the leaves $\partial \mathcal{T}$, we further modify ${\xi_1}$ by freeing the two vertices in all pairs $\{u,v\} \in H$, and denote the resulting boundary condition by $\xi_{2}$.
Again, this modification affects at most $2L$ locations, and hence by~\cite[Corollary 23]{BG24-PTRF}, 
\begin{equation}
    \label{eq: GAP comp 2}
    \gap(P_{\mathcal{T}}^{\xi_1}) \;\ge\; \gap(P_{\mathcal{T}}^{\xi_{2}})\,/\,q^{10L} .
\end{equation}
Observe that $\xi_2$ is a single-component boundary condition on the level of leaves.
Also, $\mathcal{T}$ is a tree of height $h$ and of size $O(d^h)$.
Combining~\eqref{eq: GAP comp}, \eqref{eq: gap decomp}, and \eqref{eq: GAP comp 2}, we obtain that
$\gap(P_G^\xi) 
     =
     \Omega\big( \min\big\{ \gap(P_\mathcal{T}^{\xi_{2}}), {|E|^{-1}} \big\} \big)$.
Finally, \cite[Lemma 17]{BG24-PTRF} yields that $\gap(P_\mathcal{T}^{\xi_{2}}) = \exp(-O(h))$. 
A similar bound can be deduced for $\gap(P_G^{\xi,\circlearrowleft})$ as it requires only one additional wiring modification. The mixing time bound follows from the standard relationship between gap and mixing time. 
\end{proof}

\subsection{Deferred proof of Lemma~\ref{lem: mixing unicyclic recursion}}

Recall that $h^*$ is a sufficiently large integer, and set
$
r^* := (\log_d {(h^*)^{1.1}})^2.
$
We also recall that $\Thh$ is defined in \eqref{eq: unicyclic recursion}. Lemma~\ref{lem: mixing unicyclic recursion} gives a recursive bound on $\Thh$ in terms of the height $h$, for $h \in (2r^*,h^*]$. For the reader’s convenience, we restate the lemma below.

\begin{lemma}[Restatement of Lemma~\ref{lem: mixing unicyclic recursion}]
Fix $\Delta \ge 3$ and $q > 1$.
For a sufficiently small constant $0<\delta < \frac14$, the following holds.
If $p>\ps$, then for any sufficiently large $h\in (2r^*,h^*]$ with its decomposition $(h_0,h_1,h_2)$ satisfying
\begin{equation}
    \label{eq: h0 h1 h2 decomposition condition}
h_0+h_1+h_2=h,\qquad
h_0=h_2= \delta h,
\end{equation}
there exists a constant $M>0$ such that the random-cluster dynamics satisfies 
\begin{align*}
\Th(h) &\le
M \log^2 n_h \cdot \max\!\left\{
\Th(h_0+h_1)\cdot \frac{|\mathcal{T}_{h}|}{|\mathcal{T}_{h_0+h_1}|},
\;
d^{h_0}\cdot {h}^{4} \cdot \Th(h_1+h_2)
\right\}.
\end{align*}
where $n_h$ denotes the number of edges in the $\Delta$-regular tree $\mathcal{T}_h$ of height $h$.
\end{lemma}

For any graph $G \in \mathcal{G}(h)$, let $\mathcal{T}$ denote the underlying tree representation of $G$: if $G$ is a tree, then $\mathcal{T}=G$; otherwise, $\mathcal{T}$ is the breadth-first-search tree appearing in the definition of $\mathcal{G}(h)$. Let $\rho$ be the root of $\mathcal{T}$.
Suppose that $h$ admits a decomposition $(h_0,h_1,h_2)$ satisfying \eqref{eq: h0 h1 h2 decomposition condition}. 
We define $B_0\subset G$ to be the subgraph induced by $\mathcal{T}_{\rho,h_0+h_1}$, and $B_1\subset G$ to be the subgraph induced by $\mathcal{T}\setminus \mathcal{T}_{\rho,h_0-1}$.
For each $i\in \{0,1\}$, letting $j=1-i$, we write $\Omega_{B_i\setminus B_j}$ for the set of all boundary conditions induced by configurations on $B_i\setminus B_j$, and denote by $|B_i|$ the number of edges in $B_i$.

It is immediate that $B_0=G[\mathcal{T}_{\rho,h_0+h_1}] \in \mathcal{G}(h_0+h_1)$.
With a fixed configuration $\eta \in \Omega_{B_0 \setminus B_1}$, the block $B_1$  decomposes into a collection $\mathcal{C}(\eta)$ of components as follows. The graph $\mathcal{T}\setminus \mathcal{T}_{\rho,h_0-1}$ is a disjoint union of trees, or a forest. Whenever the roots of several such trees are connected under the configuration $\eta$, we identify the union of them as a single component in $\mathcal{C}(\eta)$. In addition, if $B_1$ is not a forest and the unique edge in $E(B_1)\setminus E(\mathcal{T})$ joins a pair of vertices $(u,w)$ lying in two distinct trees in $\mathcal{T}\setminus \mathcal{T}_{\rho,h_0-1}$, then we merge the two corresponding components, if they are distinct, into a single distinguished component, and we call this component that contains $(u,w)$ the \emph{twin component}.

Note that every component $\kappa \in \mathcal{C}(\eta)$ other than the twin component is formed by $d_\kappa\ge 1$ treelike graphs in $\mathcal{G}(h_3)$ of height $h_3 := h_1+h_2$, with their roots wired together through a configuration in $\Omega_{B_0\setminus B_1}$ and with leaves wired according to some boundary condition $\xi_{(r_j)} \in \mathcal{H}(h_3,Q)$. Here, $d_\kappa$ denotes the number of treelike graphs in $\kappa$.
For the twin component, the same applies after deleting the unique edge $e_\kappa$ that joins two distinct trees, and we define $d_\kappa$ accordingly as the number of treelike graphs in the resulting graph.

The proof of Lemma~\ref{lem: mixing unicyclic recursion} relies on getting all the following pieces, which are analogs of Lemma~\ref{lem: block recursion} and Lemma~\ref{lem: 3rd bound kappa}.
\begin{lemma}
    \label{lem: block recursion unicyclic}
    Let $\Delta\ge 3, q>1,$ $p>\ps$ and let $Q\ge 0$.
    There exists a constant $M_1>0$ such that, for any
    $h\in (2r^*, h^*]$ that is sufficiently large and admits a decomposition
    $(h_0,h_1,h_2)$ satisfying \eqref{eq: h0 h1 h2 decomposition condition},
    any graph $G\in \mathcal{G}(h)$, and any boundary condition
    $\xi\in \mathcal{H}(h,Q)$, the mixing time of the random-cluster dynamics
    $P_G^\xi$ is  at most
    \begin{equation}
        \label{eq: block recursion unicyclic}
        M_1 \log^2 n_h
        \max\left\{
            \max_{\zeta \in \mathcal{H}(h_0+h_1,2)}
            \max\bigl\{
                T_\mix(P^{\zeta}_{B_0}),
                T_\mix(P^{\zeta,\circlearrowleft}_{B_0})
            \bigr\}
            \frac{n_h}{|B_0|},
            \;
            \max_{\substack{\eta\in \Omega_{B_0\setminus B_1}}}
            T_\mix(P_{B_1}^{\xi\sqcup  \eta})
            \frac{n_h}{|B_1|}
        \right\}.
    \end{equation}
The same upper bound applies to the mixing time of  $P_G^{\xi,\circlearrowleft}$.
\end{lemma}

In what follows, we use 
$P_{\kappa}^{\xi}$ to denote the random-cluster dynamics on the component $\kappa$ under the combined boundary condition 
$\xi_{(r_1)} \sqcup  \dots \sqcup  \xi_{(r_{d_\kappa})}\subseteq \xi$. 
\begin{lemma}
\label{lem: 3rd bound kappa unicyclic}
Let $p>\ps$ and let $Q\ge 0$ be an integer.
For any $\eta\in \Omega_{B_0\setminus B_1}$, for every component $\kappa \in \mathcal{C}(\eta)$ and every boundary condition $\xi \in \mathcal{H}(h, Q)$,  
\[
\max\{T_{\mix}(P_{\kappa}^{\xi}), T_{\mix}(P_{\kappa}^{\xi, \circlearrowleft})\}
    \;\le\;
    O\!\left(d_\kappa \log 2d_\kappa\right) \cdot O\!\left(h_3^{2}\right)
    \cdot \Th({h_3}).
\]
\end{lemma}

We first complete the proof of Lemma~\ref{lem: mixing unicyclic recursion}, and then return to establish the auxiliary lemmas.

\begin{proof}[Proof of Lemma~\ref{lem: mixing unicyclic recursion}]
The proof is obtained by modifying the proof of Lemma~\ref{lem: mixing recursion}. 
The major changes are that we invoke Lemma~\ref{lem: block recursion unicyclic} and Lemma~\ref{lem: 3rd bound kappa unicyclic} in place of Lemma~\ref{lem: block recursion} and Lemma~\ref{lem: 3rd bound kappa}. 
Once Lemma~\ref{lem: block recursion unicyclic} is applied, to bound the term 
\[
  \max_{\zeta \in \mathcal{H}(h_0+h_1,2)}
  \max\bigl\{
      T_\mix(P^{\zeta}_{B_0}),
      T_\mix(P^{\zeta,\circlearrowleft}_{B_0})
  \bigr\},
\] in \eqref{eq: block recursion unicyclic}, 
we use the fact that $B_0\in \mathcal{G}(h_0+h_1)$ together with the assumption $Q\ge 2$, and hence the above quantity is at most $\Th(h_0+h_1)$.
To bound the mixing time of $T_\mix(P_{B_1}^{\xi\sqcup  \eta})$, we make a further modification to the argument for the independence among components in $\mathcal{C}(\eta)$. 
We analyze a pair of censoring dynamics that start from all-wired and all-free configurations. 
Let $E_Q$ be the set of edges incident to vertices in non-singleton, non-giant components of $\partial \mathcal{T}$. 
By assumption, there are at most $2Q$ such edges.
We first run a burn-in phase so that these edges are closed. Under this event we then run the dynamics that censors updates to $E_Q$ for $d^{h_0}\cdot h^4 \cdot \tau_{mix}(h_1+h_2)$ steps, and at the end of these steps, $B_1\setminus E_Q$ would be coupled with at least a constant probability.  Finally, we run extra $O(|B_1|)$ steps to couple the edges in $E_Q$ while censoring all updates except those to $E_Q$.  In both the beginning the ending phase, the probability of success is at least a constant as there are at most $2Q$ many of these edges. 
\end{proof}

\begin{proof}
    [Proof of Lemma~\ref{lem: block recursion unicyclic}]
    As in the proof of Lemma~\ref{lem: block recursion}, given
$h\in (2r^*, h^*]$, $G\in \mathcal{G}(h)$, and $\xi\in \mathcal{H}(h,Q)$,
we analyze the mixing time of the systematic scan block dynamics $P_{\textsc{bd}}^{\xi}$ on $G$
which alternately updates $B_1$ and $B_0$.
We then use Lemma~\ref{lem: simulate systematic dynamics} to derive upper bounds on
$T_{\mix}(P_G^\xi)$ and $T_{\mix}(P_G^{\xi, \circlearrowleft})$.

To analyze the systematic scan block dynamics, we follow the proof of
Lemma~\ref{lem: coupling time 2 block dynamics} by coupling two copies of the dynamics
started from the all-wired and the all-free configurations.
In particular, we show that whenever the coupled dynamics update $B_1$ at time $t$, the resulting configuration is such that
the boundary conditions on $B_0$ induced by the configurations on $G\setminus B_0$
at times $t$ (and $t+1$) are  $(\theta,2)$-wired 
for some $\theta>0$.
To see this, note that since $G\in \mathcal{G}(h)$ and $\xi \in \mathcal{H}(h,Q)$, after a heat-bath update of $B_1$, by a stochastic domination argument, all but at most two vertices
$v\in \partial \mathcal{T}_{\rho,h_0+h_1}$ are connected to $\mathcal{C}_1(\xi)$ independently with probability at least    
    \[
    \min \left \{\mu_{G[\mathcal{T}_{v,h_2}]}^{\xi_{(v)}}(v\sim \mathcal{C}_1(\xi_{(v)})), \mu_{G[\mathcal{T}_{v,h_2}]}^{\xi_{(v)},\circlearrowleft}(v\sim \mathcal{C}_1(\xi_{(v)})) \right\}
    = \mu_{G[\mathcal{T}_{v,h_2}]}^{\xi_{(v)}}(v\sim \mathcal{C}_1(\xi_{(v)}))
    .\]
For such a vertex $v$, since $h_2 \ge r^*$, $G[\mathcal{T}_{v,h_2}] \in \mathcal{G}(h_2)$ and
$\xi_{(v)}\in \mathcal{H}(h_2,Q)$, we have
    \[
   \mu_{G[\mathcal{T}_{v,h_2}]}^{1}(v\sim \partial \mathcal T_{v,h_2}) - \mu_{G[\mathcal{T}_{v,h_2}]}^{\xi_{(v)}}(v\sim \mathcal{C}_1(\xi_{(v)}))
    \le C\, \beta^{\dist(v,\mathcal{C}_1(\xi))/3},
    \]
    where $\beta\in (0,1)$ and $C>0$ are constants specified by condition (c) in $\mathcal{H}(h,Q)$.
Hence, each such $v\in \partial \mathcal{T}_{\rho,h_0+h_1}$ is connected to
$\mathcal{C}_1(\xi)$ in $G[\mathcal{T}_{v,h_2}]$ with probability at least
\[
 \mu_{G[\mathcal{T}_{v,h_2}]}^{1}(v\sim \partial \mathcal T_{v,h_2})-  \beta^{h_2/3} =: \theta_v.
\]
Moreover, from~\cite[Lemma 2.2]{BGV20} we obtain
\[
\mu_{G[\mathcal{T}_{v,h_2}]}^{1}(v\sim \partial \mathcal T_{v,h_2}) = \Omega\left(\mu_{\mathcal{T}_{v,h_2}}^{1}(v\sim \partial\mathcal{T}_{v,h_2}) \right).
\]
Since $p>\ps$, the quantity
$\mu_{\mathcal{T}_{v,h_2}}^{1}(v\sim \partial\mathcal{T}_{v,h_2})$,
and hence also
$\mu_{\mathcal{T}_{v,h_2}}^{1,\circlearrowleft}(v\sim \partial\mathcal{T}_{v,h_2})$,
is bounded away from $0$ for all $h_2$ sufficiently large (see the argument below \eqref{eq: use WSM for theta-wired}).
Therefore, for $h_2$ sufficiently large, there exists $\theta>0$ that uniformly lower bounds all $\theta_v$; in particular, $\theta$ is $G$-independent and $h_2$-independent.

Since $B_0\in \mathcal{G}(h_0+h_1)$,
    by Corollary~\ref{cor:theta-Q-wired-bc}, if $\zeta$ is drawn from a $(\theta,2)$-wired distribution,
then $\zeta\in \mathcal{H}(h_0+h_1,2)$ with probability
$1 - \exp(-\Omega((h_0+h_1)^{1.1}))$, which is at least
$1 - \exp(-h^{1+\Omega(1)})$ provided that $\delta$ is sufficiently small.
In addition, we use Theorem~\ref{thm wsm theta,Q wired} to show that for the same $\zeta$,
for every edge $e\in E(G[\mathcal{T}_{\rho, h_0}])$, 
with probability at least $1-\exp(-\Omega(d^{\sqrt{h_1}}))$,
\[
 \mu^1_{B_0}(e\in A) - \mu^\zeta_{B_0}( e\in A) \leq O((\beta_*+o(1))^{\dist(e, \partial B_0)/2}) = O(\beta_* ^{h_1(1-o(1))/2}).
\]
We union bound this event over all those $O(d^{h_0})$ edges, and the event of $\{\zeta\in \mathcal{H}(h_0+h_1,2)\}$. 
Once this event holds, the coupling argument proceeds exactly as in the proof of Lemma~\ref{lem: coupling time 2 block dynamics}, and hence we omit the details.
\end{proof}

\begin{proof}
[Proof of Lemma~\ref{lem: 3rd bound kappa unicyclic}]
We consider first the case where $\kappa$ is a twin component.
Let $e_\kappa=\{u,v\}\in E(\kappa)$ be the unique edge such that $(V(\kappa), E(\kappa)\setminus\{e_\kappa\})$ is a forest $F$. Let $\mathcal{T}_{(u)}$ and $\mathcal{T}_{(v)}$ denote the two trees in $F$ containing $u$ and $v$, respectively, and let $r_{(u)}$ and $r_{(v)}$ denote their roots.
There are two types of twin components depending on whether   $r_{(u)}$ and $r_{(v)}$ are externally wired: 
\begin{enumerate}[(i)]
    \item $(V(\kappa), E(\kappa)\setminus\{e_\kappa\})$ is still one component of $\mathcal{C}(\eta)$;  or
    \item $(V(\kappa), E(\kappa)\setminus\{e_\kappa\})$ splits into two components of $\mathcal{C}(\eta)$.
\end{enumerate}
We consider a twin component of type~(i) first.
As in the proof of Lemma~\ref{lem: 3rd bound kappa}, we analyze a pair of censored random-cluster dynamics starting from the all-wired and all-free configurations.
We run an uncensored burn-in period after which the event $\mathcal{W}_1$ that 
there exists a path from $r_{(u)}$ to $\partial \mathcal{T}_{(u)} \cap \mathcal{C}_1(\xi)$
and that $e_\kappa$ is closed occurs with constant probability.
Under this event, $\kappa \setminus \mathcal{T}_{(u)}$ has a product structure, so it can be coupled in the second phase of the censoring scheme, and finally the two dynamics couple on $\mathcal{T}_{(u)}$ in the third phase. 
The argument is exactly as in the proof of Lemma~\ref{lem: 3rd bound kappa} under this new event. 

Next we consider a twin component of type~(ii).
The roots of these two components may be subject to four different boundary conditions on the two roots. In this case we require the following slightly different censoring scheme:
\begin{enumerate}
    \item There is an uncensored burn-in period, after which we consider the event $\mathcal{\widehat W}_1$ that the edge $e_{\kappa}$ is closed and that there exist paths from $r_{(u)}$ to $\partial \mathcal{T}_{(u)} \cap \mathcal{C}_1(\xi)$ and from $r_{(v)}$ to $\partial \mathcal{T}_{(v)} \cap \mathcal{C}_1(\xi)$. This event has constant probability of occurring.
    \item In the second phase, we censor updates to $E(\mathcal{T}_{(u)})\cup \{e_{\kappa}\} \cup E(\mathcal{T}_{(v)})$. Under $\mathcal{\widehat W}_1$ there is a product structure in the uncensored portion of $\kappa$. The dynamics is $P^{\xi,\circlearrowleft}$ or $P^\xi$ on each tree.
    \item In the third period, we censor all updates except those in $E(\mathcal{T}_{(u)})$. Under $\mathcal{\widehat W}_1$ since $e_\kappa$ is not present, $\mathcal{T}_{(u)}$ is a tree in $\mathcal G(h_3)$.
    \item Finally, in the fourth period, we censor all updates except those to $E(\mathcal{T}_{(v)})\cup \{e_{\kappa}\}$. 
\end{enumerate}
The analysis of this censoring scheme is analogous to the one in Lemma~\ref{lem: 3rd bound kappa}.
For components other than the twin component, an analogous argument holds.
\end{proof}

\section{Deferred proofs of simulating block dynamics}
\label{sec: simulating BD}
Since updating a given block may alter the boundary conditions of neighboring blocks, and hence affect their mixing behavior, in our proof of Lemma~\ref{lem: simulate systematic dynamics} we will apply a censoring scheme to regulate the update sequence. This allows us to ensure that each block is updated under “good’’ boundary conditions. To formalize this step, we invoke the following censoring inequality of \cite{PW}, in the standardized form presented in \cite{MS13, FillKahn13, BCV18}.
\begin{lemma}
    \label{lem: censoring}
    Let $p\in(0,1)$ and $q > 1$.
    Let $P$ be the random-cluster dynamics on an arbitrary graph $G=(V,E)$. Given a sequence of integers $0=t_0 < t_1 <t_2 <\dots$ and a sequence of edge sets $A_1, A_2, \dots$, let $\hat{P}$ be a censored version of $P$ that updates in the time interval $(t_{i-1}, t_{i}]$ only in $A_i \subseteq E$. 
    Then 
    \begin{equation*}
        T_{\mix}(P, \varepsilon) \le T_{\mix}\big(\hat{P}, \frac{\varepsilon}{2|E|}\big).
    \end{equation*}
    Moreover, this inequality holds for random sequences of $\{A_i\}$, provided the choice of $A_i$ is independent of the choices that generate the random-cluster dynamics.
\end{lemma}

\begin{proof}[Proof of Lemma~\ref{lem: simulate systematic dynamics}]
Fix $\varepsilon\in(0,1/4]$ and let $\varepsilon'\in(0,1/4]$ and $N>0$ satisfy \eqref{eq: 3 in 1 condition}.
Write $T^*:=T^*(\varepsilon',N)$ and $\delta:=\delta(N)=\exp(-2N/9)$.

Suppose $\{X_m\}_{m\ge0}$ is the systematic scan block dynamics with the boundary condition $\xi$.
Let $\{\bar Y_t\}_{t\ge0}$ be the censored single-edge dynamics (with the same boundary condition $\xi$)
defined as follows: for every integer $m\ge1$, during the time interval
$((m-1)T^*,\, mT^*]$ we censor all updates except those on edges in $B_{m \bmod k}$.
Let $\bar P^\xi$ denote its transition kernel.
For $\omega\in \Omega:=\{0,1\}^{E}$, write $\pr_\omega[\cdot]$ for probabilities for the indicated chain started from $\omega$.
For $m\ge0$ and $\omega\in\Omega$ define
\[
\Delta_m(\omega)
:=
\big\|
\pr_\omega(\bar Y_{mT^*}=\cdot)
-
\pr_\omega(X_m=\cdot)
\big\|_{\textsc{tv}}\,,
\]
where $\mathbb P_\omega$ denotes the law of the random-cluster dynamics initialized from $\omega$. 

We claim that for every $\omega\in\Omega$ and every $m\ge0$,
\begin{equation}\label{eq:Delta-bound}
\Delta_m(\omega) \le m(\varepsilon' + \delta + \rho).
\end{equation}
Clearly $\Delta_0(\omega)=0$. Fix $\omega$ and assume the bound holds at time $m$.
Let $i=(m+1)\bmod k$.
Conditioning on $X_m$ and using the Markov property and the triangle inequality gives
\begin{align}\label{eq:Delta-rec}
    \Delta_{m+1}(\omega)
    &\le
    \Delta_m(\omega)
    +
    \sum_{\sigma\in\Omega}
    \pr_\omega(X_m=\sigma) \cdot
        \big\|
        \pr_{\sigma}(\bar Y_{(m+1)T^*}( B_i)=\cdot)
        -
        \pr_{\sigma}(X_{m+1}( B_i) =\cdot)
        \big\|_{\textsc{tv}}.
\end{align}
Fix $\sigma\in\Omega$ and write $\eta=\sigma(E\setminus B_i)$.
The systematic dynamics draws a sample from  $\pi^{\xi\sqcup \eta}_{B_i}$ on $B_i$.
Let $U$ be the number of uncensored updates hitting $B_i$ during one epoch of length $T^*$.
Then
\[
U\sim\mathrm{Bin}(T^*, |B_i|/|E|).
\]
By the definition of $T^*$, on $\{\eta\in\Gamma_i\}$,
\[
\E [U] = T^*\frac{|B_i|}{|E|}
\ge  \max\{3 T_{\mix}(P^{\xi\sqcup \eta}_{B_i},\varepsilon'), N\}.
\]
Therefore, by a Chernoff bound,
\[
    \pr(
    U < T_{\mix}(P^{\xi\sqcup \eta}_{B_i},\varepsilon')
    )
    \le
    \pr(U \le \E[U]/3)
    \le
    \exp(-2 \E[U]/9)
    \le
    \exp(-2N/9)
    =
    \delta.
\]
Hence, on $\{\eta\in\Gamma_i\}$,
\[
\|
\pr_{\sigma}(\bar Y_{T^*} (B_i)\in\cdot)
-
\pi^{\xi\sqcup \eta}_{B_i}
\|_{\textsc{tv}}
\le
\varepsilon' + \delta,
\]
and therefore
\[
\|
\pr_{\sigma}(\bar Y_{T^*}( B_i)\in\cdot)
-
\pr_{\sigma}(X_{m+1}(B_i)\in\cdot)
\|_{\textsc{tv}}
\le
\varepsilon' + \delta.
\]
On $\{\eta\notin\Gamma_i\}$ we bound the total variation distance by $1$.
Since
$
\pr_\omega\big(X_m(E\setminus B_i)\notin\Gamma_i\big)
\le \rho
$,
plugging into \eqref{eq:Delta-rec} gives
\[
\Delta_{m+1}(\omega)
\le
\Delta_m(\omega) + (\varepsilon'+\delta)+\rho,
\]
which proves \eqref{eq:Delta-bound} by induction.
Let 
$
m_0
:=
2 T_{\mix}(\widehat P_\mathcal{B}^\xi,\varepsilon')$.
Then for every $\omega\in\Omega$,
\[
\|
\pr_\omega(X_{m_0}\in\cdot)
-
\pi^\xi
\|_{\textsc{tv}}
\le
\varepsilon',
\]
and therefore
\[
\|
\pr_\omega(\bar Y_{m_0T^*}\in\cdot)
-
\pi^\xi
\|_{\textsc{tv}}
\le
\varepsilon'
+
m_0(\varepsilon'+\delta+\rho).
\]
Taking the maximum over $\omega\in\Omega$ and using \eqref{eq: 3 in 1 condition} yields
$
T_{\mix}(\bar P^\xi,\frac{\varepsilon}{2|E|})
\le
m_0T^*$. 
By Lemma~\ref{lem: censoring}, we conclude the proof. 
\end{proof}
\end{document}